\documentclass[11pt]{article} 
\usepackage[utf8]{inputenc} 
\usepackage[margin=1in]{geometry} 
\usepackage{graphicx} 

\usepackage{booktabs} 
\usepackage{array} 
\usepackage{paralist} 
\usepackage{verbatim} 
\usepackage{mathrsfs}
\usepackage{amssymb}
\usepackage{amsthm}
\usepackage{amsmath,amsfonts,amssymb}
\usepackage{esint}
\usepackage{graphics}
\usepackage{enumerate}
\usepackage{mathtools}
\usepackage{xfrac}
\usepackage{nicefrac}
\usepackage{subcaption}
\usepackage{stmaryrd}
\usepackage[normalem]{ulem}
\usepackage{cancel}
\usepackage{enumitem}
\usepackage{fancyvrb}

\allowdisplaybreaks[4]

\usepackage{mathabx}

\usepackage[usenames,dvipsnames]{xcolor}
\usepackage[colorlinks=true, pdfstartview=FitV, linkcolor=blue, citecolor=blue, urlcolor=blue]{hyperref}
\usepackage[normalem]{ulem}

\usepackage{tikz}
\usetikzlibrary{calc}
\usepackage{pgf}
\usetikzlibrary{external}
\numberwithin{equation}{section}
\numberwithin{figure}{section}

\newtheorem{theorem}{Theorem}[section]

\newtheorem{proposition}[theorem]{Proposition}
\newtheorem{lemma}[theorem]{Lemma}

\theoremstyle{definition}
\newtheorem{definition}[theorem]{Definition}

\newtheorem{remark}[theorem]{Remark}

\usepackage{dsfont}
\usepackage{bbm}

\newcommand{\N}{\mathbb{N}}
\newcommand{\R}{\mathbb{R}}

\newcommand{\cB}{\mathcal{B}}

\newcommand{\cA}{\mathcal{A}}

\newcommand{\cQ}{\mathcal{Q}}

\newcommand{\om}{\omega}

\newcommand{\eps}{\varepsilon}

\newcommand{\1}{\mathbf{1}}

\newcommand{\data}{\textnormal{data}}
\renewcommand{\rho}{\varrho}

\DeclareMathOperator{\supp}{supp}

\newcommand{\esssup}{\operatornamewithlimits{ess\,sup}}
\newcommand{\essinf}{\operatornamewithlimits{ess\,inf}}

\DeclareMathOperator*{\osc}{osc}

\DeclareMathOperator{\tail}{Tail}

\renewcommand{\d}{\mathrm{\,d}}
\def\dxy{\,{\mathrm d}x\,{\mathrm  d}y}
\def\dxt{\,{\mathrm  d}x\,{\mathrm d}t}
\def\dyt{\,{\mathrm  d}y\,{\mathrm  d}t}

\def\dytau{\,{\mathrm  d}y\,{\mathrm  d}\tau}
\def\dxyt{\,{\mathrm d}x\,{\mathrm  d}y\,{\mathrm  d}t}

\renewcommand{\leq}{\leqslant}
\renewcommand{\geq}{\geqslant}
\renewcommand{\subset}{\subseteq}
\renewcommand{\supset}{\supseteq}

\def\Xint#1{\mathchoice
{\XXint\displaystyle\textstyle{#1}}%
{\XXint\textstyle\scriptstyle{#1}}%
{\XXint\scriptstyle\scriptscriptstyle{#1}}%
{\XXint\scriptscriptstyle\scriptscriptstyle{#1}}%
\!\int}
\def\XXint#1#2#3{{\setbox0=\hbox{$#1{#2#3}{\int}$}
\vcenter{\hbox{$#2#3$}}\kern-.5\wd0}}
\def\dashint{\Xint-}

\def\XXiint#1#2#3{\setbox0=\hbox{$#1{#2#3}{\iint}$}
    \vcenter{\hbox{$#2#3$}}\kern-0.5\wd0}

\makeatletter
\newcommand\avsuminner[2]{%
  {\sbox0{$\m@th#1\sum$}%
   \vphantom{\usebox0}%
   \ooalign{%
     \hidewidth
     \smash{\,\rule[.23em]{8.8pt}{1.1pt} \relax}%
     \hidewidth\cr
     $\m@th#1\sum$\cr
   }%
  }%
}
\makeatother

\makeatletter
\newcommand\avsuminnerr[2]{%
  {\sbox0{$\m@th#1\sum$}%
   \vphantom{\usebox0}%
   \ooalign{%
     \hidewidth
     \smash{\,\rule[.23em]{6pt}{0.7pt} \relax}%
     \hidewidth\cr
     $\m@th#1\sum$\cr
   }%
  }%
}
\makeatother

\let\originalleft\left
\let\originalright\right
\renewcommand{\left}{\mathopen{}\mathclose\bgroup\originalleft}
\renewcommand{\right}{\aftergroup\egroup\originalright}

\usepackage{titlesec}

\newcommand{\addperiod}[1]{#1.}
\titleformat{\section}
   {\centering\normalfont\bfseries\large}{\thesection.}{0.5em}{}
\titleformat{\subsection}[runin]
  {\normalfont\bfseries\normalsize}
  {\thesubsection.}
  {0.5em}
  {\addperiod}
\titleformat{\subsubsection}[runin]
  {\normalfont\bfseries}
  {\thesubsubsection.}
  {0.5em}
  {\addperiod}
\titleformat*{\subsubsection}{\normalfont\itshape}
\titleformat*{\paragraph}{\bfseries}
\titleformat*{\subparagraph}{\large\bfseries}

\title{Nonlocal parabolic De Giorgi classes} 

\author{ Simone Ciani
\thanks{Department of Mathematics of the University of Bologna, Piazza Porta San Donato, 5, 40126 Bologna, Italy.
{\footnotesize \href{mailto:simone.ciani3@unibo.it}{simone.ciani3@unibo.it}. 
}
}
  \and
  Kenta Nakamura
\thanks{Institute of Natural Sciences, Nihon University, Tokyo, Japan.
{\footnotesize \href{mailto:kentanak55@gmail.com}{kentanak55@gmail.com} (corresponding author). 
}
}
}

\usepackage[nottoc,notlot,notlof]{tocbibind}
\usepackage{fancyhdr}
\usepackage{extramarks}
\date{\today}

\begin{document}

\maketitle

\maketitle
\begingroup
\renewcommand{\thefootnote}{\ifcase\value{footnote}\or*\or**\fi}

\endgroup

\begin{abstract}
We propose a new paradigm for the point-wise regularity theory of parabolic nonlocal problems, by addressing directly the elements of a wide parabolic energy class. First we carry on a refined analysis of their local boundedness under optimal tail conditions, and then prove several weak Harnack estimates through a purely measure-theoretical framework. Then we give a novel proof of the nonlocal parabolic Harnack inequality, that avoids any covering argument or lemma {\it à la} John-Nirenberg and is valid regardless of any comparison principle. The regularity program is completed by addressing local H\"{o}lder estimates, eventually leading to a Liouville-type theorem.
\end{abstract}

\setcounter{tocdepth}{2}
\renewcommand{\baselinestretch}{0.90}\normalsize
{\small \tableofcontents}
\renewcommand{\baselinestretch}{1.0}\normalsize

\section{Introduction}

\subsection*{Motivation}
One of the central challenges in the modern regularity theory of integro-differential equations is to understand the interplay between nonlocal long-range interactions and the degeneracy caused by nonlinearity. This paper proposes a unified strategy to overcome this difficulty by developing a robust regularity theory for the \emph{nonlocal parabolic De Giorgi class}. A De Giorgi class, roughly speaking, is a family of functions satisfying certain energy inequalities (often referred to as Caccioppoli inequalities) and moreover, belonging to the appropriate function spaces that allow these inequalities to make sense \footnote{It is well-known that, the De Giorgi class is the natural environment for weak solutions to quasi-linear elliptic equations in divergence form, and quasi-minima of functionals of the Calculus of Variations, see the original work \cite{GiaGiu} or the book \cite{Giusti}.}. Our motivation is for developing the framework of the nonlocal parabolic De Giorgi classes, rather than focusing purely on specific equations from the outset, is to provide a robust regularity theory with broad applicability. By adopting this unified framework, we circumvent any comparison principle and bypass any technical aspect tied to a particular equation (as in the Moser--Trudinger method). This allows our results to be applied to a wide variety of nonlocal problems, such as jump processes, nonlocal minimal surfaces, nonlocal diffusion physics appearing in kinetic models. The purpose of this paper is also to present a modern De Giorgi--Nash--Moser theory for the De Giorgi class in its parabolic nonlocal setting, which has strong analogies to the familiar theory of various doubly nonlinear parabolic equations such as the nonlocal Trudinger equation. For the elements of the De Giorgi class (see Definition~\ref{def of DG}), the (informal) main results are:
\begin{itemize}
\item \emph{Quantitative and refined local boundedness under the optimal condition}. We develop a quantitative local boundedness (Theorem~\ref{Thm:boundedness}) under optimal tail conditions. This significantly relaxes the $L^\infty$--tail to $L^{p-1}$--tail requirements found in previous works, providing the most general framework up to date.

\item \emph{Unified Harnack estimates}. By developing new measure theoretical propagation lemmas, we prove several nonlocal weak Harnack estimates under the minimal tail condition, namely the $L^{p-1+\eps}$-tail for $\eps \in (0,\infty)$ (Theorems~\ref{Thm:weakHarnack1} and~\ref{Thm:weakHarnack2}). Moreover, we establish the nonlocal strong Harnack inequality under the minimal tail condition (Theorem~\ref{Thm:fullHarnack}). 

\item \emph{H\"{o}lder regularity}. We develop the De Giorgi measure theoretical approach under minimal tail conditions, thereby showing that members of the parabolic De Giorgi class are H\"{o}lder continuous (Theorem~\ref{Thm:Holder modulus}). 
\item \emph{A Liouville-type rigidity}. As a direct consequence of optimal Hölder continuity prove a Liouville-type rigidity theorem, which implies that bounded elements of entire De Giorgi class are constant (Theorem~\ref{Thm:Liouville}).
\end{itemize}
\smallskip


The nonlocal parabolic De Giorgi class that we address in this paper encompasses the parabolic counterpart of the seminal work of Cozzi~\cite{Coz17} and embodies local weak solutions to the fractional heat equation. Hereafter, we give a complete picture of the regularity properties enjoyed by its members, proving various versions of improved weak Harnack estimates, Harnack estimates and local H\"older continuity. The major novelty consists in avoiding parabolic covering arguments by mixing the measure theoretical approach leading to a Weak Harnack inequality, with the strategy of Moser of chaining it with refined $L^{\infty}$ bounds. This simple observation is reported here at the end of Section \ref{Sect.2}, and the versatility of the proof is therefore attractive and pivotal for further developments in the fractional nonlinear setting. The results are new even for the linear case, thereby showing along the strategy of \cite{Lia25} that the recent achievements of Kassmann and Weidner~\cite{KW23a} are structural properties, independent of any underlying equation (see also \cite{Strom} for a first global approach). Precise results of will be stated at Section~\ref{Sect.2.2},  while here we present their simplest qualitative consequence.

\begin{theorem}[Nonlocal strong Harnack inequality]\label{Thm:fullHarnack}
Let $u \in \mathbf{PDG}^{s,p}(\Omega_T,\gamma_{\mathrm{DG}}, \eps)$ in the sense of Definition~\ref{def of DG} be such that $u \geq 0$ in $\cQ_{R_0}$ with $R_0:=4\cdot 6^{\nicefrac{1}{sp}}\rho$. Then, there exists a constant $C_\mathsf{H}$ depending only on the $\data$ such that
\begin{align*}
&\sup_{B_{\rho}(x_0) \times \left(t_0-\frac{1}{2}(2\rho)^{sp}, t_0\right]}u +\left(\dashint_{t_0-\frac{1}{2}(2\rho)^{sp}}^{t_0} \Big[\tail (u_+(t) ; B_\rho(x_0) )\Big]^{p-1}\d{t}\right)^{\frac{1}{p-1}} \notag\\[2pt]
& \quad \leq C_\mathsf{H}\left[ \inf_{B_\rho (x_0)\times \left(t_0+\frac{3}{4}(4\rho)^{sp}, t_0+(4\rho)^{sp} \right]}u +\left(\dashint_{I_{R_0}(t_0)} \Big[\tail \left(u_-(t)\,;B_{4\rho}(x_0)\right)\Big]^{p-1+\eps}\d{t}\right)^{\frac{1}{p-1+\eps}}\right],
\end{align*}
provided that the following inclusion is satisfied
\[
B_{2\rho}(x_0) \times \left(t_0-(2\rho)^{sp}, t_0+6(4\rho)^{sp}\right] \subset \cQ_{R_0} \Subset \Omega_T.
\]
\end{theorem}
\begin{remark}
The definition of the second term appearing on the right side, called the ``tail'' term, is presented in Section~\ref{Fractional spaces}. If we impose the global non-negativity of $u$,  the influence of the distant values of $u_-$ fades and the Harnack estimate gets simplified into the classic statement
\begin{align*}
\sup_{B_{\rho}(x_0) \times \left(t_0-\frac{1}{2}(2\rho)^{sp},t_0\right]}u \leq C_\mathsf{H} \inf_{B_\rho (x_0)\times \left(t_0+\frac{3}{4}(4\rho)^{sp}, t_0+(4\rho)^{sp} \right]}u\,. 
\end{align*} 
Since we deal with local estimates, the time-gap phenomenon is unavoidable, see ~\cite[Counterexample 6.7, Page 38]{LW25}.
\end{remark}

\subsection*{Brief review of the literature}
The regularity theory of \emph{De Giorgi classes} emerged in the 1960s with the celebrated work by De Giorgi~\cite{DG57} through the study of minimizers of regular functionals
\[
u \quad\mapsto \quad \mathscr{F}[u \,;\Omega]:=\int_\Omega f(x,u, \nabla u)\d{x}\,.
\]\noindent As observed in \cite{LU}, the local theory developed by De Giorgi in \cite{DG57} is valid for functions satisfying merely the Caccioppoli energy inequalities. In the years that came great has been the impact of the method invented by De Giorgi, since a whole new set of functions called {\it quasi-minima}, embodying solutions to elliptic equations as well as minima of functionals and obstacle problems, had been shown to belong to the De Giorgi class. We refer to the book \cite{Giusti} for a comprehensive historical overview. In particular, quite some years later DiBenedetto and Trudinger (see~\cite{DT}) proved that elements of the elliptic De Giorgi class indeed satisfy a point-wise Harnack inequality. Their approach relied in deriving positivity estimates in measure with a power-like dependance, by employing a proper Krylov-Safonov covering argument. Even higher integrability of the gradient ~\cite{DG} and Phragmen-Lindel\"of type results are still valid in the framework of De Giorgi classes, disregarding of any sort of comparison principle, \cite{CiaGiaLi}. Some of the main results are valid also in metric measure spaces, as Kinnunen and Shanmugalingam~\cite{KS} have shown. 
\smallskip

On the parabolic side, the pioneering work of Ladyzhenskaya, Solonnikov and Ural'tseva~\cite{Lady} gave the start to the study of parabolic De Giorgi classes while, only later, Wieser~\cite{Wieser} and Wang~\cite{Wieser} proved that a parabolic notion of quasi-minima fits in the class, and that a Harnack inequality holds true also in this case, by means of a complex parabolic covering argument. After that, these classes have been intervened and developed by many authors. For instance, we have been mainly motivated by the work of Gianazza and Vespri (see~\cite{GV}), that extended the original definition of parabolic De Giorgi classes given in~\cite{Lady} to a more general growth of order $p>1$, and proved the Harnack estimates directly by sticking with the measure theoretical technique. Their substantial technique is condensed in a theoretical expansion of positivity machine, combined with the so-called clustering lemma (see~\cite{DGV, DV}). As we will see hereafter, a main difference is that there is no need of the local clustering Lemma (or fractional dicrete isoperimetric inequality, compare with \cite{Coz17}) to get a power-like expansion of positivity in the fractional framework: being a nonlocal object defined in all the space, the fractional Caccioppoli inequalities already encode all the information needed.  Finally, the aforementioned regularity results have been extended to the metric measure space by~\cite{KMMP} for the only parabolic linear case. This short preamble is just to highlight our guiding principle that {\it point-wise regularity properties such as oscillation and Harnack estimates are consequences of the membership to an energy class}; it is clearly an incomplete list, and it refers mostly to our taste and inspiration in the subsequent framework. The literature also in this case is very wide, we refer to \cite{DG23} for a first introduction.
\smallskip

Differently to the local case, the study of nonlocal De Giorgi classes has not seen so many contributions. As for nonlocal elliptic De Giorgi classes, in ~\cite{Min11} a first nonlocal De Giorgi type iteration had been implemented with the aim of estimating pointwise local gradient bounds, while Cozzi (see~\cite{Coz17}) gave a whole satisfying picture for the nonlinear nonlocal case, through a fine iterative control of the tail term, see also the foundational works by Di~Castro, Kuusi, and Palatucci~\cite{DKP14, DKP16} for the case of minimizers of fractional integrals. In his work, Cozzi proved various properties such as Harnack inequality and H\"older continuity, by means of a weak Harnack estimate obtained by a sophisticated cube decomposition {\emph{\'a la Krylov--Safonov}}. See \cite{CCMV25} for a brief proof that bypasses the covering argument. Up to our knowledge, the study of nonlocal parabolic De Giorgi classes was firstly addressed in~\cite{Nak23}. Clearly, the aim of the present work is to complete that picture.

\section{Preliminaries and the main results}\label{Sect.2}

\subsection{Notation}
We introduce quite a bit of notation, most of it standard in the literature. We will be coherent with the notation of \cite{Nak23}.

\subsubsection{General notation}

We denote $s \wedge k := \min\{s,k\}$ and $s \vee k:=\max\{s, k\}$. Moreover, we write $(s-k)_\pm:=\pm(s-k) \vee 0$. For nonnegative real numbers $a,b \in \R$, the symbol $a \lesssim b$ means that there exists a constant $C>0$ such that $a \leq C b$. When the constant $C$ depends on parameters, say $d,s,p$, we will denote $a \lesssim_{d,s,p} b$. The symbol $\1_{A}$ denotes the indicator function on a set $A$. To shorten the notation, we denote $\sup_A \equiv \esssup_A$  and  $\inf_A \equiv \essinf_A$, respectively. 
 
By $|U|$, we denote the Lebesgue measure of a measurable set $U \subset \R^d$ with $d \geq 1$. If  $0<|U|<\infty$, then for any integrable function $g:U \to \R$, its integral average is denoted by
\[
(g)_U:=\dashint_{U}g(x)\d{x}:=\frac{1}{|U|}\int_{U}g(x)\d{x}.
\]

Finally, we denote by $C$, $c$ or sometimes $\gamma$ different positive constants in a given context, which may vary from line to line. 
To lighten that notation, we denote
\[
\data := \left(d,s,p,\gamma_{\mathrm{DG}}, \varepsilon\right)
\]
where $\gamma_{\mathrm{DG}}>0$ and $\varepsilon \in (0,\infty]$ are the constants of the \emph{De Giorgi constant} appearing in Definition~\ref{def of DG}.  When $\eps=\infty$, we interpret in a natural way that the $\data$ depends only on $(d,s,p,\gamma_{\mathrm{DG}})$ otherwise stated.

\subsubsection{Space-time cylinders}
The open ball of radius $\rho>0$ centered at $x_0 \in \R^d$ is denoted by $B_\rho(x_0):=\left\{x \in \R^d: |x-x_0|<\rho\right\}$. For a fixed vertex $z_0=(x_0,t_0) \in \R^d \times \R$, we define the (one-sided) parabolic cylinder $Q_{\rho, \theta}(z_0)$ by
\[
Q_{\rho,\theta}(z_0):=B_\rho(x_0) \times (t_0-\theta, t_0]; 
\]
in particular, when $\theta=\rho^{sp}$ we will shorten $Q_{\rho}(z_0)\equiv Q_{\rho,\rho^{sp}}(z_0)$. If no confusion arises or it is clear from the context which center is meant, we will omit it from the notation, that is, $B_\rho \equiv B_\rho(x_0)$ and $Q_{\rho,\theta}\equiv Q_{\rho,\theta}(z_0)$, etc. Throughout the paper, we fix the two-sided ambient cylinder
\[
\cQ_R(z_0):=B_{R}(x_0) \times (t_0-R^{sp},t_0+R^{sp}]=:B_R(x_0) \times I_R(t_0).
\]
Throughout the paper, for $T>0$ and a bounded domain $\Omega \subseteq \R^d$  we denote $\Omega_T:=\Omega \times (0,T)$. 

\subsubsection{Fractional spaces}\label{Fractional spaces}
In this subsection, we present various functional spaces used in the article. Let $s \in (0,1)$, $p \in [1,\infty)$ and let us denote the H\"{o}lder conjugate of $p$ by $p^\prime:=p/(p-1) \in (1,\infty]$. The following nonlocal \emph{tail} captures the long-range interactions caused by the nonlocal problem:
\[
\tail\left(u\,; B_R(x_0)\right):=\left(R^{sp}\int_{\R^d \setminus B_R(x_0)}\frac{|u(x)|^{p-1}}{|x-x_0|^{d+sp}}\d{x} \right)^{\nicefrac{1}{(p-1)}}, \quad x_0 \in \R^d, \quad R>0.
\]
When $x_0=0$ or it is clear from the context, we will often write $\tail\left(u\,;B_R(x_0)\right) =\tail(u\,;B_R)$. We define the corresponding weighted Lebesgue space, called \emph{Tail space}, by
\[
L^{p-1}_{sp}(\R^d):=\left\{u \in L^{p-1}_\mathrm{loc}(\R^d) : \int_{\R^d}\frac{|u(x)|^{p-1}}{(1+|x|)^{d+sp}}\d{x}<\infty \right\}\,.
\]
It is straightforward to check that
\[
u \in L^{p-1}_{sp}(\R^d) \quad \iff \quad \tail\left(u\,; B_R(x_0)\right)<\infty, \quad \forall x_0 \in \R^d, \quad \forall R>0.
\]
Such a nonlocal tail and tail space were originally introduced in~\cite{DKP14,DKP16}. 

We next introduce the \emph{fractional Sobolev space} $W^{s,p}(\Omega)$. A measurable function $u:\Omega \to \R$ is in the fractional Sobolev space $W^{s,p}(\Omega)$ if and only if, as for the norm,
\[
\|u\|_{W^{s,p}(\Omega)}:=\|u\|_{L^p(\Omega)}+[u]_{W^{s,p}(\Omega)} <+\infty,
\]
where 
\[
[u]_{W^{s,p}(\Omega)}:=\left(\int_\Omega \int_\Omega \frac{|u(x)-u(y)|^p}{|x-y|^{d+sp}}\dxy \right)^{\nicefrac{1}{p}}
\]
is the \emph{Gagliardo} semi-norm. The fundamental tools and useful results in the fractional Sobolev spaces are presented in the comprehensive monographs~\cite{DNPV12, KP18, FeRo24, ADV25}.

Concerning the parabolic setting, we refer to the Bochner integral:  given $p \in [1,\infty)$, $I \subset \R$ and an arbitrary Banach space $X$, we denote by $L^p(I ; X)$ the space of Lebesgue-measurable mappings $u : I \to X$ such that
\[
\|u\|_{L^p(I ; X)}:=\left(\int_I \|u(\cdot, t)\|_X^p\d{t} \right)^{\nicefrac{1}{p}}<\infty.
\]
Finally, we define $C(I; X)$ as the space of continuous maps $t \mapsto \|u(\cdot,t)\|_X$.

\subsection{Statement of the main results}\label{Sect.2.2}
We begin by presenting the definition of De Giorgi class adopted in this paper, which partially follows those in celebrated references~\cite[Definition 2.1]{GV}, ~\cite{Lia21} in terms of the parabolic type, and generalizes the one in~\cite[Definitions 1 and 2]{Nak23}.
\begin{definition}[Nonlocal parabolic De Giorgi class]\label{def of DG}
For a vertex $z_0=(x_0,t_0) \in \Omega_T$, fix a cylinder $Q_{\rho,\tau}(z_0) \Subset \Omega_T$.  Let $\zeta$ be a nonnegative, piecewise smooth cutoff function such that $\zeta(\cdot,t)$ is compactly supported in $B_\rho(x_0)$ for all $t \in (t_0-\tau,t_0)$.  Let $p \in (1,\infty), s \in (0,1)$ and $\eps \in (0, \infty]$. A measurable function $u$ belongs to the \emph{nonlocal parabolic De Giorgi class} $\mathbf{PDG}_{\pm}^{s,p}(\Omega_T,\gamma_{\mathrm{DG}}, \eps)$ if 
\begin{itemize}
\item $u$ belongs to the functional space
\begin{equation}\label{d1}
u \in C\left([0,T]\,; L^p(\Omega)\right) \,\cap\, L^p\left(0,T\,; W^{s,p}(\Omega)\right);
\end{equation}
\item the long-range condition for every $x_0 \in \R^d$ and $\rho>0$
\begin{equation}\label{d1'}
\tail\left(u\,; B_\rho(x_0)\right) \in L^{p-1+\eps}_\mathrm{loc}([0,T])\quad \iff \quad u \in L^{p-1+\eps}_\mathrm{loc}\left(0,T\,; L_{sp}^{p-1}(\R^d)\right),
\end{equation}
is satisfied;
\item there exists a constant $\gamma_{\mathrm{DG}}$ such that the following inequality is valid:
\begin{align}\label{d2}
&\sup\limits_{t_0-\tau<t<t_0}\int_{B_\rho(x_0)}\zeta^pw_\pm^p(x,t)\d{x} \notag\\[4pt]
&\quad \quad \quad +\int_{t_0-\tau}^{t_0}\int_{B_\rho(x_0)}\int_{B_\rho(x_0)}\min\left\{\zeta^p(x,t), \zeta^p(y,t)\right\}\dfrac{\big|w_{\pm}(x,t)-w_{\pm}(y,t)\big|^p}{|x-y|^{d+sp}}\dxyt \notag \\[4pt]
& \quad \quad \quad +\int_{Q_{\rho,\tau}(z_0)}\zeta^pw_\pm(x,t)\d{x}\left(\int_{\R^d}\frac{w_\mp^{p-1}(y,t)}{|x-y|^{d+sp}}\d{y}\right)\d{t} \notag\\[4pt]
& \quad \quad \leq \int_{B_\rho(x_0)}\zeta^pw_\pm^p(x,t_0-\tau)\d{x}+\gamma_{\mathrm{DG}}\int_{Q_{\rho,\tau}(z_0)}\left|\partial_t \zeta^p\right|w_\pm^p\dxt \notag\\[4pt]
&\quad \quad \quad +\gamma_{\mathrm{DG}}\int_{t_0-\tau}^{t_0}\int_{B_\rho(x_0)}\int_{B_\rho(x_0)}\min\left\{w_\pm^p(x,t), w_\pm^p(y,t)\right\}\dfrac{\big|\zeta(x,t)-\zeta(y,t)\big|^p}{|x-y|^{d+sp}}\dxyt \notag\\[4pt]
&\quad \quad \quad +\gamma_{\mathrm{DG}}\int_{Q_{\rho,\tau}(z_0)}\zeta^pw_\pm(x,t)\d{x} \left(\sup_{x\,\in\, \supp \zeta(\cdot,t)}\int_{\R^d \setminus B_\rho(x_0)}\frac{w_\pm^{p-1}(y,t)}{|x-y|^{d+sp}}\d{y}\right)\d{t} \notag\\[4pt]
&\quad \quad  \quad \mp\gamma_{\mathrm{DG}}\int_{Q_{\rho,\tau}(z_0)}k^\prime(t)\zeta^p\Big(|u|+|k(t)|\Big)^{p-2}w_\pm(x,t)\dxt,
\end{align}
where we have denoted $
w_\pm(x,t):=\left(u(x,t)-k(t)\right)_\pm$ for brevity, where $k(t)$ is an arbitrary absolutely continuous function on $(0,T)$. 
\end{itemize}
Moreover, we define
\[
\mathbf{PDG}^{s,p}(\Omega_T,\gamma_{\mathrm{DG}},\eps):=\mathbf{PDG}_+^{s,p}(\Omega_T,\gamma_{\mathrm{DG}}, \eps) \cap \mathbf{PDG}_-^{s,p}(\Omega_T,\gamma_{\mathrm{DG}}, \eps),
\]
\end{definition}
\begin{remark}
It immediately follows from the definition that
\[
u \in \mathbf{PDG}^{s,p}_\pm(\Omega_T,\gamma_{\mathrm{DG}},\eps) \quad \iff \quad -u \in \mathbf{PDG}^{s,p}_\mp(\Omega_T,\gamma_{\mathrm{DG}},\eps).
\]
\end{remark}
%
The first main result of this paper is a more quantitative version of the local boundedness proved
in~\cite{Nak23}. Our approach is different from the one of~\cite[Theorem 1.1, Corollary 3.1]{Nak23} and is based on the time-dependent truncation methods, whose linear case was introduced in\cite{KW23a}.

\begin{theorem}[Quantitative local boundedness]\label{Thm:boundedness}
Let $u\in\mathbf{PDG}_\pm^{s,p}(\Omega_T,\gamma_{\mathrm{DG}}, \eps)$ in the sense of  Definition~\ref{def of DG}. Suppose that $Q_{\rho, \theta}(z_0) \subseteq \Omega_T$ with $z_0=(x_0,t_0) \in \Omega_T$, and let $\sigma \in  (0,1)$ and $\nu \in (0,p]$. There exist positive constants $C, q <\infty$, both depending only on $(d,s,p,\gamma_{\mathrm{DG}})$, and $C_\nu \equiv C_\nu (d,s,p,\gamma_{\mathrm{DG}},\nu)<\infty$ such that
\begin{align*}
\sup_{Q_{\sigma \rho, \sigma \theta}(z_0)}u_\pm &\leq \frac{C}{(1-\sigma)^{(d+p)q+\frac{d+sp}{p-1}}}\left( \frac{\theta}{(\sigma \rho)^{sp}}\dashint_{t_0-\theta}^{t_0} \Big[ \tail \left(u_\pm(t)\,; B_{\sigma \rho}(x_0)\right) \Big]^{p-1}\d{t}\right)^{\nicefrac{1}{(p-1)}}\\
&\quad \quad \quad \quad \quad \quad +C_\nu \left(\frac{\boldsymbol{\cA}}{(1-\sigma)^{(d+p)qp}}\dashint_{Q_{\rho,\theta}(z_0)}u_\pm^\nu\dxt \right)^{\nicefrac{1}{\nu}},
\end{align*}
holds, where 
\[
\boldsymbol{\cA}:=\left(\dfrac{\rho^{sp}}{\theta}+\dfrac{1}{\sigma^{d+sp}}\right)^{pq}\left(\dfrac{\theta}{\rho^{sp}}\right) \quad \mbox{and} \quad
q:=\begin{cases}
\nicefrac{(d+sp)}{p^2s}, \quad &\textrm{if}\quad sp<d,\\[1mm]
\nicefrac{3}{p}, \quad &\textrm{if}\quad sp\geq d.
\end{cases}
\]
\end{theorem}


\begin{remark} The constant $C_{\nu}$ blows up as soon as $\nu$ vanishes. Moreover, comparing with the previous result~\cite[Theorem 1.1]{Nak23}, one can obtain by interpolation, for $u \in \mathbf{PDG}_\pm^{s,p}(\Omega_T,\gamma_{\mathrm{DG}}, \infty)$,
\begin{align}\label{eq:another-bnd-1}
\sup_{Q_{\nicefrac{\rho}{2}, \nicefrac{\rho^{sp}}{2}}(z_0)}u_\pm &\leq \delta \sup_{t \in (t_0-\rho^{sp}, t_0)}\tail \left(u_\pm(t)\,; B_{\nicefrac{\rho}{2}}(x_0)\right) +C(\delta)\left(\dashint_{Q_{\rho}(z_0)}u_\pm^p\dxt\right)^{\nicefrac{1}{p}}
\end{align}
for a constant $C(\delta)$ that blows up as $\delta \downarrow 0$. The parameter $\delta \in (0,1]$, appearing in the display, plays a role of interpolation between the local and nonlocal terms.

One could expect naturally in Theorem~\ref{Thm:boundedness} that the $L^\infty$-tail in~\eqref{eq:another-bnd-1} replaced by $L^{p-1}$-ones, namely that for all $\delta \in (0,1)$
\begin{align}\label{eq:another-bnd-2}
\sup_{Q_{\nicefrac{\rho}{2}, \nicefrac{\rho^{sp}}{2}}(z_0)}u_\pm &\leq \delta \left(\dashint_{t_0-\rho^{sp}}^{t_0} \Big[\tail \left(u_\pm(t)\,; B_{\nicefrac{\rho}{2}}(x_0)\right) \Big]^{p-1}\d{t}\right)^{\nicefrac{1}{(p-1)}}+C(\delta)\left(\dashint_{Q_\rho(z_0)}u_\pm^p\dxt\right)^{\nicefrac{1}{p}}.
\end{align}
Unfortunately, this happens to be false, because of the refined tail term. The argument is similar to the ones presented in~\cite[Proposition 5.1]{Lia25} or~\cite[Example 5.2]{KW23a}.
\end{remark}

Next, we turn our attention to several nonlocal weak Harnack inequalities. The first result is the following nonlocal weak Harnack inequality under $L^{p-1+\eps}$-tail condition.
\begin{theorem}[Nonlocal weak Harnack inequality I]\label{Thm:weakHarnack1}
Let $u \in \mathbf{PDG}_{-}^{s,p}(\Omega_T,\gamma_{\mathrm{DG}}, \eps)$ in the sense of~Definition~\ref{def of DG} be such that $u \geq 0$ in $\cQ_R$. Then, there exist constants $\beta \in (0,1)$ and $C_{\mathsf{wH}}$, both depending only on $\data$, such that
\begin{align*}
&\left(\dashint_{B_\rho(x_0)}u^\beta(x,t_0)\d{x}\right)^{\nicefrac{1}{\beta}} \\[2pt]
& \quad \quad \leq C_{\mathsf{wH}} \inf_{B_{2\rho(x_0)}}u(t)+\left(\frac{\rho}{R}\right)^{\frac{sp\eps}{(p-1)(p-1+\eps)}} \left(\dashint_{I_R(t_0)} \Big[\tail \left(u_-(t)\,;B_R(x_0)\right)\Big]^{p-1+\eps}\d{t}\right)^{\nicefrac{1}{(p-1+\eps)}}
\end{align*}
for almost all $t \in \left(t_0+\frac{1}{2}(8\rho)^{sp},t_0+2(8\rho)^{sp}\right]$, provided that
\[
B_{2\rho}(x_0) \times \left(t_0+\frac{1}{2}(8\rho)^{sp},t_0+2(8\rho)^{sp}\right]\subset \cQ_R \Subset \Omega_T.
\]
\end{theorem}

We observe that, letting $\eps = \infty$  in the above display, the $L^{p-1+\eps}$-tail term on the right side becomes $L^\infty$-tail exactly, and therefore the result covers~\cite[Theorem 1.1]{Nak23}. Now, in sharp contrast to the argument presented for Theorem~\ref{Thm:weakHarnack1}, we can derive a nonlocal weak Harnack inequality that improves the previous by controlling also the long-range influences.

\begin{theorem}[Nonlocal weak Harnack inequality II]\label{Thm:weakHarnack2}
Let $u \in \mathbf{PDG}_{-}^{s,p}(\Omega_T,\gamma_{\mathrm{DG}}, \eps)$ in the sense of~Definition~\ref{def of DG} be such that $u \geq 0$ in $\cQ_R$. Then, there exists $\eta \in (0,1)$ depending only on $\data$ such that
\begin{align*}
\inf_{B_\rho(x_0)} u(t)&+\left(\frac{\rho}{R}\right)^{\frac{sp\eps}{(p-1)(p-1+\eps)}}\left(\dashint_{I_R(t_0)} \Big[\tail \left(u_-(t)\,;B_R(x_0)\right)\Big]^{p-1+\eps}\d{t}\right)^{\nicefrac{1}{(p-1+\eps)}} \\
& \geq \eta \left[\left(\dashint_{Q_{2\rho}(z_0)}u^{p-1}\dxt \right)^{\nicefrac{1}{(p-1)}}+\left(\dashint_{t_0-(2\rho)^{sp}}^{t_0} \Big[\tail \left(u_+(t)\,;B_{2\rho}(x_0)\right)\Big]^{p-1}\d{t} \right)^{\nicefrac{1}{(p-1)}}\right]
\end{align*}
for almost every $t  \in \left(t_0+\frac{3}{4}(4\rho)^{sp}, t_0+(4\rho)^{sp} \right]$, provided that
\[
B_{2\rho} (x_0)\times \left(t_0-(2\rho)^{sp}, t_0+6(4\rho)^{sp}\right] \subset \cQ_R \Subset \Omega_T.
\]
\end{theorem}

\begin{remark} By directly applying the H\"older inequality, it is straightforward to check that when $(p-1)>\beta$ (this condition is satisfied for instance if $p\geq2$), Theorem~\ref{Thm:weakHarnack1} follows from Theorem~\ref{Thm:weakHarnack2}; while the aforementioned implication is not for free when $1<p<2$. 
 \end{remark}

Finally, the property of local H\"{o}lder continuity. To the best of our knowledge, the H\"{o}lder continuity for nonlocal De Giorgi classes has only been established  for the elliptic case (see~\cite{Coz17}) up to now. The control of tail terms is undoubtedly decisive and the time-gap phenomenon encumbers the elliptic strategy Harnack-to-H\"older estimates. Here we found expedient to adopt the De Giorgi measure theoretical approach, in a fashion reminiscent of the proof of \cite{Lia24b}. We derive hereafter a H\"{o}lder modulus of continuity, under assuming an optimal $L^{p-1+\eps}$-tail condition. It is noteworthy that the result remains novel even when $\eps = \infty$.

\begin{theorem}[H\"{o}lder modulus of continuity]\label{Thm:Holder modulus}
Let $u \in \mathbf{PDG}^{s,p}(\Omega_T,\gamma_{\mathrm{DG}}, \eps)$ in the sense of Definition~\ref{def of DG}. Then $u$ has a locally H\"{o}lder continuous representative in $\Omega_T$. More precisely, let $z_0=(x_0,t_0) \in \Omega_T$, and assume the validity of the inclusion $Q_r(z_0) \subset Q_\rho(z_0) \subset \cQ_R \Subset \Omega_T$. Then, there exist constants $C(\data) \in (1,\infty)$ and $\beta(d,s,p,\gamma_{\mathrm{DG}}) \in (0,1)$ such that 
\[
\osc_{Q_r(z_0)} u \leq C\boldsymbol{\om} \left(\frac{r}{\rho}\right)^\beta,
\]
holds, where
\[
\boldsymbol{\om}=2\sup_{\cQ_R} |u|+\left(\frac{\rho}{R}\right)^{\frac{sp\eps}{(p-1)(p-1+\eps)}} \left(\dashint_{I_R(t_0)} \Big[\tail \left(u(t)\,; B_R(x_0)\right)\Big]^{p-1+\eps}\d{t}\right)^{\nicefrac{1}{(p-1+\eps)}}.
\]
The constant $\boldsymbol{\gamma}_\eps$ is stable as $\eps \to \infty$.
\end{theorem}

As a fairly straightforward consequence, we have the following rigidity result.

\begin{theorem}[Liouville-type rigidity theorem]\label{Thm:Liouville} Let $u$ belong to $\mathbf{PDG}^{s,p}(\R^{d+1},\gamma_{\mathrm{DG}}, \eps)$ in the local sense of Definition~\ref{def of DG}. If $u$ is bounded, then it is constant. 
\end{theorem}

\subsection{Proofs of Theorems~\ref{Thm:fullHarnack} and~\ref{Thm:Liouville}}\label{Sect.2.3}
The first theorem is obtained in Moser--Trudinger's fashion by chaining Theorem~\ref{Thm:boundedness} with Theorem~\ref{Thm:weakHarnack2}, while the second one is reminiscent of \cite{Gla69}.

\begin{proof}[Proof of Theorem~\ref{Thm:fullHarnack}]
Let $(x_0,t_0)=(0,0)$ for simplicity. Let $M:=\sup_{Q_{\rho, \nicefrac{(2\rho)^{sp}}{2}} }u$. Assume that $u$ is nonnegative $\cQ_{R}$ with $R \in (0, \infty)$ being specified later. In Theorem~\ref{Thm:boundedness}, let $\nu=p-1$, $\sigma=\nicefrac{1}{2}$ and replace $\rho$ by $2\rho$, and, subsequently, take $\theta=(2\rho)^{sp}$. Then, the relevant constants depend only on $(d,s,p,\gamma_{\mathrm{DG}})$. Thus, we have that
\begin{align*}
M \leq C\left(\dashint_{-(2\rho)^{sp}}^0 \Big[\tail\left(u_+(t) ; B_\rho\right)\Big]^{p-1}\d{t}\right)^{\nicefrac{1}{(p-1)}}+C\left(\dashint_{Q_{2\rho}}u_+^{p-1}\dxt\right)^{\nicefrac{1}{(p-1)}}
\end{align*}
for a constant $C(d,s,p,\gamma_{\mathrm{DG}})<\infty$. As for the tail term, we can estimate that
\begin{align*}
\Big[\tail\left(u_+(t)\,; B_\rho\right)\Big]^{p-1}
&\leq \frac{1}{2^{sp}}\Big[\tail\left(u_+(t) ; B_{2\rho}\right)\Big]^{p-1}+C(d)\dashint_{B_{2\rho}}u_+^{p-1}(t)\d{x}.
\end{align*}
On the other hand, by Theorem~\ref{Thm:weakHarnack2} there exists $\eta(\data) \in (0,1)$ such that
\begin{align*}
&\inf_{B_\rho\times \left(\frac{3}{4}(4\rho)^{sp}, (4\rho)^{sp} \right]}u+\left(\frac{\rho}{R}\right)^{\frac{sp\eps}{(p-1)(p-1+\eps)}}\left(\dashint_{I_R} \Big[\tail \left(u_-(t)\,;B_R\right)\Big]^{p-1+\eps}\d{t}\right)^{\nicefrac{1}{(p-1+\eps)}} \\
& \quad \quad \quad \geq \eta \left[\left(\dashint_{Q_{2\rho}}u^{p-1}\dxt \right)^{\nicefrac{1}{(p-1)}}+\left(\dashint_{-(2\rho)^{sp}}^{0} \Big[\tail \left(u_+(t)\,;B_{2\rho}\right)\Big]^{p-1}\d{t} \right)^{\nicefrac{1}{(p-1)}}\right],
\end{align*}
provided that $B_{2\rho}\times \left(-(2\rho)^{sp}, 6(4\rho)^{sp}\right] \subset \cQ_R \Subset \Omega_T$.

Finally, we are able to deduce the following tail-controlled inequality
\begin{align}\label{e.tail.result}
&\left(\dashint_{-\frac{1}{2}(2\rho)^{sp}}^0 \Big[\tail (u_+(t) ; B_\rho )\Big]^{p-1}\d{t}\right)^{\nicefrac{1}{(p-1)}} \notag\\
&\quad \quad \quad \leq C(\data)\left[M+\left(\frac{R}{\rho}\right)^{\nicefrac{sp}{(p-1+\eps)}}\left(\dashint_{I_R} \Big[\tail \left(u_-(t)\,;B_R\right)\Big]^{p-1+\eps}\d{t}\right)^{\nicefrac{1}{(p-1+\eps)}}\right]. 
\end{align}
Indeed, we apply~\eqref{d3'}$_-$ with $Q_{\rho,\frac{1}{2}(2\rho)^{sp}} \subset Q_{2\rho,(2\rho)^{sp}}$ and $k=2M$. Using this and $u \geq 0$ in $Q_{2\rho,(2\rho)^{sp}} \subset \cQ_R$, we estimate that, for $c=c(d,s,p,\gamma_{\mathrm{DG}})<\infty$, 
\begin{align*}
&\int_{Q_{\rho,\frac{1}{2}(2\rho)^{sp}}}(u(x,t)-2M)_-\left[\int_{\R^d} \frac{(u(y,t)-2M)_+^{p-1}}{|x-y|^{d+sp}}\d{y}\right]\d{x}\d{t}\\
&\quad  \leq \frac{c}{\rho^{sp}}\int_{Q_{2\rho,(2\rho)^{sp}}}(u-2M)_-^p\d{x}\d{t} \\
&\quad \quad \quad \quad +\frac{c}{\rho^{sp}}\int_{Q_{2\rho,(2\rho)^{sp}}}(u-2M)_-\Big[\tail \left((u-2M)_-\,;B_{2\rho}\right)\Big]^{p-1}\d{x}\d{t}\\
&\quad  \leq  cM^p\rho^d+cM\rho^d \dashint_{-(2\rho)^{sp}}^0\Big[\tail \left(u_-(t)\,;B_{2\rho}\right)\Big]^{p-1}\d{t}.
\end{align*}
Using $0 \leq u \leq M$ in $Q_{\rho,\frac{1}{2}(2\rho)^{sp}} \subset \cQ_R$ and the following elemental inequality~\cite[Lemma 4.4]{Coz17}
\[
(u(y,t)-2M)_+^{p-1} \geq (1 \wedge  2^{2-p}) u_+(y,t)^{p-1}-(2M)^{p-1},
\]
a similar argument to the proof of Lemma~\ref{Lm:newmt2} implies that
\begin{align*}
&\int_{Q_{\rho,\frac{1}{2}(2\rho)^{sp}}}(u(x,t)-2M)_-\left[\int_{\R^d} \frac{(u(y,t)-2M)_+^{p-1}}{|x-y|^{d+sp}}\d{y}\right]\d{x}\d{t} \\
&\quad \quad \quad \quad \quad \geq cM\rho^d \dashint_{-\frac{1}{2}(2\rho)^{sp}}^0\Big[\tail \left(u_+(t)\,;B_{\rho}\right)\Big]^{p-1}\d{t}-cM^p\rho^d
\end{align*}
for a constant $c(d,s,p)<\infty$. Combining the above displays and rearranging yield that
\[
\left(\dashint_{-\frac{1}{2}(2\rho)^{sp}}^0\Big[\tail \left(u_+(t)\,;B_{\rho}\right)\Big]^{p-1}\d{t}\right)^{\nicefrac{1}{(p-1)}} \leq c \left[M+\left(\dashint_{-(2\rho)^{sp}}^0\Big[\tail \left(u_-(t)\,;B_{2\rho}\right)\Big]^{p-1}\d{t} \right)^{\nicefrac{1}{(p-1)}}\right].
\]
Finally, rearranging the tail term on the right side and using H\"{o}lder inequality conclude~\eqref{e.tail.result}.

Altogether, combining the preceding three estimates and rearranging conclude that
\begin{align*}
&\sup_{B_{\rho} \times \left(-\frac{1}{2}(2\rho)^{sp}, 0\right]}u +\left(\dashint_{-\frac{1}{2}(2\rho)^{sp}}^{0} \Big[\tail (u_+(t) ; B_\rho )\Big]^{p-1}\d{t}\right)^{\frac{1}{p-1}} \notag\\[2pt]
& \quad \leq C_\mathsf{H}\left[ \inf_{B_\rho \times \left(\frac{3}{4}(4\rho)^{sp}, (4\rho)^{sp} \right]}u \right. \notag\\[2pt]
&\left. \quad \quad \quad \quad \quad +\left[\left(\frac{R}{\rho}\right)^{\frac{sp}{p-1+\eps}}+\left(\frac{\rho}{R}\right)^{\frac{sp\eps}{(p-1)(p-1+\eps)}} \right]\left(\dashint_{I_R} \Big[\tail \left(u_-(t)\,;B_R\right)\Big]^{p-1+\eps}\d{t}\right)^{\frac{1}{p-1+\eps}}\right]\,.
\end{align*}
Finally, selecting $R=4\cdot 6^{\nicefrac{1}{sp}}\rho=:R_0$ in the last display concludes, taking into account that $B_{2\rho}\times \left(-(2\rho)^{sp}, 6(4\rho)^{sp}\right] \subset \cQ_{R_0}$, that the statement of Theorem~\ref{Thm:fullHarnack}. The proof is complete.
\end{proof}

We next apply Theorem~\ref{Thm:Holder modulus} to prove the Liouville type rigidity, thereby completing the proof of~Theorem~\ref{Thm:Liouville}.

\begin{proof}[Proof of Theorem~\ref{Thm:Liouville}]
Set $M:= \sup_{\R^{d+1}} |u|<+\infty$. We then estimate that
\begin{align*}
\boldsymbol{\om}&=2\sup_{\cQ_R(z_0)} |u|+\left(\frac{\rho}{R}\right)^{\frac{sp\eps}{(p-1)(p-1+\eps)}} \left(\dashint_{I_R(t_0)} \Big[\tail \left(u(t)\,;B_R(x_0)\right)\Big]^{p-1+\eps}\d{t}\right)^{\nicefrac{1}{(p-1+\eps)}}\\&\leq M\left[ 2+\left(\frac{\rho}{R}\right)^{\frac{sp\eps}{(p-1)(p-1+\eps)}}  R^{sp} \int_{\R^d\setminus{B_R(x_0)}} \frac{\d{x}}{|x-x_0|^{d+sp}}\right]\\
&\leq M\left[2+\frac{\omega_d}{sp}\left(\frac{\rho}{R}\right)^{\frac{sp\eps}{(p-1)(p-1+\eps)}} \right] \leq \gamma M,
\end{align*}
for $\gamma=2+\omega_d/sp$, where $\omega_d=|\mathbb{S}_1^{d-1}|=d|B_1|$ is the $(d-1)$-dimensional surface measure of the unit sphere in $\R^d$. We then choose $(x,t)=z_0$, $(y,t^\prime)$ in $\R^{d+1}$. Let $r>0$ be fixed big enough to satisfy $(y,t^\prime) \in Q_r(z_0)$. Finally, selecting the continuous representative of $u$ (see for~\cite[Section 2.1]{CCMV25}), we let $R=2\rho$ and send $\rho \rightarrow \infty$ in the oscillation estimate to conclude
\[
|u(x,t)-u(y,t^\prime)| \leq \osc_{Q_r(z_0)} u \leq C\boldsymbol{\om} \left(\frac{r}{\rho}\right)^\beta\leq C\gamma M \left(\frac{r}{\rho}\right)^\beta \rightarrow 0,
\] 
as desired.
\end{proof}

\section{Local boundedness}\label{Sect.3}

\subsection{Qualitative local boundedness}
In this subsection, we state the key auxiliary proposition which is used to prove Theorem~\ref{Thm:boundedness}.

\begin{proposition}[Qualitative local boundedness]\label{Prop:qualitative bnd}
For $\eps \in (0,\infty)$ let $u$ belong to $\mathbf{PDG}_\pm^{s,p}(\Omega_T,\gamma_{\mathrm{DG}}, \eps)$ in the sense of Definition~\ref{def of DG}. If  $u \in L^{p-1+\eps}_{\mathrm{loc}}(\Omega_T)$, then, $u$ is locally bounded in $\Omega_T$. More precisely, there exists a constant $C>0$ depending only on $\data$ such that on every concentric cylinders $\sigma^\prime Q_{R,S}(z_0) \subset \sigma Q_{R,S}(z_0) \subset \Omega_T$ for $z_0=(x_0,t_0) \in \Omega_T$ and $0<\sigma^\prime < \sigma \leq 1$, we have
\begin{align*}
\sup_{\sigma^\prime Q_{R,S}(z_0)} u_\pm &\leq \frac{C}{(\sigma-\sigma^\prime)^{\frac{d+p}{(\kappa-1)(p-1+\eps)}}}  \left[\left(\frac{S}{R^{sp}}+1 \right)+\frac{S}{R^{sp+\frac{d(p-1)}{p-1+\eps}}} \widetilde{\mathsf{Tail}}\right]^{\frac{\kappa}{(\kappa-1)(p-1+\eps)}} \\[4pt]
& \quad \quad \quad \times \left(\dashint_{\sigma Q_{R,S}(z_0)}\left[1+u_\pm^{p-1+\eps} \right]\dxt \right)^{\nicefrac{1}{(p-1+\eps)}},
\end{align*}
where $\kappa:=1+\frac{\kappa_\ast-1}{\kappa_\ast}$ with $\kappa_\ast$ is the Sobolev embedding exponent as in Proposition~\ref{FS}, and to lighten the notation we set $\sigma Q_{R,S}(z_0):=Q_{\sigma R, \sigma S}(z_0)$ and
\[
\widetilde{\mathsf{Tail}}:=\left(\dashint_{t_0-\sigma S}^{t_0} \Big[\tail (u_\pm(t)\,; B_{\sigma^\prime R}(x_0) )\Big]^{p-1+\eps}\d{t}\right)^{\frac{p-1}{p-1+\eps}}.
\]
\end{proposition}

\begin{proof} It suffices to check the plus case only, and we split the proof into three steps.
\smallskip

\emph{Step 1.}\quad As we will see later at the beginning of Section~\ref{Sect.4}, selecting a proper cutoff function implies~\eqref{d3'}$_+$. Let $(x_0,t_0)=(0,0)$. For $0<\rho <R$ and $0<\tau <S$, multiply~\eqref{d3'}$_+$ with $r_1=(R+\rho)/2$, $r_2=R$, $\tau_1=\tau$ and $\tau_2=S$ by $k^\alpha$ for $\alpha >-1$ and discard the two fractional terms on the left side of~\eqref{d3'}$_+$. We then integrate the resulting inequality over $k \in [0, \infty)$ to get, for every $t \in (-S,0)$, 
\begin{align}\label{eq:quali-bnd-1}
\int_0^\infty &k^\alpha\int_{B_{\frac{R+\rho}{2}}}(u(t)-k)_+^p\d{x}\d{k}\notag \\[4pt]
&\leq \gamma_{\mathrm{DG}}\left[\frac{2^pR^{(1-s)p}}{(R-\rho)^p} +\frac{1}{S-\tau}\right] \int_0^\infty k^\alpha \int_{Q_{R,S}}(u-k)_+^p\dxt \d{k} \notag \\[4pt]
&\quad \quad +\frac{\gamma_{\mathrm{DG}}  2^{d+sp}R^{d}}{(R-\rho)^{d+sp}}\int_0^\infty k^\alpha \int_{Q_{R,S}}(u-k)_+\d{x} \Big[\tail ((u-k)_+(t)\,; B_{R} )\Big]^{p-1}\d{t} \d{k},
\end{align}
After making a change of valuable $ k \to u_+ \lambda$ and Fubini's theorem, a straightforward manipulation implies that the first term on the left side of~\eqref{eq:quali-bnd-1} is written as
\begin{align*}
\int_0^\infty k^\alpha \int_{B_{\frac{R+\rho}{2}}}(u(t)-k)_+^p\d{x}\d{k}=\left(\int_0^1\lambda^\alpha (1-\lambda)^p\d{\lambda}\right)\int_{B_{\frac{R+\rho}{2}}}u_+^{\alpha+p+1}(t)\d{x}.
\end{align*}
This, together with Lemma~\ref{Lm:useful}, implies that
\[
\int_0^\infty k^\alpha \int_{B_{\frac{R+\rho}{2}}}(u(t)-k)_+^p\d{x}\d{k} \geq \frac{1}{2(\alpha+p+1)^{p+2}}\left(\int_{B_{\frac{R+\rho}{2}}}u_+^{\alpha+p+1}(t)\d{x} \right)
\]
and, likewise, the first integral on the right hand of~\eqref{eq:quali-bnd-1} is estimated as
\[
\int_0^\infty k^\alpha \int_{Q_{R,S}}(u-k)_+^p\dxt \d{k} \leq \frac{2}{\alpha+p+1}\int_{Q_{R,S}}u_+^{\alpha+p+1}\dxt.
\]
Analogously, by further appealing to Fubini's theorem and Lemma~\ref{Lm:useful} we estimate that
\begin{align*}
\int_0^\infty &k^\alpha \int_{Q_{R,S}}(u-k)_+\d{x} \Big[\tail ((u-k)_+(t)\,; B_{R} )\Big]^{p-1}\d{t} \d{k}\\
&= \left(\int_0^1 \lambda^\alpha (1-\lambda)^p\d{\lambda} \right)\int_{Q_{R,S}}u_+^{\alpha+2}\d{x} \Big[\tail (u_+(t)\,; B_{R} )\Big]^{p-1}\d{t}\\
&\leq \frac{2}{\alpha+p+1}\int_{Q_{R,S}}u_+^{\alpha+2}\d{x} \Big[\tail (u_+(t)\,; B_{R} )\Big]^{p-1}\d{t}, 
\end{align*}
which, together with H\"{o}lder's inequality with exponents $\left(\frac{p-1+\eps}{\eps}, \frac{p-1+\eps}{p-1}\right)$, yields that
\begin{align*}
\int_0^\infty &k^\alpha \int_{Q_{R,S}}(u-k)_+\d{x} \Big[\tail ((u-k)_+(t)\,; B_{R} )\Big]^{p-1}\d{t} \d{k}\\
&\leq \frac{2}{\alpha+p+1} \left(\int_{-S}^0 \Big[\tail (u_+(t)\,; B_{R} )\Big]^{p-1+\eps}\d{t}\right)^{\frac{p-1}{p-1+\eps}} \\
&\quad \quad \quad \quad \,\, \,\times \left(\int_{Q_{R,S}}u_+^{(\alpha+2)\frac{p-1+\eps}{\eps}}\dxt \right)^{\frac{\eps}{p-1+\eps}}.
\end{align*}
Combining the previous three displays with~\eqref{eq:quali-bnd-1} and rearranging give that
\begin{align}\label{eq:quali-bnd-2}
 \frac{1}{(\alpha+p+1)^{p+1}}&\sup_{t \in (-S, \,0)} \dashint_{B_{\frac{R+\rho}{2}}}u_+^{\alpha+p+1}(t)\d{x}  \notag \\[4pt]
&\leq C\left[\frac{R^{(1-s)p}}{(R-\rho)^p} +\frac{1}{S-\tau}\right] \frac{|Q_{R,S}|}{|B_{\frac{R+\rho}{2}}|}\dashint_{Q_{R,S}} u_+^{\alpha+p+1}\dxt \notag\\[4pt]
& \quad \quad \quad +\frac{C R^{d}}{(R-\rho)^{d+sp}} \frac{|Q_{R,S}|^{\frac{\eps}{p-1+\eps}}}{|B_{\frac{R+\rho}{2}}|}\left(\int_{-S}^0 \Big[\tail (u_+(t)\,; B_{R} )\Big]^{p-1+\eps}\d{t}\right)^{\frac{p-1}{p-1+\eps}}  \notag\\[4pt]
&\quad \quad \quad \quad \quad \quad \quad \quad \quad  \times \left(\dashint_{Q_{R,S}}u_+^{(\alpha+2)\frac{p-1+\eps}{\eps}}\dxt\right)^{\frac{\eps}{p-1+\eps}} \notag\\[4pt]
&\leq C\left[\frac{R^{(1-s)p}}{(R-\rho)^p} +\frac{1}{S-\tau}\right] \frac{|Q_{R,S}|}{|B_{\frac{R+\rho}{2}}|}\dashint_{Q_{R,S}} u_+^{\alpha+p+1}\dxt \notag\\[4pt]
& \quad \quad \quad +\frac{C R^{d}}{(R-\rho)^{d+sp}} \frac{|Q_{R,S}|^{\frac{\eps}{p-1+\eps}}}{|B_{\frac{R+\rho}{2}}|}\left(\int_{-S}^0 \Big[\tail (u_+(t)\,; B_{R} )\Big]^{p-1+\eps}\d{t}\right)^{\frac{p-1}{p-1+\eps}}  \notag\\[4pt]
&\quad \quad \quad \quad \quad \quad \quad \quad \quad  \times \left(\dashint_{Q_{R,S}}\left[ 1+u_+^{(\alpha+2)\frac{p-1+\eps}{\eps}} \right] \dxt\right)
\end{align}
with a positive constant $C(\data)<\infty$, where in the last line we used
\[
\left(\dashint_{Q_{R,S}}u_+^{(\alpha+2)\frac{p-1+\eps}{\eps}}\dxt\right)^{\frac{\eps}{p-1+\eps}} \leq \dashint_{Q_{R,S}}\left[ 1+u_+^{(\alpha+2)\frac{p-1+\eps}{\eps}} \right] \dxt.
\]
For $\eps >1$, we select $\alpha \in (-1,\infty)$ which satisfies
\[
\alpha+p+1=(\alpha+2)\frac{p-1+\eps}{\eps} \quad \iff \quad \alpha=\eps-2.
\]
Inserting this into~\eqref{eq:quali-bnd-2} yields
\begin{align}\label{eq:quali-bnd-3}
 \frac{1}{(p-1+\eps)^{p+1}}&\sup_{t \in (-S, \,0)} \dashint_{B_{\frac{R+\rho}{2}}}u_+^{p-1+\eps}(t)\d{x}  \notag \\[4pt]
&\leq C\left\{\left[\frac{R^{(1-s)p}}{(R-\rho)^p} +\frac{1}{S-\tau}\right] \frac{|Q_{R,S}|}{|B_{\frac{R+\rho}{2}}|}+\frac{R^{d}}{(R-\rho)^{d+sp}} \frac{|Q_{R,S}|^{\frac{\eps}{p-1+\eps}}}{|B_{\frac{R+\rho}{2}}|}\mathbf{T}_{R,S}\right\}  \notag\\[4pt]
&\quad \quad \quad \quad \quad \quad \quad \quad \quad  \times \left(\dashint_{Q_{R,S}} \left[1+u_+^{p-1+\eps} \right]\dxt\right),
\end{align}
where we used the shot-hand notation 
\[
\mathbf{T}_{R,S}:=\left(\int_{-S}^0 \Big[\tail (u_+(t)\,; B_{R} )\Big]^{p-1+\eps}\d{t}\right)^{\frac{p-1}{p-1+\eps}}\,.
\]

\emph{Step 2.}\quad Having at hand~\eqref{eq:quali-bnd-3}, we derive a reverse H\"{o}lder type inequality. To lighten the notation, we denote $w:=u_+^{\nicefrac{(p-1+\eps)}{p}}$ and 
\[
\widetilde{Q}:=Q_{\frac{R+\rho}{2}, \frac{S+\tau}{2}}=\widetilde{B}\times \widetilde{T}, \quad \widetilde{B}:=B_{\frac{R+\rho}{2}}, \quad \widetilde{T}:=\left(-\tfrac{S+\tau}{2}, 0\right).
\]
Let $\varphi$ be a cutoff function in $Q_{R,S}$ satisfying
\[
\1_{Q_{\rho,\tau}} \leq \varphi \leq \1_{\widetilde{Q}}, \quad |\nabla \varphi| \lesssim 1/(R-\rho), \quad \supp \varphi \subset \widetilde{Q}.
\]
Upon appealing to Proposition~\ref{FS} with $\mathsf{d}=\nicefrac{(R-\rho)}{2R} \in (0,1)$ we obtain that, for $\kappa=1+\frac{\kappa_\ast-1}{\kappa_\ast}$, 
\begin{align}\label{eq:quali-bnd-4}
\dashint_{Q_{\rho, \tau}}&w^{\kappa p}\dxt \leq \frac{1}{|Q_{\rho, \tau}|}\dashint_{\widetilde{Q}}(\varphi w)^{\kappa p}\dxt \notag\\[4pt]
&\leq \frac{C}{|Q_{\rho, \tau}|}\left[ \left(\frac{R+\rho}{2}\right)^{sp} \int_{\widetilde{T}} \iint \nolimits_{\widetilde{B} \times \widetilde{B}} \frac{|\varphi w(x,t)-\varphi w(y,t)|^p}{|x-y|^{d+sp}}\dxyt \right.\notag\\
&\quad \quad \quad \quad \left.+\left(\frac{2R}{R-\rho}\right)^{d+sp}\int_{\widetilde{Q}}(\varphi w)^p\dxt\right] \left(\sup_{t \in \widetilde{T}} \dashint_{\widetilde{B}}(\varphi w)^p(t)\d{x} \right)^{\frac{\kappa_\ast-1}{\kappa_\ast}}\notag\\[4pt]
&\leq C\frac{|Q_{R,S}|}{|Q_{\rho, \tau}|}\left[ \left(\frac{R+\rho}{2}\right)^{sp}\dashint_{Q_{R,S}}w^p(x,t) \left(\int_{\widetilde{B}} \frac{\d{y}}{|x-y|^{d+sp}}\right)\dxt \right.\notag\\
&\quad \quad \quad \quad \left.+\left(\frac{2R}{R-\rho}\right)^{d+sp}\int_{\widetilde{Q}}(\varphi w)^p\dxt\right] \left(\sup_{t \in (-S,0)} \dashint_{\widetilde{B}}w^p(t)\d{x} \right)^{\frac{\kappa_\ast-1}{\kappa_\ast}} \notag\\[4pt]
&\leq C\frac{|Q_{R,S}|}{|Q_{\rho, \tau}|} \left(\frac{2R}{R-\rho}\right)^{d+sp} \left[\dashint_{Q_{R,S}} w^p\dxt+\left(\sup_{t \in (-S,0)} \dashint_{B_{\frac{R+\rho}{2}}}w^p(t)\d{x} \right)\right]^\kappa,
\end{align}
where, in the penultimate line, we used
\[
\int_{\widetilde{B}} \frac{\d{y}}{|x-y|^{d+sp}} \leq \int_{B_{R+\rho}(x)} \frac{\d{y}}{|x-y|^{d+sp}} \lesssim_{d,s,p}\left(\frac{R+\rho}{2}\right)^{-sp}, \quad \forall x \in \widetilde{B}\equiv B_{\frac{R+\rho}{2}}.
\]
Merging this last estimate~\eqref{eq:quali-bnd-4} with~\eqref{eq:quali-bnd-3} concludes that
\begin{align}\label{eq:quali-bnd-5}
\dashint_{Q_{\rho, \tau}}&w^{\kappa p}\dxt \leq C(p-1+\eps)^{(p+1)\kappa} \frac{|Q_{R,S}|}{|Q_{\rho, \tau}|} \left(\frac{2R}{R-\rho}\right)^{d+sp} \notag \\[4pt]
&\times \left\{\left[\frac{R^{(1-s)p}}{(R-\rho)^p} +\frac{1}{S-\tau}\right] \frac{|Q_{R,S}|}{|B_{\frac{R+\rho}{2}}|}+\frac{R^{d}}{(R-\rho)^{d+sp}} \frac{|Q_{R,S}|^{\frac{\eps}{p-1+\eps}}}{|B_{\frac{R+\rho}{2}}|}\mathbf{T}_{R,S}\right\}^\kappa  \notag\\[4pt]
&\quad \quad \quad \quad \quad \quad \quad \quad \quad  \times \left(\dashint_{Q_{R,S}} \left[1+w^{p} \right]\dxt\right)^\kappa.
\end{align}

\emph{Step 3.}\quad We perform a Moser-type iteration using~\eqref{eq:quali-bnd-5}. For $0<\sigma^\prime <\sigma \leq 1$ and $i \in \N_0$, let  
\[
\begin{cases}
\rho_i:=\sigma^\prime R +\dfrac{\sigma-\sigma^\prime}{2^{i}}R, \quad \tau_i:=\sigma^\prime S +\dfrac{\sigma-\sigma^\prime}{2^{i}}S,\\[8pt]
\widetilde{\rho}_i:=\dfrac{1}{2}(\rho_i+\rho_{i+1}), \quad \widetilde{\tau}_i:=\dfrac{1}{2}(\tau_i+\tau_{i+1}),\\[4pt]
B_i:=B_{\rho_i}, \quad T_i:=(-\tau_i,0], \quad Q_i:=B_i \times T_i, \\[4pt]
\widetilde{B}_i:=B_{\widetilde{\rho}_i}, \quad \widetilde{T}_i:=(-\widetilde{\tau}_i,0], \quad \widetilde{Q}_i:=\widetilde{B}_i \times \widetilde{T}_i.
\end{cases}
\]
Observe, by construction, that
\[
\sigma Q_{R, S}=Q_0 \supset \cdots \supset Q_i \supset Q_{i+1} \supset \cdots \supset Q_\infty=\sigma^\prime Q_{R, S}, \quad \forall i \in \N_0.
\]
Analogously, we define $\{\chi_i\}_{i \in \N_0}$ via $\chi_i:=\kappa^i p$ and then, recursively, for $i \in \N_0$,
\[
W_i:=\left(\dashint_{Q_i}\left[1+w^{\chi_i}\right]\dxt \right)^{\nicefrac{1}{\chi_i}} \quad \mbox{and} \quad \mathbf{T}_i:=\mathbf{T}_{\rho_i, \tau_i}.
\]
Note that $\chi_i \to \infty$ as $i \to \infty$. It is straightforward to check that
\begin{align*}
\frac{|Q_{i}|}{|Q_{i+1}|} \left(\frac{2\rho_i}{\rho_i-\rho_{i+1}}\right)^{d+sp} &\leq 2^{d+1} \left(\frac{2^{i+2}}{\sigma-\sigma^\prime}\right)^{d+sp},\\[4pt]
\left[\frac{\rho_i^{(1-s)p}}{(\rho_i-\rho_{i+1})^p} +\frac{1}{\tau_i-\tau_{i+1}}\right] \frac{|Q_{i}|}{|\widetilde{B}_{i}|} &\leq 2^d\left(\frac{2^{i+1}}{\sigma-\sigma^\prime}\right)^{p}\left(\frac{S}{R^{sp}}+1 \right),\\[4pt]
\frac{\rho_i^{d}}{(\rho_i-\rho_{i+1})^{d+sp}} \frac{|Q_{i}|^{\frac{\eps}{p-1+\eps}}}{|\widetilde{B}_{i}|} \mathbf{T}_i&\leq C(d,p,\eps) \frac{4^{i(d+sp)}}{(\sigma-\sigma^\prime)^d}\frac{S}{R^{sp+\frac{d(p-1)}{p-1+\eps}}} \\
& \quad \quad \times \underbrace{\left(\dashint_{-\sigma S}^0 \Big[\tail (u_+(t)\,; B_{\sigma^\prime R} )\Big]^{p-1+\eps}\d{t}\right)^{\frac{p-1}{p-1+\eps}}}_{=:\overline{\mathbf{T}}}.
\end{align*}
Thus, appealing to~\eqref{eq:quali-bnd-5} over $Q_{i+1}$ and $Q_i$ and rearranging give
\begin{align*}
\dashint_{Q_{i+1}}w^{\chi_{i+1}}\dxt &\leq C(p-1+\eps)^{(p+1)\kappa} \frac{4^{i\left[(d+sp)+(d+p)\kappa \right]}}{(\sigma-\sigma^\prime)^{d+p}}  \\[4pt]
&\quad \times \left[\left(\frac{S}{R^{sp}}+1 \right)+\frac{S}{R^{sp+\frac{d(p-1)}{p-1+\eps}}} \overline{\mathbf{T}}\right]^\kappa \left(\dashint_{Q_{i}} \left[1+w^{\chi_i} \right]\dxt\right)^\kappa
\end{align*}
and therefore, adding $1$ to the left-hand side implies that the recursive inequality
\begin{equation}\label{eq:quali-bnd-6}
W_{i+1}^{\chi_{i+1}} \leq \frac{C}{(\sigma-\sigma^\prime)^{d+p}}\Big[(p-1+\eps)^{p+1}\boldsymbol{b}^{i} \mathbf{K}W_i^{\chi_i}\Big]^\kappa,
\end{equation}
where to lighten the notation we set
\[
\boldsymbol{b}:=4^{(d+sp)+(d+p)}>1 \quad \mbox{and} \quad \mathbf{K}:=\left(\frac{S}{R^{sp}}+1 \right)+\frac{S}{R^{sp+\frac{d(p-1)}{p-1+\eps}}}  \overline{\mathbf{T}}>1.
\]
Consequently, we iterate~\eqref{eq:quali-bnd-6} for $i=0,1,\ldots, k$ to obtain that
\begin{align}\label{eq:quali-bnd-7}
\left(\dashint_{Q_k}w^{\chi_k}\dxt \right)^{\nicefrac{1}{\chi_k}} \leq W_k &\leq \left[\frac{C}{(\sigma-\sigma^\prime)^{d+p}}\right]^{\nicefrac{1}{\chi_k}}\Big[(p-1+\eps)^{p+1}\boldsymbol{b}^{k-1} \mathbf{K}W_{k-1}^{\chi_{k-1}}\Big]^{\nicefrac{\kappa}{\chi_k}} \notag\\
& \,\, \,\vdots \notag\\
&\leq \left[\frac{C}{(\sigma-\sigma^\prime)^{d+p}}\right]^{\mathsf{S}(k)}(p-1+\eps)^{(p+1)\kappa \mathsf{S}(k)}\boldsymbol{b}^{\mathsf{T}(k)} \mathbf{K}^{\kappa \mathsf{S}(k)}\left(W_0^{\chi_0}\right)^{\nicefrac{\kappa^k}{\chi_k}},
\end{align}
where
\[
\mathsf{S}(k):=\sum_{i=0}^{k-1}\frac{\kappa^i}{\chi_k} \quad \mbox{and}\quad \mathsf{T}(k):=\sum_{i=0}^{k-1}i \frac{\kappa^{k-i}}{\chi_k}.
\]
Observe that
\[
\lim_{k \to \infty}\mathsf{S}(k)=\frac{1}{p(\kappa-1)} \quad \mbox{and} \quad \lim_{k \to \infty}\mathsf{T}(k)=\frac{\kappa}{p(\kappa-1)^2}.
\]
Thus, sending $k \to \infty$ in~\eqref{eq:quali-bnd-7} yields
\[
\sup_{Q_\infty} w \leq \frac{C_{\mathrm{prod}}}{(\sigma-\sigma^\prime)^{\frac{d+p}{p(\kappa-1)}}}\mathbf{K}^{\frac{\kappa}{p(\kappa-1)}}W_0
\]
with $C_{\mathrm{prod}}:=C^{\frac{1}{p(\kappa-1)} }(p-1+\eps)^{\frac{(p+1)\kappa}{p(\kappa-1)}}\boldsymbol{b}^{\frac{\kappa}{p(\kappa-1)^2}}$, namely that, for $\eps \in (1,\infty)$, 
\begin{equation}\label{eq:quali-bnd-8}
\sup_{\sigma^\prime Q_{R, S}} u_+ \leq \frac{C_{\mathrm{prod}}^{\frac{p}{p-1+\eps}}}{(\sigma-\sigma^\prime)^{\frac{d+p}{(\kappa-1)(p-1+\eps)}}} \mathbf{K}^{\frac{\kappa}{(\kappa-1)(p-1+\eps)}} \left(\dashint_{\sigma Q_{R,S}}\left[1+u_+^{p-1+\eps} \right]\dxt \right)^{\nicefrac{1}{(p-1+\eps)}}.
\end{equation}

\emph{Step 4.}\quad In this final step, we verify the estimate~\eqref{eq:quali-bnd-8} holds true for all $\eps \in (0,\infty)$. Let $\sigma_1,\sigma_2$ be such that $0<\sigma ^\prime \leq \sigma_1 <\sigma_2 \leq \sigma  <1$. Fixing $\eps \in (1,\infty)$ which satisfies~\eqref{eq:quali-bnd-8}, take $\eps^\prime \in (0,1]$ arbitrarily. By~\eqref{eq:quali-bnd-8} and Young's inequality we estimate that
\begin{align*}
\sup_{\sigma_1 Q_{R,S}} u_+ &\stackrel{\eqref{eq:quali-bnd-8}}{\leq} \frac{C_{\mathrm{prod}}^{\frac{p}{p-1+\eps}}}{(\sigma_2-\sigma_1)^{\frac{d+p}{(\kappa-1)(p-1+\eps)}}} \mathbf{K}^{\frac{\kappa}{(\kappa-1)(p-1+\eps)}} \left(\dashint_{\sigma_2 Q_{R,S}}\left[1+u_+ \right]^{p-1+\eps}\dxt \right)^{\nicefrac{1}{(p-1+\eps)}}\\
&\leq  
\left(\frac{C_{\mathrm{prod}}^p\mathbf{K}^{\frac{\kappa}{\kappa-1}}}{(\sigma_2-\sigma_1)^{\frac{d+p}{\kappa-1}}}\dashint_{\sigma_2 Q_{R,S}}\left[1+u_+ \right]^{p-1+\eps^\prime}\dxt \right)^{\nicefrac{1}{(p-1+\eps)}} \left(1+\sup_{\sigma_2 Q_{R,S}}u_+\right)^{\nicefrac{(\eps-\eps^\prime)}{(p-1+\eps)}} \\
&\leq \frac{1}{2}\sup_{\sigma_2 Q_{R,S}}u_+ + C\left(\frac{C_{\mathrm{prod}}^p\mathbf{K}^{\frac{\kappa}{\kappa-1}}}{(\sigma_2-\sigma_1)^{\frac{d+p}{\kappa-1}}}\dashint_{\sigma_2 Q_{R,S}}\left[1+u_+ \right]^{p-1+\eps^\prime}\dxt \right)^{\nicefrac{1}{(p-1+\eps^\prime)}}.
\end{align*}
The desired estimate~\eqref{eq:quali-bnd-8} for $\eps \in (0,1]$ follows from using a standard iteration argument (see for example~\cite[Lemma 6.1, Page 191]{Giusti}). The proof is complete.
\end{proof}

\begin{remark}
In the limit case $u \in \mathbf{PDG}_\pm^{s,p}(\Omega_T, \gamma_{\mathrm{DG}}, \infty)$ we can also deduce the qualitative local boundedness, namely that, there exists a constant $C(d,s,p,\gamma_{\mathrm{DG}})<\infty$ such that on all concentric cylinder $\sigma^\prime Q_{R,S}(z_0) \subset \sigma Q_{R,S}(z_0) \subset \Omega_T$ for $z_0=(x_0,t_0) \in \Omega_T$ and $0<\sigma^\prime < \sigma \leq 1$, we have
\begin{align*}
\sup_{\sigma^\prime Q_{R,S}(z_0)} u_\pm &\leq \frac{C}{(\sigma-\sigma^\prime)^{\nicefrac{(d+p)}{(\kappa-1)}}}  \left[\left(\frac{S}{R^{sp}}+1 \right)+\frac{S}{R^{sp}}\tail_\infty\right]^{\nicefrac{\kappa}{(\kappa-1)}} \left(\dashint_{\sigma Q_{R,S}(z_0)}\left[1+u_\pm^{p} \right]\dxt \right)^{\nicefrac{1}{p}},
\end{align*}
where  $\sigma Q_{R,S}(z_0):=Q_{\sigma R, \sigma S}(z_0)$ and
\[
\tail_\infty:=\sup_{t \in (t_0-\sigma S, t_0)} \tail \left(u_\pm(t)\,; B_{\sigma^\prime R}(x_0)\right).
\]
It is rather easy to check this assertion for $u \in \mathbf{PDG}_\pm^{s,p}(\Omega_T, \gamma_{\mathrm{DG}}, \infty)$. Indeed, in the proof of Proposition~\ref{Prop:qualitative bnd}, it is enough to skip using H\"{o}lder's inequality and slightly modify Moser's iteration argument.
\end{remark}

\subsection{Proof of Theorem~\ref{Thm:boundedness}}
We will prove~Theorem~\ref{Thm:boundedness} using a new time-dependent truncation method based on the linear case as in~\cite[Proposition 5.1]{Lia25}. The reader can also consult the argument of~\cite{KW23a} for the linear case.  

\begin{proof}[Proof of Theorem~\ref{Thm:boundedness}]
The proof is technically intricate and thus, for the sake of readability, we split it into four steps.  The initial three steps are meant to establish the (quantitative) $L^\infty$--$L^p$ local boundedness of $u \in \mathbf{PDG}_+^{s,p}(\Omega_T, \gamma_{\mathrm{DG}}, \eps)$, while the fourth is meant to refine it to a $L^\infty$--$L^\nu$ bound for any $ \nu \in (0,p)$.

\smallskip

\emph{Step 1.} \quad Let $z_0=(x_0,t_0)=(0,0)$, suppose $Q_{R, \theta} \subset \Omega_T$ and take $\rho \in (0,R)$. Fix $\sigma \in (0,1)$ and let $k \in [0,\infty)$ be a parameter to be specified later. We then introduce, for $i\in \N_0$, the following 
\[
\begin{cases}
k_i:=(1-2^{-i})k,\\[4pt]
\rho_i:=\sigma \rho +2^{-i}(1-\sigma)\rho, \quad \theta_i:=\sigma \theta+2^{-i}(1-\sigma)\theta,\\[4pt]
\widetilde{\rho}_i:=\frac{1}{2}(\rho_i+\rho_{i+1}), \quad \widehat{\rho}_i:=\frac{3}{4}\rho_i+\frac{1}{4}\rho_{i+1}, \quad \overline{\rho}_i:=\frac{1}{4}\rho_i+\frac{3}{4}\rho_{i+1}, \\[4pt]
\widetilde{\theta}_i:=\frac{1}{2}(\theta_i+\theta_{i+1}), \quad \,\widehat{\theta}_i:=\frac{3}{4}\theta_i+\frac{1}{4}\theta_{i+1}, \quad \,\overline{\theta}_i:=\frac{1}{4}\theta_i+\frac{3}{4}\theta_{i+1}, \\[4pt]
B_i:=B_{\rho_i}, \quad \widetilde{B}_i:=B_{\widetilde{\rho}_i}, \quad \widehat{B}_i:=B_{\widehat{\rho}_i}, \quad \overline{B}_i:=B_{\overline{\rho}_i}, \\[4pt]
Q_i:=B_i \times (-\theta_i,0], \quad \widetilde{Q}_i:=\widetilde{B}_i \times (-\widetilde{\theta}_i,0], \\[4pt]
\widehat{Q}_i:=\widehat{B}_i \times (-\widehat{\theta}_i, 0], \quad \overline{Q}_i:=\overline{B}_i \times (-\overline{\theta}_i, 0]. 
\end{cases}
\]
By construction it is clear that, for any $i \in \N_0$, 
\[
\left\{
\begin{array}{l}
0=k_0\leq k_i \leq k_{i+1} \uparrow k_\infty=k,  \\[2mm]
Q_0=Q_{\rho, \theta}, \quad Q_\infty=Q_{\sigma \rho, \sigma \theta}, \quad Q_{i+1} \subset \overline{Q}_i \subset \widetilde{Q}_i \subset \widehat{Q}_i \subset Q_i.
\end{array}
\right.
\]
We then select the cutoff function $\zeta \in C_c^\infty(Q_i ; [0,1])$ which vanishes outside $\widehat{Q}_i$ and satisfies
$\zeta\equiv 1$ on $\widetilde{Q}_i$, and
\[
|\nabla \zeta| \leq \frac{2^{i+4}}{(1-\sigma)\rho} \quad \mbox{and} \quad |\partial_t\zeta| \leq \frac{2^{i+4}}{(1-\sigma)\theta}.
\]
Set
\[
\ell(t):=\left(\widetilde{\gamma}\int_{-\theta}^t \int_{\R^d \setminus B_R}\frac{u_+^{p-1}(y,\tau)}{|y|^{d+sp}}\dytau\right)^{\nicefrac{1}{(p-1)}},
\]
where the constant $\widetilde{\gamma}>0$ is to be determined later. Set $M_0:=\sup_{Q_0} u_+$, which is finite thanks to Proposition~\ref{Prop:qualitative bnd}. In the energy formulation~\eqref{d2}, we choose now the time-dependent level $k(t)$ as
\[
k(t)=\left(\ell(t)^{p-1}+k_{i+1}^{p-1}\right)^{\nicefrac{1}{(p-1)}}+\frac{1}{2}M_0
\]
and adopt the shorthand notation
\[
w_+(x,t):= \left[u(x,t)-\left(\ell(t)^{p-1}+k_{i+1}^{p-1}\right)^{\nicefrac{1}{(p-1)}}-\frac{1}{2}M_0\right]_+.
\]
It is clear that $k(t) \geq \frac{1}{2}M_0$ and $k^\prime(t) \geq 0$. Under this choice, implementing~\eqref{d2} over concentric cylinders $\widetilde{Q}_i$ and $Q_i$ gives,
\begin{align}\label{eq:bnd-est-1}
&\sup_{t \in (-\widetilde{\theta}_i,0]}\int_{\widetilde{B}_i}w_+^p\d{x}+\int_{-\widetilde{\theta}_i}^0 \int_{\widetilde{B}_i}\int_{\widetilde{B}_i} \frac{|w_+(x,t)-w_+(y,t)|^p}{|x-y|^{d+sp}}\dxyt 
\notag\\[4pt]
& \quad \quad \leq \gamma_{\mathrm{DG}}\int_{Q_i}\left|\partial_t \zeta^p\right|w_+^p\dxt \notag\\[4pt]
&\quad \quad \quad +\gamma_{\mathrm{DG}}\int_{-\theta_i}^{0}\int_{B_i}\int_{B_i}\min\left\{w_+^p(x,t), w_+^p(y,t)\right\}\dfrac{\big|\zeta(x,t)-\zeta(y,t)\big|^p}{|x-y|^{d+sp}}\dxyt \notag\\[4pt]
&\quad \quad \quad +\gamma_{\mathrm{DG}}\int_{Q_i}\zeta^pw_+(x,t)\d{x} \left(\sup_{x\,\in\, \widehat{B}_i}\int_{\R^d \setminus B_i}\frac{w_+^{p-1}(y,t)}{|x-y|^{d+sp}}\d{y}\right)\d{t} \notag\\[4pt]
&\quad \quad  \quad -\gamma_{\mathrm{DG}}\int_{Q_i}\zeta^pk^\prime(t)\Big(|u|+|k(t)|\Big)^{p-2}w_+(x,t)\dxt\notag\\[4pt]
&\quad \quad =:  (\mathrm{I})+(\mathrm{II})+(\mathrm{III})+(\mathrm{IV})
\end{align}
with obvious meanings of $(\mathrm{I})$--$(\mathrm{IV})$. Here, on the left side, we have discarded the term including the integral over whole $\R^d$.

\smallskip

\emph{Step 2.} \quad With~\eqref{eq:bnd-est-1} at hand, we estimate four terms separately. As for $(\mathrm{I})$, a straightforward computation shows that
\[
(\mathrm{I}) \leq \frac{2^{i+4}}{(1-\sigma)\theta} \int_{Q_i}w_+^p\dxt.
\]
An estimation of $(\mathrm{II})$ is also standard. In fact, we can deduce that
\begin{align*}
(\mathrm{II}) &\leq \frac{2^{(i+4)p}}{(1-\sigma)^p\rho^p}\int_{-\theta_i}^{0}\int_{B_i}\int_{B_i}\dfrac{w_+^p(x,t)}{|x-y|^{d+p(s-1)}}\dxyt\\
&\leq \frac{c2^{ip}}{(1-\sigma)^p\rho^{sp}}\int_{Q_i}w_+^p\dxt
\end{align*}
for a constant $c\equiv c(d,s,p)<\infty$, where to obtain the last line we used fact that
\[
\int_{B_i}\frac{\d{y}}{|x-y|^{d+p(s-1)}} \leq \int_{B_{2\rho_i}(x)}\frac{\d{y}}{|x-y|^{d+p(s-1)}} \leq c(d,s,p)\rho^{(1-s)p}, \quad \forall x \in B_i.
\]
The third procedure is started by splitting $(\mathrm{III})$ into two integrals:
\begin{align*}
(\mathrm{III})_1&:=\gamma_{\mathrm{DG}}\int_{Q_i}\zeta^pw_+(x,t)\d{x} \left(\sup_{x\,\in\, \widehat{B}_i}\int_{B_R \setminus B_i}\frac{w_+^{p-1}(y,t)}{|x-y|^{d+sp}}\d{y}\right)\d{t},\\[4pt]
(\mathrm{III})_2&:=\gamma_{\mathrm{DG}}\int_{Q_i}\zeta^pw_+(x,t)\d{x} \left(\sup_{x\,\in\, \widehat{B}_i}\int_{\R^d \setminus B_R}\frac{w_+^{p-1}(y,t)}{|x-y|^{d+sp}}\d{y}\right)\d{t}.
\end{align*}
Observe that
\[
\frac{|y-x|}{|y|} \geq 1-\frac{\widehat{\rho}_i}{\rho_i}=\frac{\rho_i-\rho_{i+1}}{4\rho_i} \geq \frac{1-\sigma}{2^{i+3}}\quad \mbox{for} \quad |y| \geq \rho_i,\,|x| \leq \widehat{\rho}_i,
\]
whereas, if $|y| \geq R$ and $|x| \leq \widehat{\rho}_i$ then 
\[
\frac{|y-x|}{|y|}  \geq 1-\frac{\widehat{\rho}_i}{R} \geq 1-\frac{\rho}{R}
\]
holds. Thus, we can bound
\begin{align*}
(\mathrm{III}) &\leq (\mathrm{III})_1 +(\mathrm{III})_2\\
&\leq \gamma_{\mathrm{DG}} \int_{Q_i}\zeta^pw_+(x,t)\left[ \left(\frac{2^{i+2}}{1-\sigma}\right)^{d+sp} \int_{B_R \setminus B_i}\frac{w_+^{p-1}(y, t)}{|y|^{d+sp}}\d{y}\right]\dxt\\
& \quad \quad+\gamma_{\mathrm{DG}} \int_{Q_i}\zeta^pw_+(x,t) \left[ \left(\frac{R}{R-\rho}\right)^{d+sp} \int_{\R^d \setminus B_R}\frac{u_+^{p-1}(y, t)}{|y|^{d+sp}}\d{y}\right]\dxt,
\end{align*}
where, to obtain the last line, the second integral was estimated by $w_+\leq u_+$. 

Next, we turn our attention to $(\mathrm{IV})$. For this, we split the estimate of $(\mathrm{IV})$ into two cases: $p \geq 2$ and $p \in (1,2)$. In the first case $ p \geq 2$, by discarding the term $|u|$ and using the nonnegativity of $k^\prime(t)$, we estimate that
\begin{align*}
(\mathrm{IV}) &=-\gamma_{\mathrm{DG}}\int_{Q_i}\zeta^pk^\prime(t)\Big(|u|+|k(t)|\Big)^{p-2}w_+(x,t)\dxt\\
&\leq -\gamma_{\mathrm{DG}}\int_{Q_i}\zeta^pk^\prime(t)k(t)^{p-2}w_+(x,t)\dxt\\
&=-\frac{\gamma_{\mathrm{DG}}}{p-1}\int_{Q_i}\zeta^p \frac{\d{}}{\d{t}}k(t)^{p-1}\cdot w_+(x,t)\dxt.
\end{align*}
\noindent We next consider the case $p \in (1,2)$. Observe that, for $u >k(t)$, 
\[
|u|+|k(t)| \leq 2 M_{0} \quad \Longrightarrow \quad -\frac{1}{\Big(|u|+|k(t)|\Big)^{2-p}} \leq -\frac{1}{(2M_{0})^{2-p}}.
\]
This, together with the nonnegativity of $k^\prime(t)$ and $k(t) \geq M_0/2$, yields that
\begin{align*}
(\mathrm{IV}) 
&=-\frac{\gamma_{\mathrm{DG}}}{p-1}\int_{Q_i}\zeta^p \frac{\d{}}{\d{t}}k(t)^{p-1} \cdot \frac{k(t)^{2-p}}{\Big(|u|+|k(t)|\Big)^{2-p}}w_+(x,t)\dxt \\
&\leq -\frac{\gamma_{\mathrm{DG}}}{p-1}\int_{Q_i}\zeta^p \frac{\d{}}{\d{t}}k(t)^{p-1} \cdot \frac{(M_0/2)^{2-p}}{\big(2M_{0}\big)^{2-p}}w_+(x,t)\dxt\\
&=-\frac{\gamma_{\mathrm{DG}}}{4^{2-p}(p-1)}\int_{Q_i}\zeta^p \frac{\d{}}{\d{t}}k(t)^{p-1} \cdot w_+(x,t)\dxt.
\end{align*}
Consequently, we obtain, for $p \in (1,\infty)$, that
\[
(\mathrm{IV}) \leq -\frac{\gamma_{\mathrm{DG}}}{4^{(2-p)_+}(p-1)}\int_{Q_i}\zeta^p \frac{\d{}}{\d{t}}k(t)^{p-1}\cdot w_+(x,t)\dxt=:\mathbf{T}.
\]
The idea now is to make $(\mathrm{III})_2+\mathbf{T}$ vanish. For this, we select the free parameter $\widetilde{\gamma}$ appearing in the definition of $\ell(t)$, so that
\[
\widetilde{\gamma}=4^{(2-p)_+}(p-1)\left(\frac{R}{R-\rho}\right)^{d+sp}.
\]
This immediately implies that
\[
\frac{\d{}}{\d{t}}k(t)^{p-1} =4^{(2-p)_+}(p-1)\left(\frac{R}{R-\rho}\right)^{d+sp} \int_{\R^d \setminus B_R}\frac{u_+^{p-1}(y,t)}{|y|^{d+sp}}\d{y}.
\]
As a consequence, we obtain
\begin{align}\label{eq:bnd-est-2}
(\mathrm{III})+(\mathrm{IV})&\leq (\mathrm{III})_1+(\mathrm{III})_2+\mathbf{T}\notag\\
&\leq \gamma_{\mathrm{DG}} \frac{4^{(i+2)(d+sp)}C(d)}{[\sigma (1-\sigma)]^{d+sp}\rho^{sp}} \left(\frac{R}{\rho}\right)^d \left[\sup_{t \in (-\theta,0]}\dashint_{B_R}w_+^{p-1}(y,t)\d{y}\right] \notag\\
&\quad \quad \quad\quad \quad \quad\quad \quad\times \int_{Q_i} w_+(x,t)\dxt \notag\\
&\leq \gamma_{\mathrm{DG}} \frac{4^{(i+2)(d+sp)}C(d)}{[\sigma (1-\sigma)]^{d+sp}\rho^{sp}} \boldsymbol{\mathfrak{S}}^{p-1}\int_{Q_i} w_+(x,t)\dxt\,,
\end{align}
where
\[
\displaystyle \boldsymbol{\mathfrak{S}}:=\left[\left(\frac{R}{\rho}\right)^d \sup_{t \in (-\theta,0]}\dashint_{B_R}u_+^{p-1}(y,t)\d{y}\right]^{\nicefrac{1}{(p-1)}}.
\]
At this stage, we further estimate the last integral in the above display. To this end, let us denote, for $i \in \N_0$, 
\[
\widetilde{w}_+(x,t):=\left[u(x,t)-\left(\ell(t)^{p-1}+k_i^{p-1}\right)^{\nicefrac{1}{(p-1)}}-\frac{1}{2}M_0\right]_+
\]
and
\[
A_i:=\left\{u(x,t)-\frac{1}{2}M_0>\left(\ell(t)^{p-1}+k_{i+1}^{p-1}\right)^{\nicefrac{1}{(p-1)}} \right\} \cap Q_i.
\]
The following Chebyshev type estimate\footnote{
Using the algebraic inequalities as in Lemma~\ref{Lm:alg-est-2} and~\cite[Lemma 2.2]{BDLMBS25}, one can verify~\eqref{eq:bnd-est-3}. Indeed, observe that \begin{align*}\int_{Q_i}\widetilde{w}_+^p(x,t)\dxt &\geq \int_{A_i}\left[u(x,t)-\left(\ell(t)^{p-1}+k_i^{p-1}\right)^{\nicefrac{1}{(p-1)}}-\frac{1}{2}M_0\right]_+^p \dxt \notag\\
&\geq \int_{A_i}\underbrace{\left[\left(\ell(t)^{p-1}+k_{i+1}^{p-1}\right)^{\nicefrac{1}{(p-1)}}-\left(\ell(t)^{p-1}+k_{i}^{p-1}\right)^{\nicefrac{1}{(p-1)}}\right]^p}_{=:\mathsf{F}_p}\dxt,
\end{align*}
Under the constraint~\eqref{eq:k-condition-1}, one can estimate that
$
\mathsf{F}_p \geq \frac{1}{c(p)}\frac{k^p}{2^{(i+2)p(2 \vee p)}}
$
with $c(p)=1$ if $1<p<2$ and $c(p)=(p-1)^p2^{\frac{p+1}{p-1}(p-2)p}$ if $p \geq 2$. We omit the details.
}
\begin{equation}\label{eq:bnd-est-3}
|A_i| \leq c(p)\frac{2^{(i+2)p(2 \vee p)}}{k^p}\int_{Q_i}\widetilde{w}_+^p(x,t)\dxt
\end{equation}
holds true for some constant $c(p)<\infty$, provided that
\begin{equation}\label{eq:k-condition-1}
k \geq \left[C_\ast\left(\frac{R}{R-\rho}\right)^{d+sp} \int_{-\theta}^0 \int_{\R^d \setminus B_R}\frac{u_+^{p-1}(y,\tau)}{|y|^{d+sp}}\dytau \right]^{\nicefrac{1}{(p-1)}}=\ell(0)
\end{equation}
with $C_\ast:=4^{(2-p)_+}(p-1)$. By~\eqref{eq:bnd-est-3}, the fact that $w_+ \leq \widetilde{w}_+$ and H\"{o}lder's inequality, it holds 
\begin{align*}
\int_{Q_i}w_+(x,t)\dxt &\leq |A_i|^{1-\nicefrac{1}{p}}\left(\int_{Q_i}\widetilde{w}_+^p(x,t)\dxt\right)^{\nicefrac{1}{p}}\\
&\!\!\stackrel{\eqref{eq:bnd-est-3}}{\leq} C\frac{2^{(i+2)(p-1)(2 \vee p)}}{k^{p-1}} \int_{Q_i}\widetilde{w}_+^p\dxt.
\end{align*}
Matching this with~\eqref{eq:bnd-est-2}, yields
\[
(\mathrm{III})+(\mathrm{IV}) \leq C 
\frac{4^{(i+2)(d+sp+(p-1)(2 \vee p))}}{[\sigma (1-\sigma)]^{d+sp}\rho^{sp}} \left(\frac{\boldsymbol{\mathfrak{S}}}{k}\right)^{p-1}\int_{Q_i}\widetilde{w}_+^p\dxt,
\]
and therefore we have obtained the estimate
\begin{align}\label{eq:bnd-est-7}
&\sup_{t \in (-\widetilde{\theta}_i,0]}\int_{\widetilde{B}_i}w_+^p\d{x}+\int_{-\widetilde{\theta}_i}^0 \int_{\widetilde{B}_i}\int_{\widetilde{B}_i} \frac{|w_+(x,t)-w_+(y,t)|^p}{|x-y|^{d+sp}}\dxyt \notag\\
& \quad \quad \leq (\mathrm{I})+(\mathrm{II})+(\mathrm{III})+(\mathrm{IV})  \notag\\
& \quad \quad \leq C \frac{\boldsymbol{b}^i}{(1-\sigma)^{d+p}} \left[\frac{1}{\rho^{sp}}+\frac{1}{\theta}+\frac{1}{\sigma^{d+sp} \rho^{sp}}\left(\frac{\boldsymbol{\mathfrak{S}}}{k}\right)^{p-1}\right]\int_{Q_i}\widetilde{w}_+^p\dxt \notag\\
& \quad \quad \leq C \frac{\boldsymbol{b}^i}{(1-\sigma)^{d+p}\rho^{sp}} \left[\frac{\rho^{sp}}{\theta}+\frac{1}{\delta^{p-1}\sigma^{d+sp}}\right]\int_{Q_i}\widetilde{w}_+^p\dxt
\end{align}
for a constant $C(d,s,p,\gamma_{\mathrm{DG}})<\infty$ and $\boldsymbol{b}:=4^{d+sp+(p-1)(2 \vee p)}$. To obtain the last line we have enforced 
\begin{equation}\label{eq:k-condition-2}
k \geq \delta \boldsymbol{\mathfrak{S}},
\end{equation}
where $\delta \in (0,1]$ is to be specified later.

\smallskip

\emph{Step 3: Iteration.} Let $\varphi$ be a cutoff function defined in $Q_i$ such that
\[
\1_{Q_{i+1}} \leq \varphi \leq \1_{\widetilde{Q}_i}, \quad |\nabla \varphi| \leq 2^{i+4}/\rho, \quad \mathrm{supp}\,(\varphi) \subset \widetilde{Q}_i.
\]
We then apply Proposition~\ref{FS} with $\mathsf{d}=2^{-i-4}$ and the H\"{o}lder inequality to get
\begin{align*}
&\int_{Q_{i+1}}w_+^p\dxt  \leq \int_{\widetilde{Q}_i}(\varphi w_+)^p\dxt
 \\
& \quad \leq \left[\int_{\widetilde{Q}_i}(\varphi w_+)^{\kappa p}\dxt\right]^{\nicefrac{1}{\kappa}}|A_i|^{1-\nicefrac{1}{\kappa}} \\
& \quad \leq C \Bigg[\rho^{sp} \int_{-\widetilde{\theta}_i}^0\int_{\widetilde{B}_i}\int_{\widetilde{B}_i} \frac{\left|\varphi w_+(x,t)-\varphi w_+(y,t)\right|^p}{|x-y|^{d+sp}}\dxyt \\
&\quad \quad \quad \quad \quad +\frac{1}{\mathsf{d}^{d+sp}}\int_{\widetilde{Q}_i}(\varphi w_+)^p\dxt\Bigg]^{\nicefrac{1}{\kappa}}\left[\sup_{t \in (-\widetilde{\theta}_i,0]}\dashint_{\widetilde{B}_i}(\varphi w_+)^p\d{x}\right]^{\frac{\kappa_\ast-1}{\kappa \kappa_\ast}}|A_i|^{1-\nicefrac{1}{\kappa}},
\end{align*}
where $\kappa=1+\frac{\kappa_\ast-1}{\kappa_\ast}$ with $\kappa_\ast$ appearing in Proposition~\ref{FS}. Furthermore, as for the fractional integral on the right-hand side, by the triangle inequality, it holds that
\begin{align*}
\left|\varphi w_+(x,t)-\varphi w_+(y,t)\right|^p \lesssim_p \left|w_+(x,t)-w_+(y,t)\right|^p\varphi^p(x,t)+w_+^p(y,t)\left|\varphi (x,t)-\varphi(y,t)\right|^p,
\end{align*}
thereby getting 
\begin{align}\label{eq:bnd-est-7-add}
&\rho^{sp}\int_{-\widetilde{\theta}_i}^0\int_{\widetilde{B}_i}\int_{\widetilde{B}_i} \frac{\left|\varphi w_+(x,t)-\varphi w_+(y,t)\right|^p}{|x-y|^{d+sp}}\dxyt \notag \\
& \quad \lesssim_p  \quad \rho^{sp}\int_{-\widetilde{\theta}_i}^0\int_{\widetilde{B}_i}\int_{\widetilde{B}_i} \frac{\left|w_+(x,t)-w_+(y,t)\right|^p}{|x-y|^{d+sp}}\dxyt \notag \\
& \quad \quad \quad \quad + \rho^{sp} \left(\frac{2^{i+4}}{\rho}\right)^p\int_{-\widetilde{\theta}_i}^0\int_{\widetilde{B}_i}w_+^p(y,t)\d{y}\left(\int_{\widetilde{B}_i}\frac{\d{x}}{|x-y|^{d+sp}}\right)\d{t} \notag \\
& \quad \lesssim_{d,s,p} \quad \rho^{sp} \int_{-\widetilde{\theta}_i}^0\int_{\widetilde{B}_i}\int_{\widetilde{B}_i} \frac{\left|w_+(x,t)-w_+(y,t)\right|^p}{|x-y|^{d+sp}}\dxyt \notag \\
& \quad \quad \quad  \quad \quad +2^{(i+4)p}\int_{\widetilde{Q}_i}\widetilde{w}_+^p(y,t)\dyt \notag \\
& \quad \!\!\!\stackrel{\eqref{eq:bnd-est-7}}{\leq} C \frac{\boldsymbol{b}^i}{(1-\sigma)^{d+p}} \left[\frac{\rho^{sp}}{\theta}+\frac{1}{\delta^{p-1}\sigma^{d+sp}}\right]\int_{Q_i}\widetilde{w}_+^p\dxt,
\end{align}
where $C\equiv C(\data)<\infty$; in the penultimate line we have used a simple estimate
\[
\int_{\widetilde{B}_i}\frac{\d{x}}{|x-y|^{d+sp}} \leq \int_{B_{2\widetilde{\rho}_i}(y)}\frac{\d{x}}{|x-y|^{d+sp}} \leq c(d,s,p)\rho^{(1-s)p} \quad \forall x \in \widetilde{B}_i.
\]
Combining the previous estimates with~\eqref{eq:bnd-est-7} and using the identity
$\frac{1}{\kappa}+\frac{\kappa_\ast-1}{\kappa\kappa_\ast}=\frac{2\kappa_\ast-1}{\kappa \kappa_\ast}
$, we obtain
\begin{align*}
&\int_{Q_{i+1}}w_+^p\dxt  \\
&\quad \quad \leq \frac{\boldsymbol{\gamma} \widetilde{\boldsymbol{b}}^i \left(\frac{\rho^{sp}}{\theta}+\frac{1}{\delta^{p-1}\sigma^{d+sp}}\right)^{\frac{2\kappa_\ast-1}{\kappa \kappa_\ast}}}{(1-\sigma)^{(d+p)\frac{2\kappa_\ast-1}{\kappa \kappa_\ast}}}\rho^{-(d+sp)\frac{\kappa_\ast-1}{\kappa \kappa_\ast}}k^{-p(1-\frac{1}{\kappa})} \left(\int_{Q_i} \widetilde{w}_+^p\dxt \right)^{1+\frac{\kappa_\ast-1}{\kappa \kappa_\ast}}
\end{align*}
with $\boldsymbol{\gamma} \equiv \boldsymbol{\gamma}(d,s,p,\gamma_{\mathrm{DG}})>1$ and $\widetilde{\boldsymbol{b}}:=2^{d\frac{\kappa_\ast-1}{\kappa \kappa_\ast}}\boldsymbol{b}>1$. 
Now, define
\[
\boldsymbol{Y}_{\!i}:=\int_{Q_i}\left[u(x,t)-\left(\ell(t)^{p-1}+k_i^{p-1}\right)^{\nicefrac{1}{(p-1)}}-\frac{1}{2}M_0\right]_+^p\dxt.
\]
Hence, the above recursive inequality gives, for every $i\in \N_0$, that
\[
\boldsymbol{Y}_{\!\! i+1} \leq \frac{\boldsymbol{\gamma} \widetilde{\boldsymbol{b}}^i \left(\frac{\rho^{sp}}{\theta}+\frac{1}{\delta^{p-1}\sigma^{d+sp}}\right)^{\frac{2\kappa_\ast-1}{\kappa \kappa_\ast}}}{(1-\sigma)^{(d+p)\frac{2\kappa_\ast-1}{\kappa \kappa_\ast}}}\cdot \frac{\boldsymbol{Y}_{\!\!i}^{1+\frac{\kappa_\ast-1}{\kappa \kappa_\ast}}}{(\rho^{d+sp}k^p)^{\frac{\kappa_\ast-1}{\kappa \kappa_\ast}}}.
\]
Thanks to the fast geometric convergence (Lemma~\ref{Lm:FGC}), we can conclude
\begin{equation}\label{eq:bnd-est-8}
\lim_{i \to \infty}\boldsymbol{Y}_{\!\!i}=\lim_{i \to \infty} \int_{Q_i}\left[u(x,t)-\left(\ell(t)^{p-1}+k_i^{p-1}\right)^{\nicefrac{1}{(p-1)}}-\frac{1}{2}M_0\right]_+^p\dxt=0
\end{equation}
provided that
\begin{align*}
\boldsymbol{Y}_{\!0}&=\int_{Q_0}\left(u-\ell(t)-\frac{1}{2}M_0\right)_+^p\dxt \\
&\leq \left[ \frac{\boldsymbol{\gamma} \left(\frac{\rho^{sp}}{\theta}+\frac{1}{\delta^{p-1}\sigma^{d+sp}}\right)^{\frac{2\kappa_\ast-1}{\kappa \kappa_\ast}}}{(1-\sigma)^{(d+p)\frac{2\kappa_\ast-1}{\kappa \kappa_\ast}}(\rho^{d+sp}k^p)^{\frac{\kappa_\ast-1}{\kappa \kappa_\ast}}}\right]^{-1/\left(\frac{\kappa_\ast-1}{\kappa\kappa_\ast}\right)}\widetilde{\boldsymbol{b}}^{-1/\left(\frac{\kappa_\ast-1}{\kappa\kappa_\ast}\right)^2},
\end{align*}
namely that, by denoting $q=\frac{2\kappa_\ast-1}{p(\kappa_\ast-1)}$,
\begin{align}\label{eq:k-condition-3}
 k &\geq \frac{\widetilde{\boldsymbol{\gamma}}}{(1-\sigma)^{(d+p)q}}\left(\frac{\rho^{sp}}{\theta}+\frac{1}{\delta^{p-1}\sigma^{d+sp}}\right)^q\left(\frac{\theta}{\rho^{sp}}\right)^{\nicefrac{1}{p}}\left[\dashint_{Q_{\rho,\theta}}\left(u-\ell(t)-\frac{1}{2}M_0\right)_+^p\dxt \right]^{\nicefrac{1}{p}}
\end{align}
with a new constant $\widetilde{\boldsymbol{\gamma}} \equiv \widetilde{\boldsymbol{\gamma}}(d,s,p,\gamma_{\mathrm{DG}})<\infty$; in the above estimate, the quantity $\widetilde{\boldsymbol{b}}$ was incorporated in $\widetilde{\boldsymbol{\gamma}}$ because $\widetilde{\boldsymbol{b}}>1$ depends upon $d,s$ and $p$, and to keep the notation short, we define
\[
\boldsymbol{\cA}_\delta \equiv \boldsymbol{\cA}_\delta\left(\sigma,\rho,R,\theta\right):=\left(\frac{\rho^{sp}}{\theta}+\frac{1}{\delta^{p-1}\sigma^{d+sp}}\right)^q\left(\frac{\theta}{\rho^{sp}}\right)^{\nicefrac{1}{p}}.
\]
Note that the function $\boldsymbol{\cA}_\delta$ depends on $\delta$ as well, but is independent of the variable $R$. Now, select $k$ such that
\begin{align*}
k&=\ell(0)+\delta \boldsymbol{\mathfrak{S}}+\frac{\widetilde{\boldsymbol{\gamma}}}{(1-\sigma)^{(d+p)q}}\boldsymbol{\cA}_\delta\left(\dashint_{Q_{\rho,\theta}}u_+^p\dxt \right)^{\nicefrac{1}{p}},
\end{align*}
which is in accordance with three preceding conditions~\eqref{eq:k-condition-1},~\eqref{eq:k-condition-2} and~\eqref{eq:k-condition-3}, Lemma~\ref{Lm:alg-est-2} with $\alpha=1/(p-1)$ then yields, using also that $\ell(t) \leq \ell (0)$ and~\eqref{eq:bnd-est-8}, 
\begin{align*}
\sup_{Q_{\sigma \rho, \sigma \theta}} u_+ &\leq \frac{1}{2}M_0+2^{\frac{(2-p)_+}{p-1}} \left(\ell(t)+k\right) \notag\\
&\leq \frac{1}{2}\sup_{Q_{\rho, \theta}} u_+ +c(p) \delta \boldsymbol{\mathfrak{S}} +c(p)\left[\left(\frac{R}{R-\rho}\right)^{d+sp} \int_{-\theta}^0 \int_{\R^d \setminus B_R}\frac{u_+^{p-1}}{|x|^{d+sp}}\dxt \right]^{\nicefrac{1}{(p-1)}} \notag\\
&  \quad \quad \quad \quad \quad +c(p)\frac{\widetilde{\boldsymbol{\gamma}}}{(1-\sigma)^{(d+p)q}}\boldsymbol{\cA}_\delta\left(\dashint_{Q_{\rho,\theta}}u_+^p\dxt \right)^{\nicefrac{1}{p}} \notag\\
&\leq \frac{1}{2}\sup_{Q_{\rho, \theta}} u_+ +\frac{c(p)}{(1-\sigma)^{(d+p)q}} \mathsf{Q},
\end{align*}
where
\begin{align*}
\mathsf{Q}:=\delta \boldsymbol{\mathfrak{S}}&+\left[\left(\frac{R}{R-\rho}\right)^{d+sp} \int_{-\theta}^0 \int_{\R^d \setminus B_R}\frac{u_+^{p-1}}{|x|^{d+sp}}\dxt \right]^{\nicefrac{1}{(p-1)}}+\boldsymbol{\cA}_\delta\left(\dashint_{Q_{\rho,\theta}}u_+^p\dxt \right)^{\nicefrac{1}{p}}.
\end{align*}
Appealing to this over cylinders $Q_{\sigma   \rho, \sigma  \theta} \subset Q_{\widetilde{\sigma} \rho, \widetilde{\sigma} \theta}$ with $\nicefrac{1}{2} \leq \sigma < \widetilde{\sigma} \leq 1$, we deduce that
\[
\sup_{Q_{\sigma  \rho, \sigma \theta}} u_+ \leq \frac{1}{2}\sup_{Q_{\widetilde{\theta}\rho, \widetilde{\sigma} \theta}} u_+ +\frac{c(p)}{(\widetilde{\sigma}-\sigma)^{(d+p)q}} \mathsf{Q}.
\]
Again, a standard iteration argument (found, for example, in~\cite[Lemma 6.1, Page 191]{Giusti}) yields, 
\begin{equation}\label{eq:bnd-est-eq9}
\sup_{Q_{\sigma \rho, \sigma \theta}} u_+ \leq \frac{C}{(1-\sigma)^{(d+p)q}} \mathsf{Q}
\end{equation}
for a constant $C\equiv C(p,q)<\infty$.

\smallskip

\emph{Step 4: Interpolation.} We refine estimate~\eqref{eq:bnd-est-eq9}. For fixed $\nu \in (0,p)$, Young's inequality implies
\begin{align*}
&\frac{\widetilde{\boldsymbol{\gamma}}}{(1-\sigma)^{(d+p)q}}\boldsymbol{\cA}_\delta\left(\dashint_{Q_{\rho,\theta}}u_+^p\dxt \right)^{\nicefrac{1}{p}} \\
& \quad \leq \widetilde{\boldsymbol{\gamma}}\boldsymbol{\cA}_\delta\left[\frac{1}{(1-\sigma)^{(d+p)qp}}\dashint_{Q_{\rho,\theta}}u_+^\nu\dxt \right]^{\nicefrac{1}{p}} \left(\sup_{Q_{R,\theta}}u_+\right)^{\nicefrac{(p-\nu)}{p}}\\
&\quad \leq \delta \sup_{Q_{R,\theta}}u_++\boldsymbol{\gamma}^{\frac{p}{\nu}}\delta^{-\frac{p-\nu}{\nu}}\boldsymbol{\cA}^{\frac{p}{\nu}}_\delta\left[\frac{1}{(1-\sigma)^{(d+p)qp}}\dashint_{Q_{\rho,\theta}}u_+^\nu\dxt \right]^{\nicefrac{1}{\nu}}\,.
\end{align*}
Inserting this into~\eqref{eq:bnd-est-eq9} with $\sigma$ replaced by $\nicefrac{(1+\sigma)}{2}$ and rearranging, we deduce that
\begin{align}\label{eq:bnd-est-eq10}
\sup_{Q_{\frac{1+\sigma}{2} \rho, \frac{1+\sigma}{2} \theta}}u_+
&\leq \frac{C}{(1-\sigma)^{(d+p)q}}\delta\left[1+\left(\frac{R}{\rho}\right)^{\frac{d}{p-1}}\right]\sup_{Q_{R,\theta}} u_+ \notag\\
&\quad \quad   +\frac{C}{(1-\sigma)^{(d+p)q}}\left[\left(\frac{R}{R-\rho}\right)^{d+sp} \int_{-\theta}^0 \int_{\R^d \setminus B_R}\frac{u_+^{p-1}}{|x|^{d+sp}}\dxt \right]^{\nicefrac{1}{(p-1)}} \notag\\
&\quad \quad  +C\boldsymbol{\gamma}^{\frac{p}{\nu}}\delta^{-\frac{p-\nu}{\nu}}\boldsymbol{\cA}^{\frac{p}{\nu}}_\delta\left[\frac{1}{(1-\sigma)^{(d+p)qp}}\dashint_{Q_{\rho,\theta}}u_+^\nu\dxt \right]^{\nicefrac{1}{\nu}}
\end{align}
whenever $Q_{R,\theta} \subset \Omega_T$, $\sigma \in (0,1)$ and $\rho \in (0,R)$. Constant $C$ depends only on $d,s,p$ since $q$ does. At this stage, we will perform an interpolation argument. For this, we introduce, for $i \in \N_0$,
\[
\left\{
\begin{array}{l}
\rho_i:=\rho-2^{-(i+1)}(1-\sigma)\rho, \quad \widetilde{\rho}_i:=\frac{1}{2}(\rho_i+\rho_{i+1}) \\[2mm]
\theta_i:=\theta-2^{-(i+1)}(1-\sigma)\theta,\quad \widetilde{\theta}_i:=\frac{1}{2}(\theta_i+\theta_{i+1})
\end{array}
\right.
\]
and, subsequently, define $\{\sigma_i\}_{i \in \N_0}$ inductively as $\sigma_i:=\rho_i/\widetilde{\rho}_i=\theta_i/\widetilde{\theta}_i$. By construction it holds
\[
\left\{
\begin{array}{l}
\frac{1+\sigma}{2} \rho=\rho_0 \leq \cdots \leq \rho_i \leq \rho_{i+1} \uparrow  \rho_\infty=\rho,\\[2mm]
\frac{1+\sigma}{2} \theta=\theta_0 \leq \cdots \leq \theta_i \leq \theta_{i+1} \uparrow  \theta_\infty=\theta.
\end{array}
\right.
\]
We temporary abused the letter $\sigma \in (0,1)$ as $\sigma_i$ will play the role of $\sigma$ appearing in~\eqref{eq:bnd-est-eq10}. Under these preliminary tools, we apply~\eqref{eq:bnd-est-eq10} with $(\sigma, \rho, R, \theta)$ replaced by $(\sigma_i, \widetilde{\rho}_i, \rho_{i+1}, \theta_{i+1})$ and notice that $\theta_i\leq \sigma_i \theta_{i+1}$. We therefore get, for all $i \in \N_0$, that
\begin{align*}
\sup_{Q_{\rho_i, \theta_i}}u_+ &\leq \frac{C}{(1-\sigma_i)^{(d+p)q}}\delta\left[1+\left(\frac{\rho_{i+1}}{\rho_i}\right)^{\frac{d}{p-1}}\right]\sup_{Q_{\rho_{i+1},\theta_{i+1}}} u_+ \notag\\
&\quad \quad  +\frac{C}{(1-\sigma_i)^{(d+p)q}}\left[\left(\frac{\rho_{i+1}}{\rho_{i+1}-\widetilde{\rho}_i}\right)^{d+sp} \int_{-\theta_{i+1}}^0 \int_{\R^d \setminus B_{\rho_{i+1}}}\frac{u_+^{p-1}}{|x|^{d+sp}}\dxt \right]^{\nicefrac{1}{(p-1)}} \notag\\
&\quad \quad  +C\boldsymbol{\gamma}^{\frac{p}{\nu}}\delta^{-\frac{p-\nu}{\nu}}\left[\frac{\boldsymbol{\cA}^{p}_{\delta,i}}{(1-\sigma_i)^{(d+p)qp}\widetilde{\rho}_i^d\theta_{i+1}}\int_{Q_{\widetilde{\rho}_i,\theta_{i+1}}}u_+^\nu\dxt \right]^{\nicefrac{1}{\nu}}\\
&=:(\mathrm{V})+(\mathrm{VI})+(\mathrm{VII}).
\end{align*}
Here we abbreviated
\[
\boldsymbol{\cA}^{p}_{\delta,i} \equiv \boldsymbol{\cA}^{p}_{\delta,i}\left(\sigma_i, \widetilde{\rho}_i, \rho_{i+1},\theta_{i+1}\right):= \left(\frac{\widetilde{\rho}_i^{sp}}{\theta_{i+1}}+\frac{1}{\delta^{p-1}\sigma_i^{d+sp}}\right)^{pq} \left(\frac{\theta_{i+1}}{\widetilde{\rho}_i^{sp}}\right).
\]
Here notice, by $\sigma_i \geq \sigma$, that
\begin{align}\label{eq:bnd-est-eq11}
\boldsymbol{\cA}^{p}_{\delta,i} &\leq  \left(\frac{\rho^{sp}}{\frac{1+\sigma}{2}\theta}+\frac{1}{\delta^{p-1}\sigma^{d+sp}}\right)^{pq} \left[\frac{\theta}{(\frac{1+\sigma}{2}\rho)^{sp}}\right] \notag\\
&\leq 2^{p(q+s)} \left(\frac{\rho^{sp}}{\theta}+\frac{1}{\delta^{p-1}\sigma^{d+sp}}\right)^{pq} \left(\frac{\theta}{\rho^{sp}}\right)^{}=2^{p(q+s)}\boldsymbol{\cA}^{p}_{\delta}.
\end{align}
We are going to estimate three integrals $(\mathrm{V})$--$(\mathrm{VII})$.  by observation that
\[
1-\sigma_i=1-\frac{\rho_i}{\widetilde{\rho}_i} =\frac{\rho_{i+1}-\rho_i}{\rho_{i+1}+\rho_i} \geq \frac{1-\sigma}{2^{i+2}},
\]
$(\mathrm{V})$ can be estimated, using $\rho_{i+1}/\widetilde{\rho}_i \leq 2$, 
\begin{align*}
(\mathrm{V})\leq C \frac{2^{i(d+p)q}}{(1-\sigma)^{(d+p)q}}\delta\sup_{Q_{\rho_{i+1},\theta_{i+1}}} u_+,
\end{align*}
whereas, it is straightforward to check that
\begin{align*}
(\mathrm{VI}) &\leq C \left(\frac{2^{i+2}}{1-\sigma}\right)^{(d+p)q}\left[\left(\frac{2^{i+2}}{1-\sigma}\right)^{d+sp} \int_{-\theta}^0\int_{\R^d \setminus B_{\sigma \rho}}\frac{u_+^{p-1}}{|x|^{d+sp}}\dxt \right]^{\nicefrac{1}{(p-1)}} \\
&\leq C \frac{2^{i \left((d+p)q+\frac{d+sp}{p-1}\right)} }{(1-\sigma )^{(d+p)q+\frac{d+sp}{p-1}}}\left[ \frac{\theta}{(\sigma \rho)^{sp}}\dashint_{-\theta}^0\Big[\tail (u_+(t) ; B_{\sigma \rho})\Big]^{p-1}\d{t} \right]^{\nicefrac{1}{(p-1)}}.
\end{align*}
Analogously, the denominator of the bracket in $(\mathrm{VII})$ is estimated as
\begin{align*}
(1-\sigma_i)^{\frac{(d+p)q}{p}}\widetilde{\rho}_i^d\theta_{i+1} &\geq \left(\frac{1-\sigma}{2^{i+2}}\right)^{\frac{(d+p)q}{p}} \left(\frac{1+\sigma}{2}\right)^{d+1}\rho^d\theta \\
&\geq \left(\frac{1-\sigma}{2^{i+2}}\right)^{\frac{(d+p)q}{p}} \frac{\rho^d\theta}{2^{d+1}}.
\end{align*}
This estimate together with~\eqref{eq:bnd-est-eq11} provides that
\begin{align*}
(\mathrm{VII}) &\stackrel{\eqref{eq:bnd-est-eq11}}{\leq} C\widetilde{\boldsymbol{\gamma}}^{\frac{p}{\nu}}\delta^{-\frac{p-\nu}{\nu}}2^{\frac{p}{\nu}(q+s)}\left[\frac{\boldsymbol{\cA}^{p}_{\delta}}{\left(\frac{1-\sigma}{2^{i+2}}\right)^{(d+p)qp} \frac{\rho^d\theta}{2^{d+1}}}\int_{Q_{\rho,\theta}}u_+^\nu\dxt \right]^{\nicefrac{1}{\nu}}\\
& \,\, \,\leq  C\widetilde{\boldsymbol{\gamma}}^{\frac{p}{\nu}}\delta^{-\frac{p-\nu}{\nu}}2^{\frac{(q+1)p}{\nu}\left[(d+p)(i+2)+d+2\right]}\left[\frac{\boldsymbol{\cA}^{p}_{\delta}}{(1-\sigma)^{(d+p)qp}}\dashint_{Q_{\rho,\theta}}u_+^\nu\dxt \right]^{\nicefrac{1}{\nu}}\,.
\end{align*}
Finally, after merging all these estimates $(\mathrm{V})$--$(\mathrm{VII})$ with~\eqref{eq:bnd-est-eq10}, denoting
\begin{align*}
\widetilde{\boldsymbol{b}_\nu}&:=2^{\left((d+p)q+\frac{d+sp}{p-1} \right)\vee  \left(\frac{(q+1)(d+p)p}{\nu}\right)},\\[4pt]
\boldsymbol{\cB}_\delta&:=\frac{C}{(1-\sigma)^{(d+p)q+\frac{d+sp}{p-1}}}\left[ \frac{\theta}{(\sigma \rho)^{sp}}\dashint_{-\theta}^0 \Big[\tail (u_+(t) ; B_{\sigma\rho})\Big]^{p-1}\d{t} \right]^{\nicefrac{1}{(p-1)}} \\
&\quad \quad \quad \quad +C_\nu\delta^{-\frac{p-\nu}{\nu}}\left[\frac{\boldsymbol{\cA}^{p}_{\delta}}{(1-\sigma)^{(d+p)qp}}\dashint_{Q_{\rho,\theta}}u_+^\nu\dxt \right]^{\nicefrac{1}{\nu}},
\end{align*}
which eventually leads to
\[
\sup_{Q_{\rho_i, \theta_i}} u_+ \leq C\delta \sup_{Q_{\rho_{i+1}, \theta_{i+1}}}u_++\boldsymbol{\cB}_\delta\widetilde{\boldsymbol{b}_\nu}^i, \quad \forall i \in \N_0.
\]
Clearly, as seen in the definition, the quantity $\boldsymbol{\cB}_\delta$ depends on $\delta$, but is independent of $i \in \N_0$. Iterating this recursive inequalities over  $i \in \{0,1,\ldots, j-1\}$ we end up with
\[
\sup_{Q_{\rho_0, \theta_0}} u_+ \leq \left(C\delta \right)^j \sup_{Q_{\rho_j, \theta_j}}u_++\boldsymbol{\cB}_\delta\sum_{i=0}^{j-1}\left(C \delta \widetilde{\boldsymbol{b}_\nu}\right)^i.
\]
Finally, select $\delta \in (0,1]$ so that
\[
C\delta \widetilde{\boldsymbol{b}_\nu}=\nicefrac{1}{2}  \quad \iff \quad \delta=\nicefrac{1}{(2C\widetilde{\boldsymbol{b}_\nu})} ,
\]
and take the limit $j \to \infty$ in the above display. This gives that
\[
\sup_{Q_{\sigma \rho, \sigma \theta}}u_+ \leq \sup_{Q_{\frac{1+\sigma}{2} \rho, \frac{1+\sigma}{2}\theta}}u_+ \leq 2\boldsymbol{\cB}_\delta=2\boldsymbol{\cB}_{\nicefrac{1}{(2C\widetilde{\boldsymbol{b}}_\nu)}}.
\]
Therefore, the proof is finally complete. 
\end{proof}

\section{Several expansions of positivity}\label{Sect.4}
This section is devoted to establishing the fundamental tools of our regularity theory: Critical Mass Lemmata, Measure-Propagation and a quantitative \emph{Expansion of Positivity}. Unlike the local case, the non-locality of the energy class calls for a careful tracking of the tail effect.
\smallskip

Let $k$ be a positive number and let $k(t)=k$ be time-independent, and hence $k^\prime(t)=0$. By choosing a proper cutoff function $\zeta$ (see for instance~\cite[Appendix B]{Nak23}), we can obtain from the definition~\eqref{d2} the following Caccioppoli type estimate, that encodes all the energetic information needed for this section. 
\begin{align}\label{d3}
&\sup\limits_{t_0-\tau_2<t<t_0}\int_{B_{r_1}(x_0)}w_\pm^p(x,t)\d{x}\notag\\[4pt]
&\quad \quad+\int_{t_o-\tau_2}^{t_0}\int_{B_{r_1}(x_0)}\int_{B_{r_1}(x_0)}\dfrac{\big|w_{\pm}(x,t)-w_{\pm}(y,t)\big|^p}{|x-y|^{d+sp}}\dxyt \notag \\[4pt]
&\quad \quad+\int_{Q_{r_1,\tau_1}(z_0)}w_\pm(x,t)\left(\int_{\R^d}\dfrac{w_\mp^{p-1}(y,t)}{|x-y|^{d+sp}}\d{y}\right)\dxt \notag \\[4pt]
&\quad \leq \int_{B_{r_1}(x_0)}w_\pm^p(x, t_0-\tau_2)\d{x}+\gamma_{\mathrm{DG}}\frac{r_2^{(1-s)p}}{(r_2-r_1)^p}\int_{Q_{r_2,\tau_2}(z_0)}w_\pm^p\dxt \notag \\[4pt]
&\quad \quad +\gamma_{\mathrm{DG}} \frac{r_2^{d}}{(r_2-r_1)^{d+sp}}\int_{Q_{r_2,\tau_2}(z_0)}w_{\pm}(x,t)\d{x} \Big[\tail (w_\pm(t)\,; B_{r_2}(x_0) )\Big]^{p-1}\d{t},
\end{align}
and
\begin{align}\label{d3'}
&\sup\limits_{t_0-\tau_2<t<t_0}\int_{B_{r_1}(x_0)}w_\pm^p(x,t)\d{x}\notag\\[4pt]
&\quad \quad+\int_{t_0-\tau_2}^{t_0}\int_{B_{r_1}(x_0)}\int_{B_{r_1}(x_0)}\dfrac{\big|w_{\pm}(x,t)-w_{\pm}(y,t)\big|^p}{|x-y|^{d+sp}}\dxyt \notag \\[4pt]
&\quad \quad+\int_{Q_{r_1,\tau_1}(z_0)}w_\pm(x,t)\left(\int_{\R^d}\dfrac{w_\mp^{p-1}(y,t)}{|x-y|^{d+sp}}\d{y}\right)\dxt \notag \\[4pt]
&\quad \leq \gamma_{\mathrm{DG}}\left[\frac{r_2^{(1-s)p}}{(r_2-r_1)^p} +\frac{1}{\tau_2-\tau_1}\right]\int_{Q_{r_2,\tau_2}(z_0)}w_\pm^p\dxt \notag \\[4pt]
&\quad \quad +\gamma_{\mathrm{DG}} \frac{r_2^{d}}{(r_2-r_1)^{d+sp}}\int_{Q_{r_2,\tau_2}(z_0)}w_{\pm}(x,t)\d{x} \Big[\tail (w_\pm(t)\,; B_{r_2}(x_0) )\Big]^{p-1}\d{t},
\end{align}
for any $Q_{r_1,\tau_1} (z_0)\subset Q_{r_2,\tau_2}(z_0) \subset \cQ_R$, where, as usual, $w_\pm:=(u-k)_\pm$ for short.

\subsection{De Giorgi-type Lemma}\label{Sect.4.1}
Here we establish a De Giorgi-type lemma, in a similar fashion to \cite{DiB93}, also called in some other contexts "Critical mass" Lemma.

\begin{lemma}[De Giorgi type lemma]\label{Lm:DGtype}
Let $p>1,s \in (0,1)$. For $z_0=(x_0,t_0) \in \Omega_T$, suppose that $Q_{2\rho,  2\tau }(z_0) \subset \mathcal{Q}_{R}\Subset \Omega_T$, where $\tau:=\delta \rho^{sp}$ with an arbitrary number $\delta \in (0,1]$. Let $u \in \mathbf{PDG}_-^{s,p}(\Omega_T,\gamma_{\mathrm{DG}}, \eps)$, in the sense of Definition~\ref{def of DG}, satisfy $u \geq 0$ in $\mathcal{Q}_R$. Then there exists $\nu \in (0,1)$, depending only on $\data$ such that, if
\[
\left(\frac{\rho}{R}\right)^{\frac{sp\eps}{(p-1)(p-1+\eps)}} \left(\dashint_{I_R(t_0)}\Big[\tail(u_-(t)\,; B_R(x_0))\Big]^{p-1+\eps}\d{t}\right)^{\nicefrac{1}{(p-1+\eps)}} \leq \frac{k}{2}
\]
and
\[
\left|\{u<k\} \cap Q_{2\rho,  2\tau }(z_0)\right| \leq \nu \left|Q_{2\rho,  2\tau }(z_0)\right|
\]
then 
\[
u \geq \frac{k}{2} \quad \textrm{a.e.\,\,in}\,\, \,Q_{\rho, \tau }(z_0).
\]
We can track $\nu=\nu_0 \delta^q$ for some $\nu_0 \in (0,1)$ and $q>1$, both depending only on $\data$.
\end{lemma}
\begin{proof}
We assume $z_0=(x_0,t_0)=(0,0)$ and split the proof into three steps. 

\smallskip

\emph{Step 1.}\quad To set up the iteration \`{a} la De Giorgi, we define  $Q_i:=Q_{\rho_i,\tau_i}:=B_i \times T_i$ for $i \in \N_0$, where $\rho_i$ and $\tau_i$ are defined according to
\[
\rho_i:=(1+2^{-i})\rho, \quad \tau_i:=\delta (1+2^{-i})\rho^{sp}
\]
with $\delta \in (0,1]$ being an arbitrary number and where we shortened $B_i\equiv B_{\rho_i}$ and $T_i \equiv (-\tau_i,0]$. Let
\[
\widetilde{\rho}_i:=\frac{1}{2}(\rho_i+\rho_{i+1}), \quad \widetilde{\tau}_i:=\frac{1}{2}(\tau_i+\tau_{i+1}) ; 
\]
and correspondingly, let $\widetilde{B}_i:=B_{\widetilde{\rho}_i}$, $\widetilde{T}_i:=(-\widetilde{\tau}_i,0]$ and $\widetilde{Q}_i:=\widetilde{B}_i\times \widetilde{T}_i$. These cylinders are constructed in such a way that $Q_{i+1} \subset \widetilde{Q}_i \subset Q_i$ for every $i \in \N_0$. For fixed $k>0$, we define the sequences of levels
\[
k_i:=\frac{k}{2}+\frac{k}{2^{i+1}}, \quad \widetilde{k}_i:=\frac{1}{2}(k_i+k_{i+1}).
\]
Finally we set
\[
A_i:=\left\{u<k_i\right\} \cap Q_i, \quad \boldsymbol{Y}_{\!\!i}:=\frac{|A_i|}{|Q_i|} \leq 1, \quad \widetilde{w}:=(u-\widetilde{k}_i)_-.
\]
\smallskip

\emph{Step 2.}\quad To begin, let $\varphi \in C^\infty(Q_i ; [0,1])$ be a cutoff function satisfying 
\[
\1_{Q_{i+1}} \leq \varphi \leq \1_{Q_i}, \quad |\nabla \varphi| \leq 2^{i+4}/\rho, \quad \supp(\varphi) \subset \widetilde{Q}_i. 
\]
Then, an application of Proposition~\ref{FS} with $\mathsf{d}=2^{-i-4}$ and H\"{o}lder's inequality imply that
\begin{align}\label{eq:dg-est-1}
\frac{k}{2^{i+3}}\left|A_{i+1}\right|
&\leq \int_{\widetilde{Q}_i} \widetilde{w}\varphi\dxt \notag\\
&\leq \left[\int_{\widetilde{Q}_i} (\widetilde{w}\varphi)^{p \kappa}\dxt\right]^{\nicefrac{1}{p\kappa}}\left|A_i\right|^{1-\nicefrac{1}{p\kappa}} \notag\\
&\leq c \Bigg[\rho^{sp} \int_{\widetilde{T}_i}\int_{\widetilde{B}_i}\int_{\widetilde{B}_i} \frac{\left|\varphi \widetilde{w}(x,t)-\varphi \widetilde{w}(y,t)\right|^p}{|x-y|^{d+sp}}\dxt \notag\\
&\quad \quad \quad +2^{i(d+sp)}\int_{\widetilde{Q}_i}(\varphi \widetilde{w})^p\dxt\Bigg]^{\nicefrac{1}{p\kappa}}\left[\sup_{t _\in \widetilde{T}_i}\dashint_{\widetilde{B}_i}(\varphi \widetilde{w})^p(t)\d{x} \right]^{\frac{\kappa_\ast-1}{p\kappa \kappa_\ast}}\left|A_i\right|^{1-\nicefrac{1}{p\kappa}} 
\end{align}
with $\kappa=1+\frac{\kappa_\ast-1}{\kappa_\ast}$. Triangle inequality implies
\begin{align*}
\left|\varphi\widetilde{w}(x,t)-\varphi\widetilde{w}(y,t)\right|^p \lesssim_p \left|\widetilde{w}(x,t)-\widetilde{w}(y,t)\right|^p\varphi^p(x,t)+\widetilde{w}^p(y,t)\left|\varphi(x,t)-\varphi(y,t)\right|^p,
\end{align*}
and therefore~\eqref{d3'} can be applied to get 
\begin{align}\label{eq:dg-est-d}
\rho^{sp} &\int_{\widetilde{T}_i}\int_{\widetilde{B}_i}\int_{\widetilde{B}_i} \frac{\left|\varphi \widetilde{w}(x,t)-\varphi \widetilde{w}(y,t)\right|^p}{|x-y|^{d+sp}}\dxt \notag\\
&\lesssim_p \rho^{sp}  \int_{\widetilde{T}_i}\int_{\widetilde{B}_i}\int_{\widetilde{B}_i} \frac{\left| \widetilde{w}(x,t)-\widetilde{w}(y,t)\right|^p}{|x-y|^{d+sp}}\dxt \notag\\
&\quad+\rho^{sp} \left(\frac{2^{i+4}}{\rho}\right)^p\int_{\widetilde{Q}_i}\widetilde{w}^p(y,t)\d{y}\left(\int_{\widetilde{B}_i}\frac{\d{x}}{|x-y|^{d+(s-1)p}}\right)\d{t} \notag\\
&\!\!\!\stackrel{\eqref{d3'}}{\lesssim}_{d,s,p,\gamma_{\mathrm{DG}}}\rho^{sp}\left[\frac{\rho_i^{(1-s)p}}{(\rho_i-\widetilde{\rho}_i)^p}+\frac{1}{\tau_i-\widetilde{\tau}_i}\right]\int_{\widetilde{Q}_i}\widetilde{w}^p\dxt \notag\\
&\quad+\rho^{sp}\frac{\rho_i^{d}}{(\rho_i-\widetilde{\rho}_i)^{d+sp}}\int_{Q_i}\widetilde{w}\d{x}\Big[\tail(\widetilde{w}(t)\,; B_i)\Big]^{p-1}\d{t}+2^{ip}\int_{Q_i}\widetilde{w}^p\dxt\,,
\end{align}
where in the penultimate line we used a simple estimate:
\[
\int_{\widetilde{B}_i}\frac{\d{x}}{|x-y|^{d+(s-1)p}} \leq \int_{2\widetilde{B}_i(y)}\frac{\d{x}}{|x-y|^{d+(s-1)p}}\leq c(d,s,p)\rho^{(1-s)p}.
\]
A simple calculation yields
\[
\frac{\rho_i^{(1-s)p}}{(\rho_i-\widetilde{\rho}_i)^p}, \,\,\frac{\rho_i^{d}}{(\rho_i-\widetilde{\rho}_i)^{d+sp}} \leq \frac{c(d,s,p)2^{ip}}{\rho^{sp}}, \quad \frac{1}{\tau_i-\widetilde{\tau}_i} \leq \frac{c(d,s,p)2^{ip}}{\delta \rho^{sp}},
\]
thereby we conclude that
\begin{align}\label{eq:dg-est-2}
\rho^{sp} &\int_{\widetilde{T}_i}\int_{\widetilde{B}_i}\int_{\widetilde{B}_i} \frac{\left|\varphi \widetilde{w}(x,t)-\varphi \widetilde{w}(y,t)\right|^p}{|x-y|^{d+sp}}\dxt \notag\\
&\leq c2^{ip}\left(1+\frac{1}{\delta}\right)\int_{Q_i}\widetilde{w}^p\dxt+c2^{ip}\int_{Q_i}\widetilde{w}\d{x}\Big[\tail(\widetilde{w}(t)\,; B_i)\Big]^{p-1}\d{t}
\end{align}
with $c\equiv c(d,s,p,\gamma_{\mathrm{DG}})<\infty$. The integral average on $\widetilde{B}_i$ containing $(\varphi \widetilde{w})^p(\cdot, t)$ can be treated similarly. Indeed, similarly as above we have
\begin{equation}\label{eq:dg-est-3}
\sup_{t _\in \widetilde{T}_i}\dashint_{\widetilde{B}_i}(\varphi \widetilde{w})^p(t)\d{x} \leq \frac{c 2^{ip}}{|B_i|}\rho^{-sp}\Big[ (\mathrm{I})+(\mathrm{II})\Big],
\end{equation}
where we used the fact that $|B_i|/|\widetilde{B}_i| \leq c(d)$ and set
\begin{align*}
(\mathrm{I})&:=\left(1+\frac{1}{\delta}\right)\int_{Q_i}\widetilde{w}^p(x,t)\dxt, \\
(\mathrm{II})&:=\int_{Q_i}\widetilde{w}\d{x}\Big[\tail\left(\widetilde{w}(t)\,;B_i\right) \Big]^{p-1}\d{t}.
\end{align*}
A straightforward computation shows that $(\mathrm{I})  \leq 2k^p |A_i|/\delta$, while as for $(\mathrm{II})$, using $u(\cdot, t) \geq 0$ in $B_R$ and $\widetilde{k}_i \leq k$ we have that
\begin{align*}
\Big[\tail\left(\widetilde{w}(t)\,;B_i\right) \Big]^{p-1} &=\rho_i^{sp}\int_{B_R \setminus B_i} \frac{\widetilde{w}^{p-1}(x,t)}{|x|^{d+sp}}\d{x}+\rho_i^{sp}\int_{\R^d \setminus B_R} \frac{\widetilde{w}^{p-1}(x,t)}{|x|^{d+sp}}\d{x} \\
& \leq ck^{p-1}+c\left(\frac{\rho}{R}\right)^{sp}\Big[\tail\left((u(t)-k)_-\,;B_R\right) \Big]^{p-1}
\end{align*}
with a constant $c\equiv c(d,s,p)<\infty$, where in the last line we used that
\[
\int_{B_R \setminus B_i}\frac{\d{x}}{|x|^{d+sp}} \leq c(d,s,p)\rho_i^{-sp}.
\]
This in turn leads to
\begin{align*}
(\mathrm{II})
&\leq \int_{Q_i}(u-k_i)_-\d{x}\left[ck^{p-1}+c\left(\frac{\rho}{R}\right)^{sp}\Big[\tail\left((u(t)-k)_-\,;B_R\right) \Big]^{p-1}\right]\d{t}\\
&\leq ck^p|A_i|+c\underbrace{\left(\frac{\rho}{R}\right)^{sp}\int_{Q_i}(u-k_i)_-\d{x}\Big[\tail\left((u(t)-k)_-\,;B_R\right) \Big]^{p-1}\d{t}}_{=:(\mathrm{III})}.
\end{align*}
Joining these estimates with~\eqref{eq:dg-est-1}--\eqref{eq:dg-est-3} yields
\begin{align}\label{eq:dg-est-4}
\frac{k}{2^{i+3}}\left|A_{i+1}\right| &\leq c \left[2^{i(d+sp)}\left[(\mathrm{I})+(\mathrm{II})\right]\right]^{\nicefrac{1}{p\kappa}}\left[\frac{2^{ip}\rho^{-sp}}{|B_i|}\left[(\mathrm{I})+(\mathrm{II})\right]\right]^{\frac{1}{p\kappa}\cdot \frac{\kappa_\ast-1}{\kappa_\ast}}|A_i|^{1-\nicefrac{1}{p\kappa}} \notag\\
&=c2^{i(d+sp)} \left(\frac{\rho^{-sp}}{|B_i|}\right)^{\frac{1}{p\kappa}\cdot \frac{\kappa_\ast-1}{\kappa_\ast}}\Big[(\mathrm{I})+(\mathrm{II})\Big]^{\nicefrac{1}{p}}|A_i|^{1-\nicefrac{1}{p\kappa}} \notag\\
&\leq c2^{i(d+sp)} \left(\frac{\rho^{-sp}}{|B_i|}\right)^{\frac{1}{p\kappa}\cdot \frac{\kappa_\ast-1}{\kappa_\ast}}\frac{1}{\delta^{\nicefrac{1}{p}}}\Big[k^p|A_i|+(\mathrm{III})\Big]^{\nicefrac{1}{p}}|A_i|^{1-\nicefrac{1}{p\kappa}}
\end{align}
where $c\equiv c(d,s,p,\gamma_{\mathrm{DG}})<\infty$. Here in the above manipulation $1+\frac{\kappa_\ast-1}{\kappa_\ast}=\kappa$ was used.

Further, the estimate of $(\mathrm{III})$ on the right side of~\eqref{eq:dg-est-4} can be derived by a careful inspection. Indeed, denoting the time-slice of $A_i$ by
\[
A_i(t):=\left\{u(\cdot,t)<k_i\right\} \cap B_i \quad \textrm{for}\,\,\, t \in T_i, 
\]
the H\"{o}lder inequality with exponents $\left(\frac{p-1+\eps}{p-1}, \frac{p-1+\eps}{\eps}\right)$ provides that
\begin{align*}
(\mathrm{III})
&\leq c\rho^{\frac{(p-1)sp}{p-1+\eps}}k \left(\frac{\rho}{R}\right)^{\frac{\eps sp}{p-1+\eps}}\left(\dashint_{I_R} \Big[\tail\left((u(t)-k)_-\,;B_R\right) \Big]^{p-1+\eps}\d{t} \right)^{\frac{p-1}{p-1+\eps}} \\
&\quad \quad \quad \quad \quad \quad \quad \quad \quad \quad \times \left(\int_{T_i}|A_i(t)|^{\frac{p-1+\eps}{\eps}}\d{t}\right)^{\nicefrac{\eps}{(p-1+\eps)}},
\end{align*}
where we used that $\left(\frac{\rho}{R}\right)^{sp} R^{\frac{(p-1)sp}{p-1+\eps}}=\rho^{\frac{(p-1)sp}{p-1+\eps}}\left(\frac{\rho}{R}\right)^{\frac{\eps sp}{p-1+\eps}}$.
Now, if we enforce
\[
\left(\frac{\rho}{R}\right)^{\frac{sp\eps}{(p-1)(p-1+\eps)}} \left(\dashint_{I_R}\Big[\tail(u_-(t)\,; B_R)\Big]^{p-1+\eps}\d{t}\right)^{\nicefrac{1}{(p-1+\eps)}} \leq \frac{k}{2}
\]
then the above estimate becomes
\[
(\mathrm{III})\leq c\rho^{\frac{(p-1) sp}{p-1+\eps}}k^p\left(\int_{T_i}|A_i(t)|^{\frac{p-1+\eps}{\eps}}\d{t}\right)^{\nicefrac{\eps}{(p-1+\eps)}}.
\]
Overall, we have shown from~\eqref{eq:dg-est-4} that
\begin{align*}
\frac{k}{2^{i+3}}\left|A_{i+1}\right| &\leq c2^{i(d+sp)} \left(\frac{\rho^{-sp}}{|B_i|}\right)^{\frac{1}{p\kappa}\cdot  \frac{\kappa_\ast-1}{\kappa_\ast}}\frac{|A_i|^{1-\nicefrac{1}{p\kappa}}}{\delta^{\nicefrac{1}{p}}} \\
&\quad \quad \quad \times \left[|A_i|^{\nicefrac{1}{p}}+\rho^{\frac{(p-1) s}{p-1+\eps}}\left(\int_{T_i}|A_i(t)|^{\frac{p-1+\eps}{\eps}}\d{t}\right)^{\nicefrac{\eps}{p(p-1+\eps)}}\right]k.
\end{align*}
Dividing the both side of the last display by $|Q_{i+1}|$ yields that
\begin{align}\label{eq:dg-est-5}
\frac{\left|A_{i+1}\right|}{\left|Q_{i+1}\right|} &\leq c 4^{i(d+sp)} \frac{|Q_i|^{1+\nicefrac{1}{p}-\nicefrac{1}{p\kappa}}}{\left|Q_{i+1}\right|}\left(\frac{\rho^{-sp}}{|B_i|}\right)^{\frac{1}{p\kappa}\cdot \frac{\kappa_\ast-1}{\kappa_\ast}}\frac{1}{\delta^{\nicefrac{1}{p}}}\notag \\
& \quad \times \left[\left(\frac{\left|A_{i}\right|}{\left|Q_{i}\right|} \right)^{1+\nicefrac{1}{p}-\nicefrac{1}{p\kappa}}+\left(\frac{\left|A_{i}\right|}{\left|Q_{i}\right|} \right)^{1-\nicefrac{1}{p\kappa}}\frac{\rho^{\frac{(p-1) s}{p-1+\eps}}}{\left|Q_i\right|^{\nicefrac{1}{p}}}\left(\int_{T_i}|A_i(t)|^{\frac{p-1+\eps}{\eps}}\d{t}\right)^{\nicefrac{\eps}{p(p-1+\eps)}}\right].
\end{align}
Observe that
\[
\frac{|Q_i|^{1+\nicefrac{1}{p}-\nicefrac{1}{p\kappa}}}{\left|Q_{i+1}\right|}\left(\frac{\rho^{-sp}}{|B_i|}\right)^{\frac{1}{p\kappa}\cdot \frac{\kappa_\ast-1}{\kappa_\ast}} =\frac{\left|Q_i\right|}{\left|Q_{i+1}\right|} \left(\frac{\left|Q_i\right|}{\left|B_i\right|\rho^{sp}}\right)^{\nicefrac{1}{p}-\nicefrac{1}{p\kappa}} \leq c(d)\delta^{\nicefrac{1}{p}-\nicefrac{1}{p\kappa}}
\]
and 
\[
\frac{\rho^{\frac{(p-1) s}{p-1+\eps}}}{\left|Q_i\right|^{\nicefrac{1}{p}}}\left(\int_{T_i}|A_i(t)|^{\frac{p-1+\eps}{\eps}}\d{t}\right)^{\nicefrac{\eps}{p(p-1+\eps)}} \leq c\delta^{-\frac{p-1}{p(p-1+\eps)}} \left[\dashint_{T_i}\left(\frac{\left|A_i(t)\right|}{\left|B_i\right|}\right)^{\frac{p-1+\eps}{\eps}}\d{t}\right]^{\nicefrac{\eps}{p(p-1+\eps)}}
\]
with $c\equiv c(d,s,p)<\infty$. Hence, letting
\[
\boldsymbol{Z}_i:=\left[\dashint_{T_i}\left(\frac{\left|A_i(t)\right|}{\left|B_i\right|}\right)^{\frac{p-1+\eps}{\eps}}\d{t}\right]^{\frac{\eps}{(1+\lambda)(p-1+\eps)}} \leq 1 
\]
where $\lambda>0$ is to be specified later,~\eqref{eq:dg-est-5} leads to
\begin{equation*}
\boldsymbol{Y}_{\!i+1} \leq c 4^{i(d+sp)} \delta^{-\frac{1}{p\kappa}-\frac{p-1}{p(p-1+\eps)}}\left(\boldsymbol{Y}_{\!i}^{1+\beta}+\boldsymbol{Y}_{\!i}^\beta\cdot \boldsymbol{Y}_{\!i}^{\frac{p-1}{p}}\boldsymbol{Z}_i^{\frac{1+\lambda}{p}}\right)
\end{equation*}
where $c\equiv c(\data)<\infty$ and $\beta:=\frac{1}{p}-\frac{1}{p\kappa}>0$. Thanks to the H\"{o}lder inequality, we have
\[
\frac{|A_i|}{|Q_i|} = \dashint_{T_i}\frac{|A_i(t)|}{|B_i|}\d{t} \leq \left[\dashint_{T_i} \left(\frac{|A_i(t)|}{|B_i|}\right)^{\frac{p-1+\eps}{\eps}}\d{t}\right]^{\nicefrac{\eps}{(p-1+\eps)}}
\]
and therefore, 
\[
\boldsymbol{Y}_i^{\frac{p-1}{p}} \leq \boldsymbol{Z}_i^{\frac{p-1}{p}(1+\lambda)}.
\]
As a consequence, we have shown that
\begin{equation}\label{eq:dg-est-6}
\boldsymbol{Y}_{\!i+1} \leq c\boldsymbol{b}^i\left(\boldsymbol{Y}_{\!i}^{1+\beta}+\boldsymbol{Y}_{\!i}^\beta\boldsymbol{Z}_i^{1+\lambda}\right) \quad \mbox{with \,$\boldsymbol{b}=4^{d+sp}$.}
\end{equation}
Note that the free parameter $\lambda>0$ is still to be selected.
\smallskip

\emph{Step 3.}\quad In this final step, we will deduce a recursive inequality for $\{\boldsymbol{Z}_{\!i}\}_{i \in \N_0}$. Firstly, observe that
\begin{align}\label{eq:dg-est-7}
\left|A_{i+1}(t)\right|&=\left|\left\{u(\cdot,t) <k_{i+1} \right\} \cap B_{i+1}\right| \notag\\
&\leq \int_{A_{i+1}(t)}\left(\frac{\widetilde{k}_i-u}{\widetilde{k}_i-k_{i+1}}\right)^{p(1+\lambda)}\d{x} \notag\\
&\leq \left(\frac{2^{i+3}}{k}\right)^{p(1+\lambda)}\int_{B_{i+1}}(u-\widetilde{k}_i)_-^{p(1+\lambda)}\d{x}.
\end{align}
Next, as we are going to apply Lemma~\ref{Lm:GN}, choose
\[
\widetilde{q}=p(1+\lambda), \quad \widetilde{r}=\frac{sp^2(1+\lambda)}{d\lambda}, \quad m=p,
\]
which ensures the condition in Lemma~\ref{Lm:GN}. We then determine the free parameter $\lambda>0$ so that
\[
\widetilde{r}=\frac{sp^2(1+\lambda)}{d\lambda}=\frac{p(1+\lambda)(p-1+\eps)}{\eps} \quad \iff \quad \lambda=\frac{\eps sp}{d(p-1+\eps)}.
\]
These selections yield that
\begin{align*}
\frac{d\widetilde{q}}{sp\widetilde{q}+d\widetilde{r}}=\frac{1}{\widetilde{r}}\cdot\frac{d\widetilde{q}\widetilde{r}}{sp\widetilde{q}+d\widetilde{r}}=\frac{\eps} {p(1+\lambda)(p-1+\eps)}\cdot p=\frac{\eps} {(1+\lambda)(p-1+\eps)}.
\end{align*}
We thus are allowed to apply Lemma~\ref{Lm:GN} and, use~\eqref{eq:dg-est-7} and~\eqref{d2} to obtain that
\begin{align}\label{eq:dg-est-8}
\boldsymbol{Z}_{i+1}&=\left[\frac{1}{|B_{i+1}|^{\frac{p-1+\eps}{\eps}}|T_{i+1}|}\int_{T_{i+1}}|A_{i+1}(t)|^{\frac{p-1+\eps}{\eps}}\d{t}\right]^{\frac{\eps}{(1+\lambda)(p-1+\eps)}} \notag\\[4pt]
&\!\!\!\!\stackrel{\eqref{eq:dg-est-7}}{\leq} \left[\frac{1}{|B_{i+1}|^{\frac{p-1+\eps}{\eps}}|T_{i+1}|}\left(\frac{2^{i+3}}{k}\right)^{\frac{p(1+\lambda)(p-1+\eps)}{\eps}}\int_{T_{i+1}}\left\|(u-\widetilde{k}_i)_-\right\|_{L^{p(1+\lambda)}(B_{i+1})}^{\frac{p(1+\lambda)(p-1+\eps)}{\eps}}\d{t}\right]^{\frac{\eps}{(1+\lambda)(p-1+\eps)}}\notag\\[4pt]
&\!\!\leq c \frac{2^{ip}}{k^p} \left(\frac{1}{|B_{i+1}|^{\frac{p-1+\eps}{\eps}}|T_{i+1}|}\right)^{\frac{\eps}{(1+\lambda)(p-1+\eps)}}\left[\int_{\widetilde{T}_i}[(u-\widetilde{k}_i)_-]_{W^{s,p}(\widetilde{B}_i)}^p\d{t}\right. \notag\\[4pt]
&\left. \quad \quad \quad  \quad \quad \quad  \quad \quad \quad\quad +(\widetilde{\rho}_i)^{-sp}\left\|(u-\widetilde{k}_i)_-\right\|_{L^p(\widetilde{Q}_i)}^{p} +\sup_{t_ \in \widetilde{T}_i}\left\|(u-\widetilde{k}_i)_-\right\|_{L^p(\widetilde{B}_i)}^{p}\right] \notag\\[4pt]
&\!\!\!\stackrel{\eqref{d3'}}{\leq} c \frac{2^{ip}}{k^p} \left(\frac{1}{|B_{i+1}|^{\frac{p-1+\eps}{\eps}}|T_{i+1}|}\right)^{\frac{\eps}{(1+\lambda)(p-1+\eps)}} \notag\\[4pt]
&\quad \quad \quad \quad\times \left\{\left[\frac{\rho_i^{(1-s)p}}{(\rho_i-\widetilde{\rho}_i)^{p}}+\frac{1}{\tau_i-\widetilde{\tau}_i} +(\widetilde{\rho}_i)^{-sp}\right]\int_{\widetilde{Q}_i}\widetilde{w}^p\dxt \right.\notag\\[4pt]
&\left. \quad \quad \quad \quad \quad \quad +\frac{\rho_i^{d}}{(\rho_i-\widetilde{\rho}_i)^{d+sp}}\int_{Q_i}\widetilde{w}\d{x}\Big[\tail(\widetilde{w}(t)\,; B_i)\Big]^{p-1}\d{t}\right\} \notag\\[4pt]
&\leq c \frac{2^{(1+s)pi}}{k^p} \left(\frac{1}{|B_{i+1}|^{\frac{p-1+\eps}{\eps}}|T_{i+1}|}\right)^{\frac{\eps}{p(1+\lambda)(p-1+\eps)}}\rho^{-sp}\Big[(\mathrm{I})+(\mathrm{II})\Big],
\end{align}
where to obtain the last line we used the definition of $(\mathrm{I})$ and $(\mathrm{II})$, appearing in Step 2. 

Now, a completely similar argument to Step 2 ensures that
\begin{align*}
(\mathrm{I})+(\mathrm{II}) &\leq \frac{c}{\delta}\Big[k^p|A_i|+(\mathrm{III})\Big]\\
&\leq \frac{c|Q_i|}{\delta}k^p\left[\frac{|A_i|}{|Q_i|} +\frac{\rho^{\frac{(p-1) sp}{p-1+\eps}}}{|Q_i|}\left(\int_{T_i}|A_i(t)|^{\frac{p-1+\eps}{\eps}}\d{t}\right)^{\nicefrac{\eps}{(p-1+\eps)}}\right]\\
&\leq \frac{c|Q_i|}{\delta}k^p\left[\frac{|A_i|}{|Q_i|} +\left[\dashint_{T_i}\left(\frac{|A_i(t)|}{|B_i|}\right)^{\frac{p-1+\eps}{\eps}}\d{t}\right]^{\nicefrac{\eps}{(p-1+\eps)}}\right]\\
&=\frac{c|Q_i|}{\delta}k^p \left(\boldsymbol{Y}_{\!i}+\boldsymbol{Z}_i^{1+\lambda}\right)
\end{align*}
with a constant $c\equiv c(\data)<\infty$. After substituting the last display into~\eqref{eq:dg-est-8}, we estimate that
\begin{align*}
\left(\frac{1}{|B_{i+1}|^{\frac{p-1+\eps}{\eps}}|T_{i+1}|}\right)^{\frac{\eps}{(1+\lambda)(p-1+\eps)}}\rho^{-sp}|Q_i|&\leq C(d,p,\lambda,\eps)\delta^{-\frac{\eps}{(1+\lambda)(p-1+\eps)}-1}\rho^{\frac{d\lambda}{1+\lambda}-\frac{\eps sp}{(1+\lambda)(p-1+\eps)}} \\
& =C(d,p,\lambda,\eps)\delta^{-\frac{\eps d}{d(p-1+\eps)+\eps sp}-1},
\end{align*}
where we used the fact that
\[
\lambda=\frac{\eps sp}{d(p-1+\eps)} \quad \iff \quad \frac{d\lambda}{1+\lambda}-\frac{\eps sp}{(1+\lambda)(p-1+\eps)}=0
\]
and
\[
-\frac{\eps}{(1+\lambda)(p-1+\eps)}-1=-\frac{\eps d}{d(p-1+\eps)+\eps sp}-1.
\]
As a consequence, we find
\begin{equation}\label{eq:dg-est-9}
\boldsymbol{Z}_{i+1} \leq c \boldsymbol{b}^i \delta^{-\frac{\eps d}{d(p-1+\eps)+\eps sp}-1}\left(\boldsymbol{Y}_{\!i}+\boldsymbol{Z}_i^{1+\lambda}\right) \quad \mbox{with \,$\boldsymbol{b}=4^p$.}
\end{equation}
Let
\[
\boldsymbol{K}_i:=\boldsymbol{Y}_{\!i}+\boldsymbol{Z}_i^{1+\lambda}, \quad \gamma:=\max \left\{\frac{1}{p\kappa}+\frac{p-1}{p(p-1+\eps)},\,\frac{\eps d}{d(p-1+\eps)+\eps sp}+1\right\}.
\]
Then, following the argument as in~\cite[Chapter I.4, Lemma 4.2]{DiB93}, we are able to obtain from~\eqref{eq:dg-est-6} and~\eqref{eq:dg-est-9} that
\[
\boldsymbol{K}_{i+1} \leq 2c^{1+\lambda} \delta^{-\gamma(1+\lambda)}\boldsymbol{b}^{(1+\lambda)i}\boldsymbol{K}_i^{1+\min\{\beta, \lambda\}}
\]
for positive constants $\boldsymbol{b}=4^p$ and $c>1$ depending only on $\data$. Note that both positive constants $\beta$ and $\lambda$ hinge upon $d,s,p$ and $\eps$. Hence, by  Lemma~\ref{Lm:FGC}, there exists a positive constant $\nu \in (0,1)$ depending only on $\data$, so that $\boldsymbol{K}_{\!\infty}=\lim_{i \to \infty}\boldsymbol{K}_{\!i}=0$ if we demand that $\boldsymbol{K}_0 \leq \nu$, which in turn leads to
\[
|A_0|=\left|\left\{u<k_0\right\} \cap Q_0\right|=\left|\left\{u<k\right\} \cap Q_{2\rho,2\tau}\right| \leq \nu\left|Q_{2\rho,2\tau}\right|.
\]
Since $\boldsymbol{K}_{\!\infty}=0$ we have that
\[
u \geq \frac{k}{2} \quad \mbox{a.e. in $Q_{\rho,\tau}$},
\]
as desired.
\end{proof}
\subsection{Shrinking lemma and propagation of measure information} \label{Shrink+Prop}
The following Lemma shows that is possible to propagate measure-theoretical information to future times.

\begin{lemma}\label{Lm:mt}
Let $u \in \mathbf{PDG}_-^{s,p}(\Omega_T, \gamma_{\mathrm{DG}}, \eps)$ in the sense of Definition~\ref{def of DG} be such that $u \geq 0$ in $\cQ_R \subset \Omega_T$. Suppose for some constants $k>0$ and $\alpha \in (0,1]$, there holds
\[
\Big|\left\{u(\cdot,t_0) \geq k \right\} \cap B_\rho(x_0)\Big| \geq \alpha |B_\rho(x_0)|
\]
then there exist constants $\eta$ and $\delta$ in $(0,1)$, both depending only on $\data$ and $\alpha$, such that either
\[
\left(\frac{\rho}{R}\right)^{\frac{sp\eps}{(p-1)(p-1+\eps)}} \left(\dashint_{I_R(t_0)} \Big[\tail \left(u_-(t)\,;B_R(x_0)\right)\Big]^{p-1+\eps}\d{t}\right)^{\nicefrac{1}{(p-1+\eps)}}>\eta k
\]
or
\[
\Big|\left\{u(\cdot,t) \geq \eta k \right\} \cap B_\rho(x_0)\Big| \geq \frac{\alpha}{2} |B_\rho(x_0)|
\]
for every $t \in (t_0,t_0+\delta \rho^{sp}]$, provided that  $B_\rho(x_0)\times (t_0,t_0+\delta \rho^{sp}] \subset \cQ_R \Subset \Omega_T$. More precisely, we can trace the dependences by
\[
\eta =\alpha/(8p) \quad \mbox{and} \quad \delta =\delta_0\alpha^{(d+p+1)\frac{p-1+\eps}{\eps}}
\]
for a constant $\delta_0 \equiv \delta_0(\data) \in (0,1)$.
\end{lemma}
\begin{proof}
The proof is quite similar to~\cite[Lemma 4.1]{Nak23} and~\cite[Lemma 3.3]{Lia24a}, however, we report the full proof for the readability. Let $(x_0,t_0)=(0,0)$ for simplicity. The proof is condensed in showing
\[
\Big|\left\{u(\cdot,t) <\eta k \right\} \cap B_\rho\Big| \leq \left(1-\frac{\alpha}{2}\right) |B_\rho|, \quad \forall t \in (0, \delta \rho^{sp}],
\]
where $\delta \in (0,1)$ is to be determined later, after enforcing 
\begin{equation}\label{eq:mt1-est-1}
\left(\frac{\rho}{R}\right)^{\frac{sp\eps}{(p-1)(p-1+\eps)}} \left(\dashint_{I_R} \Big[\tail \left(u_-(t)\,;B_R\right)\Big]^{p-1+\eps}\d{t}\right)^{\nicefrac{1}{(p-1+\eps)}} \leq \eta k.
\end{equation}
To see this, we use an adaptation of~\eqref{d3} over the concentric cylinders $(1-\sigma)Q \subset Q$, where $Q:=B_\rho \times (0,\delta \rho^{sp}]$ and its homothety $(1-\sigma)Q:=B_{(1-\sigma)\rho} \times \left(0,\delta \left((1-\sigma)\rho\right)^{sp}\right]$, whereas $\sigma \in (0,1)$ is to be selected later as well as $\delta$, thereby getting
\begin{align}\label{eq:mt1-est-2}
&\int_{B_{(1-\sigma)\rho}\times \{t\}}w_-^p\d{x}\notag\\
& \quad \leq \int_{B_{(1-\sigma)\rho}\times \{0\}}w_-^p\d{x}+\gamma_{\mathrm{DG}}\frac{\rho^{(1-s)p}}{(\sigma \rho)^p}\int_Qw_-^p\dxt \notag\\
& \quad \quad \quad \quad \quad +\gamma_\mathrm{DG} \frac{\rho^d}{(\sigma \rho)^{d+sp}}\int_Q w_-(x,t)\d{x}\Big[\tail\left(w_-(t)\,;B_\rho \right)\Big]^{p-1}\d{t} \notag\\
& \quad =:(\mathrm{I})+(\mathrm{II})+(\mathrm{III})
\end{align}
with denoting $w_-:=(u-k)_-$ shortly. Here we have discarded two fractional integrals on the left side of~\eqref{d3}. Firstly, we turn our attention to the right side of~\eqref{eq:mt1-est-2}. The estimates of $(\mathrm{I})$ and $(\mathrm{II})$ are quite standard. Indeed, due to the measure theoretical assumption at the initial time $t_0=0$ and the nonnegativity of $u$ in $\cQ$, we have
\[
(\mathrm{I}) \leq k^p\Big|\left\{u(\cdot,0) <\eta k \right\} \cap B_\rho\Big| \leq (1-\alpha)k^p|B_\rho|,
\]
whereas a straightforward computation gives
\[
(\mathrm{II})\leq \gamma_{\mathrm{DG}}\frac{\delta}{\sigma^p}k^p|B_\rho|.
\]
The next task is to estimate the tail term $(\mathrm{III})$. As seen in Step 2 in the proof of Lemma~\ref{Lm:DGtype}, via routine computations we can split it into two terms:
\[
\Big[\tail\left(w_-(t)\,;B_\rho \right)\Big]^{p-1} \leq C k^{p-1}+C \left(\frac{\rho}{R}\right)^{sp}\Big[\tail\left(u_-(t)\,;B_R \right)\Big]^{p-1}
\]
for a constant $C(d,s,p)<\infty$. Therefore, there holds that
\begin{align*}
(\mathrm{III})&\leq c \frac{\rho^d}{(\sigma \rho)^{d+sp}}\int_Q w_-(x,t)\d{x} \left[k^{p-1}+\left(\frac{\rho}{R}\right)^{sp}\Big[\tail\left(u_-(t)\,;B_R \right)\Big]^{p-1}\right]\d{t} \\
&\leq c\sigma^{-(d+sp)}\left[k^p\delta|B_\rho|+\frac{k|B_\rho|}{R^{sp}}\int_0^{\delta\rho^{sp}}\Big[\tail\left(u_-(t)\,;B_R \right)\Big]^{p-1}\d{t}\right],
\end{align*}
where $c$ depends only on $\data$. A combination of straightforward computation and H\"{o}lder's inequality with exponents $\left(\frac{p-1+\eps}{p-1},\frac{p-1+\eps}{\eps}\right)$, provide that
\begin{align*}
&\frac{k|B_\rho|}{R^{sp}}\int_0^{\delta\rho^{sp}}\Big[\tail\left(u_-(t)\,;B_R \right)\Big]^{p-1}\d{t} \\
& \quad \leq \gamma(p,\eps)k|B_\rho| \delta^{\frac{\eps}{p-1+\eps}}\left(\frac{\rho}{R}\right)^{\frac{\eps sp}{p-1+\eps}}\left(\dashint_{I_R}\Big[\tail\left(u_-(t)\,;B_R \right)\Big]^{p-1+\eps}\d{t}\right)^{\frac{p-1}{p-1+\eps}}.
\end{align*}
All in all, if we enforce~\eqref{eq:mt1-est-1} then we can merge as
\[
(\mathrm{III}) \leq c\sigma^{-(d+sp)}\left(\delta+\delta^{\frac{\eps}{p-1+\eps}}\right)k^p|B_\rho|.
\]

Next, we will estimate the left side of~\eqref{eq:mt1-est-2}. Set $A_{k,\rho}(t):=\left\{u(\cdot,t)<k\right\} \cap B_\rho$ for $k>0$. Since, by $\left|B_\rho \setminus B_{\sigma \rho}\right| \leq d\sigma \left|B_\rho\right|$,
\begin{align*}
\left|A_{\eta k, \rho}(t)\right|&=\left|A_{\eta k, (1-\sigma)\rho}(t) \cup \left(A_{\eta k, \rho}(t) \setminus A_{\eta k, (1-\sigma)\rho}(t)\right)\right|\\
&\leq \left|A_{\eta k, (1-\sigma)\rho}(t) \right|+d\sigma \left|B_\rho\right|
\end{align*}
we have that
\begin{align*}
\int_{B_{(1-\sigma)\rho}\times \{t\}}w_-^p\d{x} 
&\geq \left|A_{\eta k, (1-\sigma)\rho}(t)\right|(1-\eta)^pk^p
\\
&\geq \left(\left|A_{\eta k, \rho}(t)\right|-d\sigma \left|B_\rho\right|\right)(1-\eta)^pk^p.
\end{align*}
Collecting all the preceding estimates in~\eqref{eq:mt1-est-2} yields that
\begin{align*}
\left|A_{\eta k, \rho}(t)\right| &\leq \frac{1}{(1-\eta)^p}\left[(1-\alpha)+\frac{C(\delta+\delta^{\frac{\eps}{p-1+\eps}})}{\sigma^{\max\{p,d+sp\}}}\right]|B_\rho|+d\sigma|B_\rho|\\
&\leq \frac{1}{(1-\eta)^p}\left[(1-\alpha)+\frac{c\delta^{\frac{\eps}{p-1+\eps}}}{\sigma^{d+p}}+d\sigma\right]|B_\rho|
\end{align*}
for a constant $C\equiv C(d,s,p,\eps)<\infty$. The free parameters now are ready to be specified. In order, these are selected to satisfy
\[
\sigma=\frac{\alpha}{8d}, \quad \frac{c\delta^{\frac{\eps}{p-1+\eps}}}{\sigma^{d+p}} \leq \frac{\alpha}{4},\quad \textrm{and}\quad \eta=\frac{\alpha}{8p} \leq \frac{\alpha}{8p(1-\frac{\alpha}{2})} \leq 1-\left(\frac{1-\frac{5}{8}\alpha}{1-\frac{1}{2}\alpha}\right)^{\nicefrac{1}{p}}.
\]
Remark that this specifies the choice of $\delta=\delta_0 \alpha^{(d+p+1)\frac{p-1+\eps}{\eps}}$ for a constant $\delta_0$ depending only on $d,s,p$ and $\eps$. 
As a consequence, we obtain
\[
\left|A_{\eta k, \rho}(t)\right| \leq \left(1-\frac{\alpha}{2}\right)|B_\rho|, \quad \forall t \in (0,\delta\rho^{sp}],
\]
which implies the desired result. 
\end{proof}

Next, we deduce the so-called ``Measure shrinking lemma'', which is known as a indispensable tool of the regularity theory of parabolic $p$-Laplace equations. In~\cite[Lemma 4.2]{Nak23}, it is established  successfully, using the fractional discrete De Giorgi's type inequality~\cite[Lemma 2.1]{Nak23}. However, as mentioned in~\cite[Lemma 3.4]{Lia24a}, its proof would be much simpler, compared with conventional ones. So, we follow such a new approach that avoids the fractional discrete De Giorgi's type inequality.

\begin{lemma}\label{Lm:shrinking}
Let $u \in \mathbf{PDG}_-^{s,p}(\Omega_T, \gamma_{\mathrm{DG}}, \eps)$ in the sense of Definition~\ref{def of DG} be such that $u \geq 0$ in $\cQ_R \subset \Omega_T$ and let $k>0$ and $\alpha \in (0,1]$. Suppose, for some $\delta \in (0,1)$ and $\sigma \in (0,\nicefrac{1}{2})$, that
\begin{equation}\label{eq:shrinking-est-0}
\Big|\left\{u(\cdot,t) \geq k \right\} \cap B_\rho(x_0)\Big| \geq \alpha |B_\rho| \quad \mbox{for all\,\, $t \in (t_0-\delta\rho^{sp},t_0]$}.
\end{equation}
There exists $C<\infty$ depending only on $\data$ and independent of $\alpha$ and $\delta$, such that if for some $\sigma \in (0, \nicefrac{1}{2})$
\[
\left(\frac{\rho}{R}\right)^{\frac{sp\eps}{(p-1)(p-1+\eps)}} \left(\dashint_{I_R(t_0)} \Big[\tail \left(u_-(t)\,;B_R(x_0)\right)\Big]^{p-1+\eps}\d{t}\right)^{\nicefrac{1}{(p-1+\eps)}} \leq \sigma k
\]
happens, then
\[
\Big|\left\{ u \leq \sigma k \right\} \cap Q_{\rho,\tau}(z_0) \Big| \leq C\frac{\sigma^{p-1}}{\delta \alpha}|Q_{\rho,\tau}(z_0)|
\]
holds true whenever $Q_{2\rho,2\tau}(z_0) \subset \cQ_R \Subset \Omega_T$ with $\tau=\delta \rho^{sp}$.
\end{lemma}
\begin{proof}
By translation, we may let $(x_0,t_0)=(0,0)$. To begin with, set
\[
w_{\pm}:=\left(u-\sigma k\right)_\pm \quad \mbox{and} \quad Q_{\rho,\tau}:=B_\rho \times (-\tau, 0] \quad \mbox{with \,\,$\tau:=\delta\rho^{sp}$}. 
\]
We use~\eqref{d3} in the concentric cylinders $Q_{\rho, \tau} \subset Q_{2\rho,\tau}$ to obtain that
\begin{align*}
&\int_{Q_{\rho,\tau}}w_-(x,t)\d{x}\left(\int_{\R^d}\frac{w_+^{p-1}(y,t)}{|x-y|^{d+sp}}\d{y}\right)\d{t}\\
&\quad \leq \int_{B_{2\rho} \times \{0\}}w_-^p\d{x}+c\rho^{-sp}\int_{Q_{2\rho,\tau}}w_-^p\dxt \\
&\quad \quad \quad \quad +c\rho^{-sp} \int_{Q_{2\rho,\tau}}w_-\d{x}\Big[\tail\left(w_-(t)\,;B_{2\rho}\right)\Big]^{p-1}\d{t}\\
& \quad =:(\mathrm{I})+(\mathrm{II})+(\mathrm{III}).
\end{align*}
While the treatment of $(\mathrm{I})$ and $(\mathrm{II})$ are quite standard, estimating $(\mathrm{III})$ needs a careful inspection as before. Indeed, a straightforward computation yields that
\[
(\mathrm{I})+(\mathrm{II})\leq c\frac{(\sigma k)^p}{\delta \rho^{sp}}|Q_{\rho,\tau}|
\]
for a constant $c\equiv c(\data)<\infty$. The analogous arguments in the previous lemma, it holds that
\begin{align*}
(\mathrm{III}) &\leq c \Bigg[(\sigma k)^k|B_\rho|+(\sigma k)\delta^{\frac{\eps}{p-1+\eps}}|B_\rho| \\[2pt]
&\quad \quad \quad \quad \times \left(\frac{\rho}{R}\right)^{\frac{sp\eps}{(p-1+\eps)}} \left(\dashint_{I_R} \Big[\tail \left(u_-(t)\,;B_R\right)\Big]^{p-1+\eps}\d{t}\right)^{\frac{p-1}{p-1+\eps}}\Bigg].
\end{align*}
After enforcing
\[
\left(\frac{\rho}{R}\right)^{\frac{sp\eps}{(p-1)(p-1+\eps)}} \left(\dashint_{I_R} \Big[\tail \left(u_-(t)\,;B_R\right)\Big]^{p-1+\eps}\d{t}\right)^{\nicefrac{1}{(p-1+\eps)}} \leq \sigma k,
\]
we are able to obtain that
\[
(\mathrm{III}) \leq c(\sigma k)^p |B_\rho|\left[1+\delta^{\frac{\eps}{p-1+\eps}} \right] \leq c(\sigma k)^p \frac{|Q_{\rho,\tau}|}{\delta \rho^{sp}}.
\]
Merging the contents of three integrals $(\mathrm{I})$-$(\mathrm{III})$, this yields
\begin{equation}\label{eq:shrinking-est-1}
\int_{Q_{\rho,\tau}}w_-(x,t)\d{x}\left(\int_{\R^d}\frac{w_+^{p-1}(y,t)}{|x-y|^{d+sp}}\d{y}\right)\d{t} \leq c\frac{(\sigma k)^p}{\delta \rho^{sp}}|Q_{\rho,\tau}|.
\end{equation}
On the one hand, upon using the fact that $\sigma<1/2$ and $|x-y| \leq 2\rho$ for all $x, y \in B_\rho$, the left side of~\eqref{eq:shrinking-est-1} can be estimated as
 \begin{align*}
 &\int_{Q_{\rho,\tau}}w_-(x,t)\d{x}\left(\int_{\R^d}\frac{w_+^{p-1}(y,t)}{|x-y|^{d+sp}}\d{y}\right)\d{t}\\
 & \quad \geq \int_{Q_{\rho,\tau}}w_-(x,t)\1_{\{u(x,t) \leq \frac{1}{2}\sigma k\}}\d{x}\left(\int_{B_\rho} \frac{w_+^{p-1}(y,t)\1_{\{u(y,t)>k\}}}{|x-y|^{d+sp}}\d{y}\right)\d{t}\\
 & \quad \geq \int_{\{u \leq \frac{1}{2}\sigma k\} \cap Q_{\rho,\tau}}\left(\sigma k -\tfrac{1}{2}\sigma k\right)\d{x}\left(\int_{\{u(\cdot, t) \geq k\} \cap B_\rho} \frac{(k-\sigma k)^{p-1}}{(2\rho)^{d+sp}}\d{y}\right)\d{t} \\
 & \quad \geq \frac{k^p\alpha \sigma}{2^{p+d+sp}\rho^{sp}}|B_1|\left|\{u \leq \tfrac{1}{2}\sigma k\} \cap Q_{\rho,\tau} \right|,
 \end{align*}
 where to obtain the last line from the penultimate one, we used~\eqref{eq:shrinking-est-0}.
Combining the preceding estimates and, subsequently, adjusting relevant  constants properly, the proof is concluded.
 \end{proof}

\subsection{Power-Like Expansion of Positivity} Having at hands the above lemmas, we can deduce a refined expansion of positivity lemma that will propagate the measure theoretical information with a power-like dependence. So we will call it ``power-like expansion of positivity'' for short. As seen in~\cite[Theorem 1.2, Lemma 5.1]{Nak23} this can also be done through the fractional clustering lemma, at the price of forcing the restriction $sp \geq 1$. In contrast, our approach here is much simpler and holds for the full range of exponents; again we have been inspired by Liao~\cite{Lia24a, Lia24b}. 

\begin{lemma}\label{Lm:mt2}
Let $u \in \mathbf{PDG}_-^{s,p}(\Omega_T, \gamma_{\mathrm{DG}}, \eps)$  in the sense of Definition~\ref{def of DG} be such that $u \geq 0$ in $\cQ_R \subset \Omega_T$. Suppose, for some constants $k>0$ and $\alpha \in (0,1)$, that
\[
\Big|\left\{u(\cdot,t_0) \geq k \right\} \cap B_\rho(x_0)\Big| \geq \alpha |B_\rho(x_0)|.
\]
There exist constants $\eta_0 \in (0,1)$ and $\beta \in (1,\infty)$ both depending only on $\data$ and $\eps$, such that
\[
u \geq \eta_0\alpha^\beta k-\left(\frac{\rho}{R}\right)^{\frac{sp\eps}{(p-1)(p-1+\eps)}} \left(\dashint_{I_R(t_0)} \Big[\tail \left(u_-(t)\,;B_R(x_0)\right)\Big]^{p-1+\eps}\d{t}\right)^{\nicefrac{1}{(p-1+\eps)}}
\]
almost everywhere in $B_{2\rho}(x_0) \times \left(t_0+\frac{1}{2}(8\rho)^{sp}, t_0+2(8\rho)^{sp}\right]$, provided
\[
B_{4\rho}(x_0) \times (t_0,t_0+4(8\rho)^{sp}] \subset \cQ_R \subset \Omega_T.
\] 
\end{lemma}

\begin{proof}
For simplicity we may assume that $(x_0,t_0)=(0,0)$. By the assumption, it is clear that
\[
\Big|\left\{u(\cdot,0) \geq k \right\} \cap B_{4\rho}\Big| \geq \alpha 4^{-d} |B_{4\rho}|,
\]
and therefore, upon applying Lemma~\ref{Lm:mt}, we have
\[
\eta=\frac{\alpha}{8p}, \quad \delta=\delta_0 \alpha^{(d+p+1)\frac{p-1+\eps}{\eps}} \quad \mbox{for} \,\,\delta_0 \equiv \delta_0(\data) \in (0,1)
\]
so that
\begin{equation}\label{eq:mt2-est-1}
\Big|\left\{u(\cdot,t) \geq \eta k \right\} \cap B_{4\rho}\Big| \geq \frac{\alpha}{2}4^{-d} |B_{4\rho}|, \quad \forall t \in (0,\delta (4\rho)^{sp}],
\end{equation}
provided that
\begin{equation}\label{eq:mt2-est-2}
\left(\frac{\rho}{R}\right)^{\frac{sp\eps}{(p-1)(p-1+\eps)}} \left(\dashint_{I_R} \Big[\tail \left(u_-(t)\,;B_R\right)\Big]^{p-1+\eps}\d{t}\right)^{\nicefrac{1}{(p-1+\eps)}} \leq \eta k.
\end{equation}
For every $\bar{t} \in \left(\frac{1}{2}\delta(4\rho)^{sp}, \delta(4\rho)^{sp}\right]$ let us consider the sifted cylinder
\[
Q^{(\nicefrac{\delta}{2})}_{4\rho}(0,\bar{t}\,):=B_{4\rho} \times \left(\bar{t}-\frac{\delta}{2}(4\rho)^{sp},\,\bar{t}\,\right].
\]
By construction, it is easy to check that the inclusion $Q^{(\nicefrac{\delta}{2})}_{4\rho}(0,\bar{t}\,) \subset B_{4\rho} \times (0,\delta (4\rho)^{sp}]$. Thus, the measure theoretical information~\eqref{eq:mt2-est-1} allows us to apply Lemma~\ref{Lm:DGtype} in the sifted cylinder $Q^{(\nicefrac{\delta}{2})}_{4\rho}(0,\bar{t}\,)$, with $k$ and $\alpha$ replaced by $\eta k$ and $\frac{\alpha}{2}4^{-d}$, respectively.
\noindent Let $\nu=\nu_0\delta^q$ be obtained in Lemma~\ref{Lm:DGtype}, where $\nu_0 \in (0,1)$ and $q>1$, both depending only on $\data$ and $\eps$. Having this $\nu$ at hand, thanks to Lemma~\ref{Lm:shrinking}, we can select $\sigma$ so that
\[
C \frac{\sigma^{p-1}}{\delta \frac{\alpha}{2}4^{-d}} \leq \nu=\nu_0\delta^q \quad \textrm{with} \quad \delta=\delta_0\alpha^{(d+p+1)\frac{p-1+\eps}{\eps}},
\]
for a constant $C\equiv C(\data, \eps)<\infty$, that is,
\begin{equation}\label{eq:mt2-est-3}
\sigma=\left(\frac{\nu \delta \alpha}{4^{d+1}C}\right)^{\nicefrac{1}{(p-1)}}=\left(\frac{\nu_0 \delta_0^{q+1} \alpha^{(q+1)\left[(d+p+1)\frac{p-1+\eps}{\eps}\right]}}{4^{d+1}C}\right)^{\nicefrac{1}{(p-1)}}.
\end{equation}
Thus, we are allowed to apply Lemma~\ref{Lm:shrinking} to obtain that
\[
\Big|\left\{ u \leq \sigma \eta k \right\} \cap Q^{(\nicefrac{\delta}{2})}_{4\rho}(0,\bar{t}\,) \Big| \leq C\frac{\sigma^{p-1}}{\delta \frac{\alpha}{2}4^{-d}}\left|Q^{(\nicefrac{\delta}{2})}_{4\rho}(0,\bar{t}\,)\right|
\]
and, subsequently, appealing to Lemma~\ref{Lm:DGtype}, we have
\begin{equation}\label{eq:mt2-est-4}
u \geq \frac{1}{2}\sigma \eta k \quad \textrm{a.e. \,\,in} \,\,\, B_{2\rho} \times \left(\bar{t}-\frac{\delta}{4}(4\rho)^{sp}, \bar{t}\,\right]
\end{equation}
provided that
\begin{equation}\label{eq:mt2-est-5}
\left(\frac{\rho}{R}\right)^{\frac{sp\eps}{(p-1)(p-1+\eps)}} \left(\dashint_{I_R} \Big[\tail \left(u_-(t)\,;B_R\right)\Big]^{p-1+\eps}\d{t}\right)^{\nicefrac{1}{(p-1+\eps)}} \leq \frac{1}{2} \sigma\eta k.
\end{equation}
Since $\bar{t} \in \left(\frac{1}{2}\delta(4\rho)^{sp}, \delta(4\rho)^{sp}\right]$ is arbitrary, the above display~\eqref{eq:mt2-est-4} leads to
\begin{equation}\label{eq:mt2-est-6}
u \geq \frac{1}{2}\sigma \eta k \quad \textrm{a.e. \,\,in} \,\,\, B_{2\rho} \times \left(\frac{1}{2}\delta(4\rho)^{sp}, \delta(4\rho)^{sp}\right]
\end{equation}
and, note that by~\eqref{eq:mt2-est-5} the condition~\eqref{eq:mt2-est-2} is automatically satisfied. Thus, from now on, it suffices to demand~\eqref{eq:mt2-est-5}.

Having at hands the preceding point-wise estimate~\eqref{eq:mt2-est-6}, we can repeat the above procedure with $\alpha=1$. This means that the pop-up parameters appearing hereafter are independent of $\alpha$, and, as implemented in~\cite[Proposition 5.3]{Nak23}, they can be taken as same values. Indeed, upon using the above pointwise estimate~\eqref{eq:mt2-est-6}, we demonstrate the procedure to find that there exist small numbers $\bar{\eta}, \bar{\delta}\in (0,\nicefrac{1}{2})$, both depending only on $\data$ and independent of $\alpha$, such that if
\begin{equation}\label{eq:mt2-est-7}
\left(\frac{\rho}{R}\right)^{\frac{sp\eps}{(p-1)(p-1+\eps)}} \left(\dashint_{I_R} \Big[\tail \left(u_-(t)\,;B_R\right)\Big]^{p-1+\eps}\d{t}\right)^{\nicefrac{1}{(p-1+\eps)}} \leq \bar{\eta}\cdot \frac{1}{2} \sigma\eta k
\end{equation}
then
\begin{equation}\label{eq:mt2-est-8}
u \geq \bar{\eta}\cdot \frac{1}{2}\sigma \eta k \quad \textrm{a.e. \,\,in} \,\,\, B_{2\rho} \times \left(\frac{1}{2}\delta(4\rho)^{sp}+\frac{1}{2}\bar{\delta}(8\rho)^{sp}, \delta(4\rho)^{sp}+\bar{\delta}(8\rho)^{sp}\right].
\end{equation}
Note that~\eqref{eq:mt2-est-5} is automatically confirmed because we have enforced~\eqref{eq:mt2-est-7} and thus, we are dealing with the condition~\eqref{eq:mt2-est-7}.

Let $n \in \N_{\geq 4}$ be such that $n=\lceil 2/\bar{\delta}\,\rceil$. Repeating this procedure, we infer that
\begin{equation}\label{eq:mt2-est-9}
u \geq \bar{\eta}^n\cdot \frac{1}{2}\sigma \eta k \quad \textrm{a.e. \,\,in} \,\,\, B_{2\rho} \times \left(\frac{1}{2}\delta(4\rho)^{sp}+\frac{1}{2}\bar{\delta}(8\rho)^{sp}, \delta(4\rho)^{sp}+n\bar{\delta}(8\rho)^{sp}\right]
\end{equation}
provided that
\begin{equation}\label{eq:mt2-est-10}
\left(\frac{\rho}{R}\right)^{\frac{sp\eps}{(p-1)(p-1+\eps)}} \left(\dashint_{I_R} \Big[\tail \left(u_-(t)\,;B_R\right)\Big]^{p-1+\eps}\d{t}\right)^{\nicefrac{1}{(p-1+\eps)}} \leq \bar{\eta}^n\cdot \frac{1}{2} \sigma\eta k.
\end{equation}
As explained above, we may enforce~\eqref{eq:mt2-est-7} replaced by~\eqref{eq:mt2-est-10}. Altogether, by~\eqref{eq:mt2-est-9} and~\eqref{eq:mt2-est-3}, together with $\eta=\nicefrac{\alpha}{8p}$, we conclude that
\[
u \geq \bar{\eta}^{\lceil \nicefrac{2}{\bar{\delta}}\,\rceil}\cdot \frac{1}{2}\sigma \eta k= \frac{1}{16p}\bar{\eta}^{\lceil \nicefrac{2}{\bar{\delta}}\,\rceil}\alpha\left(\frac{\nu_0 \delta_0^{q+1} \alpha^{(q+1)\left[(d+p+1)\frac{p-1+\eps}{\eps}\right]}}{4^{d+1}C}\right)^{\nicefrac{1}{(p-1)}}k
\]
almost everywhere in $B_{2\rho} \times \left(\frac{1}{2}(8\rho)^{sp}, 2(8\rho)^{sp}\right]$ since by the inclusion
\[
\left(\frac{1}{2}(8\rho)^{sp}, 2(8\rho)^{sp}\right] \subset \left(\frac{1}{2}\delta(4\rho)^{sp}+\frac{1}{2}\bar{\delta}(8\rho)^{sp}, \delta(4\rho)^{sp}+n\bar{\delta}(8\rho)^{sp}\right].
\]
Finally, letting
\begin{equation*}
\left\{
    \begin{array}{ll}
\displaystyle \beta:=1+\frac{q+1}{p-1}(d+p+1)\frac{p-1+\eps}{\eps}>1, \\[12pt]
\displaystyle \eta_0:=\frac{1}{16p}\bar{\eta}^{\lceil \nicefrac{2}{\bar{\delta}}\,\rceil}\left(\frac{\nu_0 \delta_0^{q+1}}{4^{d+1}C}\right)^{\nicefrac{1}{(p-1)}} \in (0,1),
    \end{array}
\right.
\end{equation*}
both depending only on $\data$ and $\eps$, and independent of $\alpha$, leads to the desired conclusion.
\end{proof}

\subsection{Proof of Theorem~\ref{Thm:weakHarnack1}}\label{Sect.4.2}
\begin{proof}[Proof of Theorem~\ref{Thm:weakHarnack1}]
The proof is almost a verbatim repetition of~\cite[Proof of Theorem 1.2]{Nak23}, using Lemma~\ref{Lm:mt2} and the layer cake formula. This concludes the proof.
\end{proof}

\section{Proof of Theorem \ref{Thm:weakHarnack2}}\label{Sect.5}

\subsection{Measure theoretical lemmas reloaded}\label{Sect.5.1}
In this subsection, we refine our measure theoretical analysis to bridge the gap between the previous expansion of positivity and the nonlocal strong Harnack inequality. While Section~\ref{Sect.4.1} provides a power-like expansion of positivity, the nonlocal strong Harnack inequality requires a more robust propagation of measure. We develop ``reloaded'' versions of the measure theoretical lemmas, specifically tailored to capture the quantitative interaction between the local average of the solution and the long-range data. These lemmas serve as the final key ingredient to the proof of~Theorem \ref{Thm:weakHarnack2}, eventually, Theorem~\ref{Thm:fullHarnack}.

\begin{lemma}\label{Lm:newmt1}
Fix the cylinder $Q_\rho(z_0) \subset \cQ_R$ for $z_0=(x_0,t_0)\in \Omega_T$. Let $u \in \mathbf{PDG}_-^{s,p}(\Omega_T, \gamma_{\mathrm{DG}}, \eps)$ in the sense of Definition~\ref{def of DG} be such that $u \geq 0$ in $\cQ_R \subset \Omega_T$ and let $k>0$ be an arbitrary number. Suppose that
\[
\left(\frac{\rho}{R}\right)^{\frac{sp\eps}{(p-1)(p-1+\eps)}} \left(\dashint_{I_R(t_0)} \Big[\tail \left(u_-(t)\,;B_R(x_0)\right)\Big]^{p-1+\eps}\d{t}\right)^{\nicefrac{1}{(p-1+\eps)}} \leq k .
\]
Then, there exists a constant $C(d,s,p,\gamma_{\mathrm{DG}})<\infty$ such that
\[
\inf_{t \in (t_0-\rho^{sp},t_0]}\Big|\left\{u(\cdot,t) \leq k \right\} \cap B_\rho(x_0) \Big| \leq \frac{C k^{p-1}}{(u^{p-1})_{Q_\rho(z_0)}}|B_\rho(x_0)|.
\]
\end{lemma}
\begin{proof}
By translation, we may let $(x_0,t_0)=(0,0)$. Let $w_\pm:=(u-k)_\pm$ for $k>0$. Upon defining 
\[
\mathbf{I}_{E}:=\int_{Q_\rho}w_-(x,t)\d{x}\left(\int_{E} \frac{w_+^{p-1}(y,t)}{|x-y|^{d+sp}}\d{y}\right)\d{t} \quad \mbox{for $E\subset \R^d$},
\]
an exactly same argument as Lemma~\ref{Lm:shrinking} ensures that
\begin{align}\label{eq:nmt1-est-1}
\mathbf{I}_{B_\rho}+\mathbf{I}_{\R^d \setminus B_\rho}&=\int_{Q_\rho}w_-(x,t)\d{x}\left(\int_{\R^d} \frac{w_+^{p-1}(y,t)}{|x-y|^{d+sp}}\d{y}\right)\d{t} \notag\\
&\leq \int_{B_{\rho} \times \{0\}}w_-^p\d{x}+c\rho^{-sp}\int_{Q_{2\rho}}w_-^p\dxt \notag\\
&\quad \quad +\gamma_\mathrm{DG} \frac{(2\rho)^d}{\rho^{d+sp}}\int_{Q_{2\rho}} w_-(x,t)\d{x}\Big[\tail\left(w_-(t)\,;B_\rho \right)\Big]^{p-1}\d{t} \notag\\
&\leq c\frac{k^p}{\rho^{sp}}|Q_\rho|=ck^p|B_\rho|
\end{align}
for a constant $c \equiv c(d,s,p,\gamma_{\mathrm{DG}})<\infty$, where to obtain the last line we have enforced that
\[
\left(\frac{\rho}{R}\right)^{\frac{sp\eps}{(p-1)(p-1+\eps)}} \left(\dashint_{I_R} \Big[\tail \left(u_-(t)\,;B_R\right)\Big]^{p-1+\eps}\d{t}\right)^{\nicefrac{1}{(p-1+\eps)}} \leq k.
\]
Clearly, $\mathbf{I}_{\R^d \setminus B_\rho}$ is nonnegative and notice that
\[
u_+^{p-1} \leq c(p)\left[(u-k)_+^{p-1}+k^{p-1}\right] \quad \mbox{and} \quad |x-y| \leq 2\rho \quad \forall x,y \in B_\rho. 
\]
Upon using these information, together with $u \geq 0$ in $Q_\rho \subset \cQ_R$, the left-hand side is estimated as
\begin{align*}
\mathbf{I}_{B_\rho}&+\mathbf{I}_{\R^d \setminus B_\rho} \geq \mathbf{I}_{B_\rho}\\
&\geq \inf_{t \in (-\rho^{sp},0]} \int_{B_\rho}w_-(x,t)\d{x} \left(\int_{Q_\rho}\frac{w_+^{p-1}(y,t)}{(2\rho)^{d+sp}}\d{y}\right)\d{t}\\
&\geq c\inf_{t \in (-\rho^{sp},0]} \int_{B_\rho}w_-(x,t)\d{x}  \Big[(u^{p-1})_{Q_\rho}-k^{p-1}\Big] \\
& \geq c(u^{p-1})_{Q_\rho} \inf_{t \in (-\rho^{sp},0]} \int_{\{u(\cdot, t) \leq \frac{k}{2}\} \cap B_\rho}w_-(x,t)\d{x}-ck^p|B_\rho| \\
&\geq c(u^{p-1})_{Q_\rho} \inf_{t \in (-\rho^{sp},0]} \Big|\left\{u(\cdot, t) \leq \tfrac{k}{2}\right\} \cap B_\rho \Big|\cdot \tfrac{k}{2}-ck^p|B_\rho|
\end{align*}
for a constant $c\equiv c(d,s,p)<\infty$. Altogether, we have proved that
\[
\inf_{t \in (-\rho^{sp},0]} \Big|\left\{u(\cdot, t) \leq \tfrac{k}{2}\right\} \cap B_\rho \Big| \leq \frac{Ck^{p-1}}{(u^{p-1})_{Q_\rho}}|B_\rho|.
\]
Finally, we adjust the relevant constants, thereby getting the desired estimate.
\end{proof}

As a by product of the proof of Lemma~\ref{Lm:newmt1}, we are able to obtain the following:

\begin{lemma}\label{Lm:newmt2}
Fix the cylinder $Q_\rho(z_0) \subset \cQ_R$ for $z_0=(x_0,t_0)\in \Omega_T$. Let $u \in \mathbf{PDG}_-^{s,p}(\Omega_T, \gamma_{\mathrm{DG}}, \eps)$  in the sense of Definition~\ref{def of DG} be such that $u \geq 0$ in $\cQ_R \subset \Omega_T$ and let $k>0$ be an arbitrary number. Suppose that
\[
\left(\frac{\rho}{R}\right)^{\frac{sp\eps}{(p-1)(p-1+\eps)}} \left(\dashint_{I_R(t_0)} \Big[\tail \left(u_-(t)\,;B_R(x_0)\right)\Big]^{p-1+\eps}\d{t}\right)^{\nicefrac{1}{(p-1+\eps)}} \leq k .
\]
Then, there exists a positive constant $C(d,s,p,\gamma_{\mathrm{DG}})<\infty$ such that
\[
\inf_{t \in (t_0-\rho^{sp},t_0]}\Big|\left\{u(\cdot,t) \leq k \right\} \cap B_\rho(x_0) \Big| \leq \frac{C k^{p-1}}{\displaystyle \dashint_{t_0-\rho^{sp}}^{t_0}\Big[\tail \left(u_+(t)\,;B_\rho(x_0)\right)\Big]^{p-1}\d{t}}|B_\rho(x_0)|.
\]
\end{lemma}
\begin{proof}
We may let $(x_0,t_0)=(0,0)$ without loss of generality. The proof departs from~\eqref{eq:nmt1-est-1}, namely, we require a careful estimate of $\mathbf{I}_{\R^d \setminus B_\rho}$ whereas $\mathbf{I}_{B_\rho} \geq 0$ is discarded. Since $|x-y| \leq 2|y|$ for all $x \in B_\rho$ and $y \in \R^d \setminus B_\rho$, upon using the estimate  $u_+^{p-1} \leq c(p)\left[(u-k)_+^{p-1}+k^{p-1}\right]$, we deduce that
\begin{align*}
\mathbf{I}_{B_\rho}&+\mathbf{I}_{\R^d \setminus B_\rho}
\geq \mathbf{I}_{\R^d \setminus B_\rho}\\
&\geq \frac{1}{2^{d+sp}}\inf_{t \in (-\rho^{sp},0]} \int_{B_\rho}w_-(x,t)\d{x} \left(\int_{-\rho^{sp}}^0\int_{\R^d \setminus B_\rho}\frac{w_+^{p-1}(y,t)}{|y|^{d+sp}}\dyt\right)\\
& \geq \frac{c(d,s,p)}{2^{d+sp}}\inf_{t \in (-\rho^{sp},0]} \int_{\{u(\cdot, t) \leq \frac{k}{2}\} \cap  B_\rho}w_-(x,t)\d{x} \\
&\quad \quad \quad\quad \quad \quad \quad\quad  \times \left(\dashint_{-\rho^{sp}}^0\Big[\tail (u_+(t) \,; B_\rho)\Big]^{p-1}\d{t}-k^{p-1}\right)\\
&\geq \frac{c(d,s,p)}{2^{d+sp}} \left[\inf_{t \in (-\rho^{sp},0]} \left|\left\{u(\cdot, t) \leq \frac{k}{2}\right\} \cap B_\rho \right|\cdot \frac{k}{2} \right.
 \\
&\quad \quad \quad\quad \quad \quad \quad\quad  \left.\times \left(\dashint_{-\rho^{sp}}^0\Big[\tail (u_+(t) \,; B_\rho)\Big]^{p-1}\d{t}\right)-k^p|B_\rho| \right].
\end{align*}
Altogether, we have shown that
\[
\inf_{t \in (-\rho^{sp},0]}\Big|\left\{u(\cdot,t) \leq \tfrac{k}{2} \right\} \cap B_\rho \Big| \leq \frac{Ck^{p-1}}{\displaystyle \dashint_{-\rho^{sp}}^{0}\Big[\tail \left(u_+(t)\,;B_\rho\right)\Big]^{p-1}\d{t}}|B_\rho|.
\]
Finally, adjusting the relevant constants conclude the proof.
\end{proof}
As a final key ingredient of Theorem~\ref{Thm:weakHarnack2}, we will deduce the following measure theoretical lemma which holds under weaker condition compared with Lemma~\ref{Lm:DGtype}. Namely, it asserts, under the measure density condition in a spatial ball at a time slice value, that the pointwise positivity still holds at later times. Notice that the assumption required here is weaker than the one in Lemma~\ref{Lm:DGtype}.

\begin{lemma}\label{Lm:weakDGtype} Let $u \in \mathbf{PDG}_-^{s,p}(\Omega_T,\gamma_{\mathrm{DG}}, \eps)$, in the sense of Definition~\ref{def of DG}, satisfy $u \geq 0$ in $\mathcal{Q}_R$ and $\eps>0$. There exist $\delta, \nu_\ast \in (0,1)$, both depending only on $\data$ and $\eps$, such that, if
\[
\left|\{u(\cdot,t_0)<k\} \cap B_{2\rho}(x_0)\right| \leq \nu_\ast \left|B_{2\rho}(x_0)\right|
\]
and
\[
\left(\frac{\rho}{R}\right)^{\frac{sp\eps}{(p-1)(p-1+\eps)}} \left(\dashint_{I_R(t_0)}\Big[\tail(u_-(t)\,; B_R(x_0))\Big]^{p-1+\eps}\d{t}\right)^{\nicefrac{1}{(p-1+\eps)}} \leq \frac{k}{4}
\]
then 
\[
u \geq \frac{k}{4} \quad \textrm{a.e.\,\,in}\,\, \,B_{\frac{\rho}{2}}(x_0) \times \left(t_0+\frac{1}{2}\delta \rho^{sp}, t_0+\delta\rho^{sp}\right],
\]
provided that $B_{2\rho}(x_0) \times \left(t_0,t_0+\delta\rho^{sp}\right] \subset \cQ_R \Subset \Omega_T$. 
\end{lemma}
\begin{proof}
The proof is similar to that of Lemma~\ref{Lm:DGtype}. We may let $(x_0,t_0)=(0,0)$ for simplicity. 
We split the argument into several steps.

\smallskip

\emph{Step 1: Set up.} We set preparations as follows: Let us define shrinking families of cylinders $Q_i:=B_i \times (0,\delta\rho^{sp}]$ and $\widetilde{Q}_i:=\widetilde{B}_i \times (0,\delta\rho^{sp}]$ for $i \in \N_0$, where 
\[
\rho_i:=(1+2^{-i})\rho, \quad \widetilde{\rho}_i:=\frac{1}{2}(\rho_i+\rho_{i+1})
\]
and we shortened $B_i:=B_{\rho_i}$ and $\widetilde{B}_i:=B_{\widetilde{\rho}_i}$. Here, $\delta \in (0,1]$ is to be chosen later. It is clear that these cylinders satisfy $Q_{i+1} \subset \widetilde{Q}_i \subset Q_i$ for every $i \in \N_0$. For $k>0$ to be fixed, we define decreasing sequences of levels in the following:
\[
k_i:=\frac{1}{2}k+\frac{1}{2^{i+1}}k, \quad \widetilde{k}_i:=\frac{1}{2}(k_i+k_{i+1}).
\]
Finally, set
\[
A_i:=\left\{u<k_i\right\} \cap Q_i, \quad \boldsymbol{Y}_{\!\!i}:=\frac{|A_i|}{|Q_i|} \leq 1, \quad \widetilde{w}:=(u-\widetilde{k}_i)_-.
\]
With these preparations, having at hand~\eqref{eq:dg-est-1}, the proof departs from~\eqref{eq:dg-est-d}. This time we use~\eqref{d3} over the cylinders $\widetilde{Q}_i$ and $Q_i$. Let $\varphi \in C^\infty(Q_i; [0,1])$ be the same cutoff function as in the proof of Lemma~\ref{Lm:DGtype}, Step 2. Using~\eqref{d3}, a straightforward computation gives that 
\begin{align}\label{eq:weakDGtype-est-1}
&\rho^{sp} \int_{\widetilde{T}_i}\int_{\widetilde{B}_i}\int_{\widetilde{B}_i} \frac{\left|\varphi \widetilde{w}(x,t)-\varphi \widetilde{w}(y,t)\right|^p}{|x-y|^{d+sp}}\dxt \notag\\
&\quad \leq \rho^{sp}\int_{B_i}\widetilde{w}^p(\cdot, 0)\d{x} +c2^{ip}\int_{Q_i}\widetilde{w}^p\dxt+c2^{ip}\int_{Q_i}\widetilde{w}\d{x}\Big[\tail(\widetilde{w}(t)\,; B_i)\Big]^{p-1}\d{t} \notag\\
& \quad =:(\mathrm{O})+(\mathrm{I})+(\mathrm{II})
\end{align}
with the obvious meaning of $(\mathrm{O})$, $(\mathrm{I})$ and $(\mathrm{II})$. By the similar argument leading to~\eqref{eq:dg-est-2}, we have that
\begin{equation*}
\sup_{t \in (0,\delta\rho^{sp}]}\dashint_{\widetilde{B}_i}(\varphi \widetilde{w})^p(t)\d{x} \leq \frac{c2^{ip}}{|B_i|}\rho^{-sp}\left[(\mathrm{O})+(\mathrm{I})+(\mathrm{II})\right].
\end{equation*}
Combining this estimate with~\eqref{eq:dg-est-1} and~\eqref{eq:weakDGtype-est-1}, we deduce that
\begin{align}\label{eq:weakDGtype-est-2}
\frac{k}{2^{i+3}}\left|A_{i+1}\right| &\leq c \left[2^{i(d+sp)}\left[(\mathrm{O})+(\mathrm{I})+(\mathrm{II})\right]\right]^{\nicefrac{1}{p\kappa}} \notag\\
& \quad \quad \times\left[\frac{2^{ip}\rho^{-sp}}{|B_i|}\left[(\mathrm{O})+(\mathrm{I})+(\mathrm{II})\right]\right]^{\frac{1}{p\kappa}\cdot \frac{\kappa_\ast-1}{\kappa_\ast}}|A_i|^{1-\nicefrac{1}{p\kappa}} \notag\\
& \leq c2^{i\frac{d+sp}{p}} \left(\frac{\rho^{-sp}}{|B_i|}\right)^{\nicefrac{1}{p\kappa}\cdot \frac{\kappa_\ast-1}{\kappa_\ast}}\Big[(\mathrm{O})+(\mathrm{I})+(\mathrm{II})\Big]^{\nicefrac{1}{p}}|A_i|^{1-\nicefrac{1}{p\kappa}},
\end{align}
where $c =c(d,s,p,\gamma_{\mathrm{DG}})<\infty$. To obtain the last inequality we used $1+\frac{\kappa_\ast-1}{\kappa_\ast}=\kappa$. The two integrals $(\mathrm{I})$ and $(\mathrm{II})$ can be estimated analogously to the proof of Lemma~\ref{Lm:DGtype}, while for the first one $(\mathrm{O})$ requires careful inspections. Indeed, carrying out an analogous argument to the proof of Lemma~\ref{Lm:DGtype}, we obtain that
\begin{equation*}
(\mathrm{I}) \leq \frac{k^p}{\delta}|A_i| \quad \mbox{and} \quad 
(\mathrm{II}) \leq ck^p|A_i|+c\rho^{\frac{(p-1)sp}{p-1+\eps}}k^p \left(\int_0^{\delta\rho^{sp}}|A_i(t)|^{\frac{p-1+\eps}{\eps}}\d{t}\right)^{\nicefrac{\eps}{(p-1+\eps)}}
\end{equation*}
for a constant $c\equiv c(\data)<\infty$, where to shorten the notation, we denoted
\[
A_i(t):=\{u(\cdot,t)<k_i\} \cap B_i \quad \mbox{for\,\, $t \in [0,\delta\rho^{sp}]$}, 
\]
and enforced that
\begin{equation}\label{eq:weakDGtype-est-3}
\left(\frac{\rho}{R}\right)^{\frac{sp\eps}{(p-1)(p-1+\eps)}} \left(\dashint_{I_R}\Big[\tail(u_-(t)\,; B_R)\Big]^{p-1+\eps}\d{t}\right)^{\nicefrac{1}{(p-1+\eps)}} \leq \frac{k}{4}.
\end{equation}
Next, we turn our attention to $(\mathrm{O})$. Since $\rho < \rho_i \leq 2\rho$ and $\widetilde{k}_i \leq k$ for any $i \in \N_0$, we observe, by the measure theoretical assumption, that
\begin{align*}
|A_i(0)|=\left|\left\{u(\cdot,0) <k_i \right\} \cap B_i\right| &\leq \left|\left\{u(\cdot,0) <k \right\} \cap B_{2\rho}\right|\\
&\leq \nu_\ast\left|B_{2\rho}\right| \\
&\leq \nu_\ast2^d\left|B_i\right|=\frac{2^d\nu_\ast}{\delta\rho^{sp}}|Q_i|,
\end{align*}
where $\nu_\ast \in (0,1)$ is still stipulated later. Therefore, this allows us to estimate $(\mathrm{O})$ as follows:
\begin{equation*}
(\mathrm{O})=\rho^{sp}\int_{B_i}\widetilde{w}^p(\cdot, 0)\d{x} \leq \rho^{sp} k^p|A_i(0)| \leq k^p\frac{2^d\nu_\ast}{\delta}|Q_i|.
\end{equation*}
After joining the preceding estimates of $(\mathrm{O})$--$(\mathrm{II})$ with~\eqref{eq:weakDGtype-est-2}, we divide the both side of the resulting inequality by $|Q_{i+1}|$ to find that
\begin{align}\label{eq:weakDGtype-est-4}
\frac{\left|A_{i+1}\right|}{\left|Q_{i+1}\right|} &\leq c 4^{i(d+sp)} \frac{|Q_i|^{1+\nicefrac{1}{p}-\nicefrac{1}{p\kappa}}}{\left|Q_{i+1}\right|}\left(\frac{\rho^{-sp}}{|B_i|}\right)^{\frac{1}{p\kappa}\cdot \frac{\kappa_\ast-1}{\kappa_\ast}}|A_i|^{1-\nicefrac{1}{p\kappa}}\notag \\[4pt]
& \quad \quad \quad \quad \quad \times \left[\left(\frac{2^d\nu_\ast}{\delta}\right)^{\nicefrac{1}{p}}\frac{1}{\left|Q_{i}\right|^{1-\nicefrac{1}{p\kappa}}}+\frac{\left|A_{i}\right|^{\nicefrac{1}{p}}}{\delta^{\nicefrac{1}{p}}\left|Q_{i}\right|^{1+\nicefrac{1}{p}-\nicefrac{1}{p\kappa}}} \right. \notag\\[4pt]
&\left. \quad \quad \quad \quad \quad \quad \quad  \quad \quad \quad \quad+\frac{\rho^{\frac{(p-1) s}{p-1+\eps}}}{\left|Q_i\right|^{1+\nicefrac{1}{p}-\nicefrac{1}{p\kappa}}}\left(\int_0^{\delta \rho^{sp}}|A_i(t)|^{\frac{p-1+\eps}{\eps}}\d{t}\right)^{\frac{\eps}{p(p-1+\eps)}}\right] \notag\\[4pt]
&\leq c 4^{i(d+sp)} \frac{|Q_i|^{1+\nicefrac{1}{p}-\nicefrac{1}{p\kappa}}}{\left|Q_{i+1}\right|}\left(\frac{\rho^{-sp}}{|B_i|}\right)^{\frac{1}{p\kappa}\cdot \frac{\kappa_\ast-1}{\kappa_\ast}}\frac{1}{\delta^{\nicefrac{1}{p}}}\notag \\[4pt]
& \quad \quad \quad \quad \quad \times \left[(2^d\nu_\ast)^{\nicefrac{1}{p}}\left(\frac{\left|A_{i}\right|}{\left|Q_{i}\right|} \right)^{1-\nicefrac{1}{p\kappa}}+\left(\frac{\left|A_{i}\right|}{\left|Q_{i}\right|} \right)^{1+\nicefrac{1}{p}-\nicefrac{1}{p\kappa}} \right.\notag\\[4pt]
&\left.\quad \quad \quad \quad \quad \quad \quad \quad \quad \quad \quad +\left(\frac{\left|A_{i}\right|}{\left|Q_{i}\right|} \right)^{1-\nicefrac{1}{p\kappa}}\frac{\rho^{\frac{(p-1) s}{p-1+\eps}}}{\left|Q_i\right|^{\nicefrac{1}{p}}}\left(\int_0^{\delta \rho^{sp}}|A_i(t)|^{\frac{p-1+\eps}{\eps}}\d{t}\right)^{\frac{\eps}{p(p-1+\eps)}}\right]\notag\\[4pt]
&\leq c 4^{i(d+sp)} \delta^{-\nicefrac{1}{p\kappa}}\left[(2^d\nu_\ast)^{\nicefrac{1}{p}}\left(\frac{\left|A_{i}\right|}{\left|Q_{i}\right|} \right)^{1-\nicefrac{1}{p\kappa}}+\left(\frac{\left|A_{i}\right|}{\left|Q_{i}\right|} \right)^{1+\nicefrac{1}{p}-\nicefrac{1}{p\kappa}} \right.\notag\\[4pt]
&\left.\quad \quad \quad \quad \quad  +\left(\frac{\left|A_{i}\right|}{\left|Q_{i}\right|} \right)^{1-\nicefrac{1}{p\kappa}}\delta^{-\frac{p-1}{p(p-1+\eps)}} \left[\dashint_{0}^{\delta \rho^{sp}}\left(\frac{\left|A_i(t)\right|}{\left|B_i\right|}\right)^{\frac{p-1+\eps}{\eps}}\d{t}\right]^{\frac{\eps}{p(p-1+\eps)}}\right],
\end{align}
where to obtain the last line from the penultimate one, as seen before, we used that 
\[
\frac{|Q_i|^{1+\nicefrac{1}{p}-\nicefrac{1}{p\kappa}}}{\left|Q_{i+1}\right|}\left(\frac{\rho^{-sp}}{|B_i|}\right)^{\frac{1}{p\kappa}\cdot \frac{\kappa_\ast-1}{\kappa_\ast}} =\frac{\left|Q_i\right|}{\left|Q_{i+1}\right|} \left(\frac{\left|Q_i\right|}{\left|B_i\right|\rho^{sp}}\right)^{\nicefrac{1}{p}-\nicefrac{1}{p\kappa}} \leq c(d)\delta^{\nicefrac{1}{p}-\nicefrac{1}{p\kappa}}
\]
and 
\begin{align*}
\frac{\rho^{\frac{(p-1) s}{p-1+\eps}}}{\left|Q_i\right|^{\nicefrac{1}{p}}}\left(\int_{0}^{\delta \rho^{sp}}|A_i(t)|^{\frac{p-1+\eps}{\eps}}\d{t}\right)^{\frac{\eps}{p(p-1+\eps)}} \leq c\delta^{-\frac{p-1}{p(p-1+\eps)}} \left[\dashint_0^{\delta \rho^{sp}}\left(\frac{\left|A_i(t)\right|}{\left|B_i\right|}\right)^{\frac{p-1+\eps}{\eps}}\d{t}\right]^{\frac{\eps}{p(p-1+\eps)}}
\end{align*}
for a constant $c\equiv c(d,s,p)<\infty$. 

\smallskip

\emph{Step 2: Alternative argument.} In this step, we consider two cases. The first case is supposing, for some $i_0 \in \N_0$, that
\[
\frac{|A_{i_0}|}{|Q_{i_0}|}=\frac{\left|\left\{u<k_{i_0}\right\} \cap Q_{i_0}\right|}{|Q_{i_0}|} \leq \nu_\ast.
\]
This in turn gives
\[
\left|\left\{u <\tfrac{1}{2}k\right\} \cap Q_\rho^{\delta}\right| \leq \left|\left\{u \leq k_{i_0}\right\} \cap Q_{i_0}\right| \leq \nu_\ast\left|Q_{i_0}\right|  \leq 2^d\nu_\ast|Q_\rho^\delta|
\]
with setting $Q_\rho^\delta:=B_\rho \times (0,\delta\rho^{sp}]$ because $\rho<\rho_{i_0} \leq 2\rho$ and $k_{i_0} \geq \frac{1}{2}k$. We keep in mind that $Q_\rho^\delta$ is rewritten as a shifted cylinder, i.e., $Q_\rho^\delta=Q_{\rho,\tau}(0,\delta\rho^{sp})$ with $\tau=\delta \rho^{sp}$. Let $\nu_0 \in (0,1)$ and $q>1$ be parameters determined in Lemma~\ref{Lm:DGtype}, both depending only on $\data$ and $\eps$. Having at hand $\delta$ temporarily, we choose $\nu_\ast$ small enough, such that
\begin{equation}\label{eq:weakDGtype-est-nu}
\nu_\ast \leq 2^{-d}\nu_0\delta^{q},
\end{equation}
thereby getting
\begin{equation}\label{eq:weakDGtype-est-5}
\left|\left\{u <\frac{1}{2}k\right\} \cap Q_\rho^{\delta}\right| \leq \nu_0 |Q_\rho^\delta|.
\end{equation}
Thus, by~\eqref{eq:weakDGtype-est-3} and~\eqref{eq:weakDGtype-est-5}, we are allowed to apply Lemma~\ref{Lm:DGtype} with the center $(0,\delta\rho^{sp})$, and therefore
\[
u \geq \frac{k}{4} \quad \mbox{a.e.\,\,in}\,\,\,B_{\frac{\rho}{2}}\times \left(\frac{\delta}{2}\rho^{sp}, \rho^{sp}\right].
\]
Note that this case does not require any quantitative characterization of $\delta$, which is still to be specified. On the other hand, we suppose, for every $i \in \N_0$, that
\begin{equation}\label{eq:weakDGtype-est-6}
\frac{|A_i|}{|Q_i|}=\frac{\left|\left\{u<k_i\right\} \cap Q_i\right|}{|Q_i|} > \nu_\ast.
\end{equation}
Set
\[
\boldsymbol{Z}_i:=\left[\dashint_0^{\delta \rho^{sp}}\left(\frac{\left|A_i(t)\right|}{\left|B_i\right|}\right)^{\frac{p-1+\eps}{\eps}}\d{t}\right]^{\frac{\eps}{(1+\lambda)(p-1+\eps)}} \quad \mbox{with} \quad \lambda=\frac{\eps sp}{d(p-1+\eps)}.
\]
After inserting~\eqref{eq:weakDGtype-est-6} into~\eqref{eq:weakDGtype-est-4}, rearranging gives
\begin{align*}
\boldsymbol{Y}_{\!i+1} &\leq c \boldsymbol{b}^i \delta^{-\frac{1}{p\kappa}-\frac{p-1}{p(p-1+\eps)}}\left(\boldsymbol{Y}_{\!i}^{1+\beta}+\boldsymbol{Y}_{\!i}^{\beta} \cdot \boldsymbol{Y}_{\!i}^{\frac{p-1}{p}}\boldsymbol{Z}_i^{\frac{1+\lambda}{p}}\right) \\
&\leq c \boldsymbol{b}^i \delta^{-\frac{1}{p\kappa}-\frac{p-1}{p(p-1+\eps)}}\left(\boldsymbol{Y}_{\!i}^{1+\beta}+ \boldsymbol{Y}_{\!i}^{\beta}\boldsymbol{Z}_i^{1+\lambda}\right)
\end{align*}
for a constant $c(d,s,p,\gamma_{\mathrm{DG}})$ with $\boldsymbol{b}=4^{d+sp}$ and $\beta:=\frac{1}{p}-\frac{1}{p\kappa}>0$. Here in the penultimate line we used $\boldsymbol{Y}_{\!i}^{\frac{p-1}{p}} \leq \boldsymbol{Z}_i^{\frac{p-1}{p}(1+\lambda)}$, which follows from H\"{o}lder's inequality.

Carrying out the above procedure, together with a complete similar argument leading to~\eqref{eq:dg-est-9}, we deduce that
\[
\boldsymbol{Z}_{i+1} \leq c\boldsymbol{b}^i \delta^{-\frac{\eps d}{d(p-1+\eps)+\eps sp}-1} \left(\boldsymbol{Y}_{\!i}+\boldsymbol{Z}_i^{1+\lambda}\right).
\]
As before, let
\[
\boldsymbol{K}_i:=\boldsymbol{Y}_{\!i}+\boldsymbol{Z}_i^{1+\lambda}, \quad \gamma:=\max \left\{\frac{1}{p\kappa}+\frac{p-1}{p(p-1+\eps)},\,\frac{\eps d}{d(p-1+\eps)+\eps sp}+1\right\}.
\]
Then, due to the argument as in~\cite[Chapter I.4, Lemma 4.2]{DiB93}, it follows from the previous two estimates that
\[
\boldsymbol{K}_{i+1} \leq 2c^{1+\lambda} \delta^{-\gamma(1+\lambda)}\boldsymbol{b}^{(1+\lambda)i}\boldsymbol{K}_i^{1+\min\{\beta, \lambda\}}
\]
for positive constants $\boldsymbol{b}=4^{d+sp}$ and $c>1$ depending only on $\data$. By Lemma~\ref{Lm:FGC}, we can find a positive number $\delta \in (0,1)$ depending only on $\data$, so that $\boldsymbol{K}_{\!\infty}=\lim_{i \to \infty}\boldsymbol{K}_{\!i}=0$ provided that
\[
\boldsymbol{K}_0 \leq \nu_0 \quad \Longrightarrow \quad |A_0|=\left|\left\{u<k_0\right\} \cap Q_0\right|\leq \nu_0|Q_0|,
\]
where $\nu_0 \in (0,1)$ is determined in Lemma~\ref{Lm:DGtype}. Indeed, we can specify $\delta \in (0,1)$ such that
\begin{align*}
\boldsymbol{K}_0 &\leq \left(2c^{1+\lambda} \delta^{-\gamma(1+\lambda)}\right)^{-\frac{1}{\min\{\beta, \lambda\}}}\boldsymbol{b}^{-\frac{1+\lambda}{\min\{\beta ,\lambda\}^2}} \leq \nu_0 \\
 &\Longleftarrow \quad \delta=\frac{1}{3}\nu_0^{\frac{\min\{\beta, \lambda\}}{\gamma(1+\lambda)}} (2c^{1+\lambda})^{\frac{1}{\gamma(1+\lambda)}}\boldsymbol{b}^{\frac{1}{\gamma \min\{\beta,\lambda\}}}.
\end{align*}
Once $\delta$ is specified in such a way, we can fix $\nu_\ast$ via~\eqref{eq:weakDGtype-est-nu} and thus $\nu_\ast$ also depends only on the $\data$. As a consequence, we have that
\[
\boldsymbol{K}_{\!\infty}=0 \quad \Longrightarrow \quad u \geq \frac{k}{2} \quad \mbox{a.e. in $B_{\frac{\rho}{2}} \times (0, \delta \rho^{sp}]$}.
\]
Finally, joining both cases yields the desired claim.
\end{proof}

\subsection{Proof of Theorem~\ref{Thm:weakHarnack2}}\label{Sect.5.2}

To begin with, let $(x_0,t_0)=(0,0)$ for simplicity. Lemmas~\ref{Lm:newmt1} and~\ref{Lm:newmt2} (with redefining $k$) imply that
\[
\displaystyle \inf\limits_{t \in (-(2\rho)^{sp},0]}\Big|\left\{u(\cdot,t) \leq k \right\} \cap B_{2\rho} \Big| \leq \left[\frac{C_1 k^{p-1}}{(u^{p-1})_{Q_{2\rho}}} \wedge \frac{C_2 k^{p-1}}{\displaystyle \dashint_{-(2\rho)^{sp}}^{0}\Big[\tail \left(u_+(t)\,;B_{2\rho}\right)\Big]^{p-1}\d{t}}\right]|B_{2\rho}|
\]
for a constant $C_1, C_2(d,s,p,\gamma_{\mathrm{DG}})<\infty$, provided that
\begin{equation}\label{eq:weakHarnack3-est-1}
\left(\frac{\rho}{R}\right)^{\frac{sp\eps}{(p-1)(p-1+\eps)}} \left(\dashint_{I_R} \Big[\tail \left(u_-(t)\,;B_R\right)\Big]^{p-1+\eps}\d{t}\right)^{\nicefrac{1}{(p-1+\eps)}} \leq k.
\end{equation}
Let $\nu \in (0,1)$ be determined in Lemma~\ref{Lm:DGtype}. We then select $k>0$ so that
\begin{align*}
&\left[\frac{C_1}{(u^{p-1})_{Q_{2\rho}}} \wedge \frac{C_2}{\displaystyle \dashint_{-(2\rho)^{sp}}^{0}\Big[\tail \left(u_+(t)\,;B_{2\rho}\right)\Big]^{p-1}\d{t}}\right]k^{p-1} \leq \nu \\[4pt]
\Longrightarrow k =\frac{1}{2} &\left[\left(\frac{\nu}{C_1}(u^{p-1})_{Q_{2\rho}}\right)^{\nicefrac{1}{(p-1)}}+\left(\frac{\nu}{C_2}\dashint_{-(2\rho)^{sp}}^{0}\Big[\tail \left(u_+(t)\,;B_{2\rho}\right)\Big]^{p-1}\d{t} \right)^{\nicefrac{1}{(p-1)}}\right].
\end{align*}
With this choice of $k$, let $t_\ast \in \left(-(2\rho)^{sp},0 \right]$ be an instant that attains the infimum. In this way, we have that
\[
\Big|\left\{u(\cdot,t_\ast) \leq k \right\} \cap B_{2\rho} \Big| \leq \nu |B_{2\rho}|
\]
and therefore, appealing to Lemma~\ref{Lm:weakDGtype}, there exists $\delta(\data) \in (0,1)$ such that 
\[
u \geq \frac{1}{4}k \quad \mbox{a.e.\,\,in} \quad B_{\frac{\rho}{2}}\times \left(t_\ast+\frac{3}{4}\delta\rho^{sp}, t_\ast+\delta\rho^{sp}\right],
\]
provided that $B_{2\rho} \times \left(-(2\rho)^{sp}, \delta \rho^{sp} \right] \subset \cQ_R \Subset \Omega_T$ and 
\begin{equation}\label{eq:weakHarnack3-est-2}
\left(\frac{\rho}{R}\right)^{\frac{sp\eps}{(p-1)(p-1+\eps)}} \left(\dashint_{I_R} \Big[\tail \left(u_-(t)\,;B_R\right)\Big]^{p-1+\eps}\d{t}\right)^{\nicefrac{1}{(p-1+\eps)}} \leq \frac{1}{4}k.
\end{equation}
Since, by~\eqref{eq:weakHarnack3-est-2}, the condition~\eqref{eq:weakHarnack3-est-1} is automatically fulfilled, we may demand~\eqref{eq:weakHarnack3-est-2} hereafter. 

We next apply the power-like expansion of positivity (Lemma~\ref{Lm:mt2}) with $\alpha=1$. Thus, we have that there exists $\eta_0 \in (0,1)$ depending only on the $\data$ such that, if we enforce 
\begin{equation}\label{eq:weakHarnack3-est-3}
\left(\frac{\rho}{R}\right)^{\frac{sp\eps}{(p-1)(p-1+\eps)}} \left(\dashint_{I_R} \Big[\tail \left(u_-(t)\,;B_R\right)\Big]^{p-1+\eps}\d{t}\right)^{\nicefrac{1}{(p-1+\eps)}} \leq \eta_0k
\end{equation}
then
\begin{equation}\label{eq:weakHarnack3-est-4}
u \geq \eta_0k \quad \mbox{a.e.\,\,in} \quad B_\rho \times \left(t_\ast+\frac{3}{4}\delta\rho^{sp}+\frac{1}{2}(4\rho)^{sp}, t_\ast+\delta\rho^{sp}+2(4\rho)^{sp} \right]
\end{equation}
if $B_{2\rho} \times \left(-(2\rho)^{sp}, 6(4\rho)^{sp}\right] \subset \cQ_R \Subset \Omega_T$. 
As argued before, we are allowed to enforce~\eqref{eq:weakHarnack3-est-2} replaced by~\eqref{eq:weakHarnack3-est-3}. Regardless of where $t_\ast$ is in the range $\left(-(2\rho)^{sp}, 0 \right]$, it is straightforward to check that
\[
\left(\frac{3}{4}(4\rho)^{sp}, (4\rho)^{sp}\right] \subset \left(t_\ast+\frac{3}{4}\delta\rho^{sp}+\frac{1}{2}(4\rho)^{sp}, t_\ast+\delta\rho^{sp}+2(4\rho)^{sp} \right],
\]
owing to the choice of $\delta$ in the proof of Lemma~\ref{Lm:weakDGtype}.
Thus, thanks to this, the stipulation of $k$ as above and two displays~\eqref{eq:weakHarnack3-est-3} and~\eqref{eq:weakHarnack3-est-4}, after letting $C:=C_1 \vee  C_2$, we can conclude that
\begin{align*}
u \geq \eta_0 \,\frac{1}{2}\left(\frac{\nu}{c}\right)^{\nicefrac{1}{(p-1)}}&\left[\left(\dashint_{Q_{2\rho}}u^{p-1}\dxt \right)^{\nicefrac{1}{(p-1)}}+\left(\dashint_{-(2\rho)^{sp}}^{0}\Big[\tail \left(u_+(t)\,;B_{2\rho}\right)\Big]^{p-1}\d{t} \right)^{\nicefrac{1}{(p-1)}} \right]\\
&-\left(\frac{\rho}{R}\right)^{\frac{sp\eps}{(p-1)(p-1+\eps)}} \left(\dashint_{I_R} \Big[\tail \left(u_-(t)\,;B_R\right)\Big]^{p-1+\eps}\d{t}\right)^{\nicefrac{1}{(p-1+\eps)}}
\end{align*}
almost everywhere in $B_\rho \times \left(\frac{3}{4}(4\rho)^{sp}, (4\rho)^{sp} \right]$, provided that
\[
B_{2\rho} \times \left(-(2\rho)^{sp}, 6(4\rho)^{sp}\right] \subset \cQ_R \Subset \Omega_T.
\]
Finally, letting $\eta:=\eta_0\frac{1}{2}\left(\frac{\nu}{c}\right)^{\nicefrac{1}{(p-1)}}$ concludes the proof.

\section{Proof of Theorem~\ref{Thm:Holder modulus}}\label{Sect.6}
In this section, we prove Theorem~\ref{Thm:Holder modulus}, that is, we derive the H\"{o}lder modulus of continuity of functions in our De Giorgi classes. The proof presented here is very straightforward, however, much more technically due to the effect of the tails (see~\cite[Page 498--499]{DG23} and~\cite[Section 4]{Lia24b} for instance). Its versatility will cover the more general case with just minor adjustments. For reader's convenience, we will split it into some subsections.

\subsection{First shot}
This subsection will be needed in the next subsection as the initial step of induction.  
To do so, we will employ the well-known two-alternatives argument. 

By translation we may assume that $(x_0,t_0)=(0,0)$. Given the ambient cylinder $\cQ_R$ and $\rho \in (0, R)$, set
\begin{equation}\label{eq:first-est-1}
\boldsymbol{\om}:=2\sup_{\cQ_R} |u|+\left(\frac{\rho}{R}\right)^{\frac{sp\eps}{(p-1)(p-1+\eps)}} \left(\dashint_{I_R} \Big[\tail \left(u(t)\,;B_R\right)\Big]^{p-1+\eps}\d{t}\right)^{\nicefrac{1}{(p-1+\eps)}}.
\end{equation}
Note that this quantity is finite thanks to Theorem~\ref{Thm:boundedness} and~\eqref{d1'} in Definition~\ref{def of DG}. Moreover, we define
\[
\boldsymbol{\mu}^+:=\sup_{Q_\rho} u, \qquad \boldsymbol{\mu}^-:=\inf_{Q_\rho} u,
\]
which implies by~\eqref{eq:first-est-1} that
\[
\osc_{Q_\rho} u \leq \boldsymbol{\om}. 
\]

\noindent In prior to proceed, we exhibit a prerequisite lemma:

\begin{lemma}\label{Lm:keylemma}
Let $u \in \mathbf{PDG}_\pm^{s,p}(\Omega_T, \gamma_{\mathrm{DG}}, \eps)$ in the sense of Definition~\ref{def of DG}. Suppose, for constants $\alpha, \xi \in (0,1)$, that 
\[
\Big|\Big\{\pm \left(\boldsymbol{\mu}^\pm-u(\cdot,t_1) \right)\geq \xi \om \Big\} \cap B_{\rho}\Big| \geq \alpha |B_\rho|.
\]
Then, there exist constants $\delta \in (0,1)$, $\eta \in (0,\nicefrac{1}{2})$, both depending  on $\data$ and $\alpha$, and $\boldsymbol{\gamma}(\data)\in [1,\infty)$ such that, either
\[
\boldsymbol{\gamma} \left(\frac{\rho}{R}\right)^{\frac{sp\eps}{(p-1)(p-1+\eps)}} \left(\dashint_{I_R} \Big[\tail \left((u-\boldsymbol{\mu}^\pm)_\pm(t)\,;B_R\right)\Big]^{p-1+\eps}\d{t}\right)^{\nicefrac{1}{(p-1+\eps)}} > \eta \xi \boldsymbol{\om}
\]
or
\[
\pm(\boldsymbol{\mu}^\pm-u) \geq \eta\xi \boldsymbol{\om} \quad \mbox{a.e.\,\,in}\,\,\,B_{2\rho} \times \left(t_1+\frac{1}{2}\delta \rho^{sp}, t_1+\delta \rho^{sp}\right]
\]
provided that $B_{4\rho} \times \left(t_1,t_1+\delta\rho^{sp}\right] \subset \cQ_R$. Moreover, it can be traced explicitly that 
\[
\eta=\frac{\alpha}{8p} \quad \mbox{and}\quad \delta=\delta_0\alpha^{(d+p+1)\frac{p-1+\eps}{\eps}}
\]
for a constant $\delta_0 \equiv \delta_0(d,s,p,\eps) \in (0,1)$.
\end{lemma}
 
\begin{proof}
We deal with the minus case since the other case is treated similarly. Moreover, the proof closely follows the one of Lemma~\ref{Lm:mt2}, and thus we confine ourselves to give a sketch of it. Indeed, mimicking the proof of Lemma~\ref{Lm:mt} replaced $k$ by $k=\boldsymbol{\mu}^-+\xi\boldsymbol{\om}$, we have 
\[
\eta=\frac{\alpha}{8p}, \quad \delta=\delta_0\alpha^{(d+p)\frac{p-1+\eps}{\eps}+1} \quad \mbox{for \,$\delta_0=\delta_0(d,s,p,\eps) \in (0,1)$}
\]
such that
\[
\Big|\left\{u-\boldsymbol{\mu}^-\geq \eta \xi \boldsymbol{\om}\right\} \cap B_{4\rho} \Big| \geq \frac{\alpha}{2}4^{-d}|B_{4\rho}|, \quad \forall t \in (t_1, t_1+\delta\rho^{sp}],
\]
after enforcing
\[
\boldsymbol{\gamma} \left(\frac{\rho}{R}\right)^{\frac{sp\eps}{(p-1)(p-1+\eps)}} \left(\dashint_{I_R} \Big[\tail \left((u-\boldsymbol{\mu}^-)_-(t)\,;B_R\right)\Big]^{p-1+\eps}\d{t}\right)^{\nicefrac{1}{(p-1+\eps)}}\leq \eta \xi \boldsymbol{\om}.
\]
After that, we make a De Giorgi type lemma like Lemma~\ref{Lm:DGtype}, whose proof is replicable by replacing $k_i$, $\widetilde{k}_i$ and $w_-$ with
\[
k_i:=\boldsymbol{\mu}^-+\frac{\xi \boldsymbol{\om}}{2}+\frac{\xi \boldsymbol{\om}}{2^{i+1}}, \quad \widetilde{k}_i:=\frac{1}{2}(k_{i}+k_{i+1}), \quad \widetilde{w}:=(u-\widetilde{k}_i)_-, \quad \forall i \in \N_0,
\]
respectively. The remaining part follows from an analogous fashion as in Lemma~\ref{Lm:mt2}.
 \end{proof}
Now, let $\delta \in (0,1)$ be specified in Lemma~\ref{Lm:keylemma} with taking $\alpha=1/2$. We set, for a small quantity $c \in (0,1)$, depending only on $\data$ and $\eps$, to be eventually determined in~\eqref{eq:ind-est-8} and moreover~\eqref{eq:modulus-est-2}, that $\tau:=\delta (c\rho)^{sp}$. The approach unfolds along two alternatives such as the concentration of $u$ in the measure; namely, either

\begin{align*}
&\Big|\left\{u(\cdot, -\tau)-\boldsymbol{\mu}^- >\frac{1}{4}\boldsymbol{\om} \right\} \cap B_{c\rho}\Big| \geq \frac{1}{2} |B_{c\rho}|
\intertext{or}
&
\Big|\left\{\boldsymbol{\mu}^+-u(\cdot, -\tau)> \frac{1}{4}\boldsymbol{\om} \right\} \cap B_{c\rho} \Big| \geq \frac{1}{2} |B_{c\rho}|.
\end{align*}
It should be stressed that the small quantity $c$ is to be role of implementing the forthcoming induction argument, as well.

We may assume $\boldsymbol{\mu}^+-\boldsymbol{\mu}^- \geq \frac{1}{2}\boldsymbol{\om}$ without loss of generality and therefore, one of the above alternatives must hold, which is confirmed by contradiction argument. On the other hand, the remaining case $\boldsymbol{\mu}^+-\boldsymbol{\mu}^- < \frac{1}{2}\boldsymbol{\om}$ will be incorporated into the forthcoming estimations. 

For simplicity, we deal with the second alternative. Then, Lemma~\ref{Lm:keylemma} with $t_1=-\tau=-\delta(c\rho)^{sp}$, $\alpha=1/2$ and $\xi=1/4$, yields that, either
\[
\boldsymbol{\gamma}\left(\frac{c\rho}{\rho}\right)^{\frac{sp\eps}{(p-1)(p-1+\eps)}}\left(\dashint_{-(c\rho)^{sp}}^0\Big[\tail \left((u-\boldsymbol{\mu}^+)_+(t) ; B_\rho\right)\Big]^{p-1+\eps}\d{t} \right)^{\nicefrac{1}{(p-1+\eps)}} >\eta \boldsymbol{\om}
\]
or
\[
\boldsymbol{\mu}^+-u \geq \frac{1}{4}\eta \boldsymbol{\om} \quad \mbox{a.e.\,\,in}\,\,\,Q_{c\rho}^{(\nicefrac{\delta}{2})}:=B_{c\rho} \times \left(-\frac{\delta}{2}(c\rho)^{sp}, 0\right]
\]
with $\eta=1/(16p)$. After redefining $\nicefrac{\eta}{4}$ as $\eta$ and denoting $\boldsymbol{\gamma}^\prime :=\boldsymbol{\gamma} / (4\eta)$, the above implication eventually leads to
\begin{equation}\label{eq:first-est-2}
\osc_{Q_{c\rho}^{(\nicefrac{\delta}{2})}}u \leq \max \left\{ (1-\eta)\boldsymbol{\om},\, \mathbf{S} \right\}=:\boldsymbol{\om}_1
\end{equation}
with
\[
\mathbf{S}:=\boldsymbol{\gamma}^\prime c^{\frac{sp\eps}{(p-1)(p-1+\eps)}}\left(\dashint_{-(c\rho)^{sp}}^0\Big[\tail \left((u-\boldsymbol{\mu}^+)_+(t) ; B_\rho\right)\Big]^{p-1+\eps}\d{t} \right)^{\nicefrac{1}{(p-1+\eps)}}.
\]
It immediately follows from~\eqref{eq:first-est-2} and $\eta<1/2$ that $
\boldsymbol{\om}_1 \leq (1-\eta)\boldsymbol{\om} <\frac{1}{2}\boldsymbol{\om}
$.

On the other hand, by splitting the tail term, we estimate that
\begin{align*}
&\dashint_{-(c\rho)^{sp}}^0\Big[\tail \left((u-\boldsymbol{\mu}^+)_+(t) ; B_{\rho}\right)\Big]^{p-1+\eps}\d{t} \\
& \quad \stackrel{\!\!\!\!\!\!\!\!\!\!\!\eqref{eq:first-est-1}}{\lesssim_{s,p,\eps}} \boldsymbol{\om}^{p-1+\eps}+\dashint_{-(c\rho)^{sp}}^0 \left[\rho^{sp} \int_{B_R \setminus B_\rho} \frac{u_+^{p-1}}{|x|^{d+sp}}\d{x}\right]^{\frac{p-1+\eps}{p-1}}\d{t} \\
&\quad \quad \quad \quad \quad \quad \quad +\dashint_{-(c\rho)^{sp}}^0\left[\rho^{sp} \int_{\R^d \setminus B_R} \frac{u_+^{p-1}}{|x|^{d+sp}}\d{x}\right]^{\frac{p-1+\eps}{p-1}} \d{t}\\
& \quad \lesssim_{s,p,\eps} 2\boldsymbol{\om}^{p-1+\eps}+ \left(\frac{\rho}{R} \right)^{\frac{sp(p-1+\eps)}{p-1}} \dashint_{-(c\rho)^{sp}}^0 \Big[\tail \left(u_+(t) ; B_R\right)\Big]^{p-1+\eps}\d{t},
\end{align*}
and thus it follows from~\eqref{eq:first-est-1} that
\begin{align*}
\mathbf{S} &\leq \widetilde{\boldsymbol{\gamma}} c^{\frac{sp\eps}{(p-1)(p-1+\eps)}} \boldsymbol{\om} \\
&\quad + \widetilde{\boldsymbol{\gamma}} c^{\frac{sp\eps}{(p-1)(p-1+\eps)}-\frac{sp}{p-1+\eps}} \left(\frac{\rho}{R}\right)^{\frac{sp\eps}{(p-1)(p-1+\eps)}} \left(\dashint_{I_R} \Big[\tail \left(u_+(t) ; B_R\right)\Big]^{p-1+\eps}\d{t} \right)^{\nicefrac{1}{(p-1+\eps)}}\\
&\!\!\!\stackrel{\eqref{eq:first-est-1} }{\leq} \widetilde{\boldsymbol{\gamma}}\boldsymbol{\om}
\end{align*}
for a constant $\widetilde{\boldsymbol{\gamma}}<\infty$ depending only on $\data$. Hence, we have shown that
\begin{equation}\label{eq:first-est-3}
\mathbf{S} \leq \widetilde{\boldsymbol{\gamma}}\boldsymbol{\om}, \qquad \boldsymbol{\om}_1 \leq \frac{1}{2}\boldsymbol{\om}.
\end{equation}
Now, we set $\rho_1:=\lambda \rho$ with $\lambda \leq c$. Select $\lambda$ so that
\[
\lambda \leq \min \left\{c, \left(\frac{\delta}{2}\right)^{\nicefrac{1}{sp}}c \right\} \quad \Longrightarrow \quad  Q_{\rho_1} \subset Q_{c\rho}^{(\nicefrac{\delta}{2})}. 
\]
Altogether, we have shown that
\[
Q_{\rho_1} \subset Q_\rho, \quad \mbox{and} \quad \osc_{Q_{\rho_1}} u \leq \boldsymbol{\om}_1,
\]
and thus a combination of this,~\eqref{eq:first-est-2} and \eqref{eq:first-est-3} implies the initial step of induction over $i \in \N_0$.
\subsection{Induction step}
The goal of this subsection is to complete an induction argument. To do so, let $c \in (0,1)$ and $\lambda \leq c$, both same quantities appearing in the last section, be two free parameters to be specified later. With this at hand, supposing up to $i=1,\ldots ,j$ for $j \in \N$, we have proved inductively that
\[
\begin{cases}
\rho_0=\rho, \quad \rho_i=\lambda \rho_{i-1}, \quad \frac{1}{2} \boldsymbol{\om}_{i-1} \leq \boldsymbol{\om}_i \leq \boldsymbol{\gamma}_\texttt{J} \boldsymbol{\om}_{i-1}, \\[4pt]
\boldsymbol{\om}_0=\boldsymbol{\om}, \quad \boldsymbol{\om}_i=\displaystyle \max \Big\{ (1-\eta)\boldsymbol{\om}_{i-1}, \mathbf{S}_{i-1}\Big\},\\[4pt]
B_i \equiv B_{\rho_i}, \quad Q_i \equiv Q_{\rho_i}:=B_i \times (-\rho_i^{sp},0],\\[4pt]
\boldsymbol{\mu}_i^+:=\sup\limits_{Q_i}u, \quad \boldsymbol{\mu}_i^-:=\inf\limits_{Q_i}u, \quad \boldsymbol{\om}_i=\osc\limits_{Q_i}u,
\end{cases}
\]
where we used the shorthand notation $\mathbf{S}_{i-1}$, displayed in the second line, that
\[
   \left\{
    \begin{array}{ll}
       \mathbf{S}_{i-1}:=\displaystyle \boldsymbol{\gamma}_{\texttt{Q}}c^{\frac{sp\eps}{(p-1)(p-1+\eps)}}\left(\dashint_{-(c\rho_{i-1})^{sp}}^0\Big[\tail \left((u-\boldsymbol{\mu}_{i-1}^+)_+(t)\,; B_{i-1}\right)\Big]^{p-1+\eps}\d{t} \right)^{\nicefrac{1}{(p-1+\eps)}}, \\[4pt]
       \mathbf{S}_0:=\mathbf{S}\\[4pt]
    \end{array}
    \right.
\]
with $\mathbf{S}$ being as~\eqref{eq:first-est-2}, and $\boldsymbol{\gamma}_{\texttt{J}}, \boldsymbol{\gamma}_{\texttt{Q}} >1$ depending on $\data$ and $\eps$. Notice that, the inequality $\frac{1}{2} \boldsymbol{\om}_{i-1} \leq \boldsymbol{\om}_i$ is readily follows from by definition of $\boldsymbol{\om}_i$ and $\eta<\nicefrac{1}{2}$. Throughout this subsection, we aim at proving that the above oscillation estimate continues to holds for the $(j+1)$-th step. Let $\delta \in (0,1)$ the quantity as before  and abbreviate $\tau:=\delta (c\rho_j)^{sp}$ and $cB_j:=B_{c\rho_j}$.  We then consider two alternatives, namely, either

\begin{align*}
&\Big|\left\{u(\cdot, -\tau)-\boldsymbol{\mu}_j^- >\frac{1}{4}\boldsymbol{\om}_j \right\} \cap cB_j \Big| \geq \frac{1}{2} |cB_j|
\intertext{or}
&
\Big|\left\{\boldsymbol{\mu}_j^+-u(\cdot, -\tau)> \frac{1}{4}\boldsymbol{\om}_j \right\} \cap cB_j \Big| \geq \frac{1}{2} |cB_j|.
\end{align*}
As mentioned in the previous section, we may confine ourselves to $\boldsymbol{\mu}_j^+-\boldsymbol{\mu}_j^- \geq \frac{1}{2}\boldsymbol{\om}_j$. The remaining case $\boldsymbol{\mu}_j^+-\boldsymbol{\mu}_j^- < \frac{1}{2}\boldsymbol{\om}_j$ will be incorporated into the forthcoming estimates. As before, we deal with the second alternative case. Having the second alternative at hands, an application of Lemma~\ref{Lm:keylemma} yields that either
\[
\boldsymbol{\gamma}_{\texttt{Q}} \left(\frac{c\rho_j}{\rho_j}\right)^{\frac{sp\eps}{(p-1)(p-1+\eps)}}\left(\dashint_{-(c\rho_{j})^{sp}}^0\Big[\tail \left((u-\boldsymbol{\mu}_{j}^+)_+(t) ; B_{j}\right)\Big]^{p-1+\eps}\d{t} \right)^{\nicefrac{1}{(p-1+\eps)}} >\eta \boldsymbol{\om}_j
\]
or
\[
\boldsymbol{\mu}_j^+-u \geq \eta \boldsymbol{\om}_j \quad \mbox{a.e.\,\,in}\,\,\,Q_{c\rho_j}^{(\nicefrac{\delta}{2})}:=B_{c\rho_j} \times \left(-\frac{\delta}{2}(c\rho_j)^{sp}, 0\right]
\]
owing to the choice of $\tau=\delta(c\rho_j)^{sp}$. Thus, this together with the $j$-th induction assumption yields that either
\begin{align*}
\osc_{Q_{c\rho_j}^{(\nicefrac{\delta}{2})}}u \leq \boldsymbol{\om}_j \leq \boldsymbol{\gamma}_{\texttt{K}}c^{\frac{sp\eps}{(p-1)(p-1+\eps)}}\left(\dashint_{-(c\rho_{j})^{sp}}^0\Big[\tail \left((u-\boldsymbol{\mu}_{j}^+)_+(t) ; B_{j}\right)\Big]^{p-1+\eps}\d{t} \right)^{\nicefrac{1}{(p-1+\eps)}}
\end{align*}
or
\begin{align*}
\osc_{Q_{c\rho_j}^{(\nicefrac{\delta}{2})}}u=\sup_{Q_{c\rho_j}^{(\nicefrac{\delta}{2})}} u -\inf_{Q_{c\rho_j}^{(\nicefrac{\delta}{2})}}u \leq \boldsymbol{\mu}_j^+-\eta \boldsymbol{\om}_j-\boldsymbol{\mu}_j^- \leq (1-\eta)\boldsymbol{\om}_j
\end{align*}
with denoting $\boldsymbol{\gamma}_{\texttt{K}}=\boldsymbol{\gamma}_{\texttt{Q}}/\eta$. Altogether, we have that
\begin{equation}\label{eq:ind-est-0}
\osc_{Q_{c\rho_j}^{(\nicefrac{\delta}{2})}}u \leq \max\Big\{ (1-\eta) \boldsymbol{\om}_j, \mathbf{S}_j\Big\}=:\boldsymbol{\om}_{j+1},
\end{equation}
where, for ease of notation, we wrote
\[
\mathbf{S}_j:=\boldsymbol{\gamma}_{\texttt{K}}c^{\frac{sp\eps}{(p-1)(p-1+\eps)}}\left(\dashint_{-(c\rho_{j})^{sp}}^0\Big[\tail \left((u-\boldsymbol{\mu}_{j}^+)_+(t) ; B_{j}\right)\Big]^{p-1+\eps}\d{t} \right)^{\nicefrac{1}{(p-1+\eps)}}.
\]
The argument now goes into further steps for the readability.
\smallskip

\emph{Step 1.}\quad In this step, we will show
\begin{equation}\label{eq:ind-est-1}
\boldsymbol{\om}_{j+1} \leq \widetilde{\boldsymbol{\gamma}} \boldsymbol{\om}_j
\end{equation}
for a constant $\widetilde{\boldsymbol{\gamma}} \equiv \widetilde{\boldsymbol{\gamma}} (\data)<\infty$.  The main task to check~\eqref{eq:ind-est-1} is condensed in careful estimates of the tail term as follows. Splitting the tail term into the integrations over small annuluses yields  
\begin{align*}
&\Big[\tail \left((u-\boldsymbol{\mu}_{j}^+)_+(t) ; B_{j}\right)\Big]^{p-1+\eps} \\
& \quad \lesssim_{p,\eps} \left[\underbrace{\rho_j^{sp} \int_{\R^d \setminus B_j} \frac{(u-\boldsymbol{\mu}_j^+)^{p-1}_+}{|x|^{d+sp}}\d{x}}_{=:\mathbf{T}_0}\right]^{\frac{p-1+\eps}{p-1}}+ \left[ \sum_{i=1}^j\underbrace{\rho_j^{sp}\int_{B_{i-1} \setminus B_i} \frac{(u-\boldsymbol{\mu}_j^+)^{p-1}_+}{|x|^{d+sp}}\d{x}}_{=:\mathbf{T}_i}\right]^{\frac{p-1+\eps}{p-1}},
\end{align*}
and thus we deduce that, for a constant $\boldsymbol{\gamma}<\infty$
\begin{align}\label{eq:ind-est-2}
\mathbf{S}_j &\leq \boldsymbol{\gamma}_{\texttt{K}} \boldsymbol{\gamma}c^{\frac{sp\eps}{(p-1)(p-1+\eps)}}\left(\dashint_{-(c\rho_j)^{sp}}^0 \mathbf{T}_0^{\frac{p-1+\eps}{p-1}}\d{t}\right)^{\nicefrac{1}{(p-1+\eps)}} \notag \\[4pt]
 &\quad \quad \quad +\boldsymbol{\gamma}_{\texttt{K}} \boldsymbol{\gamma}c^{\frac{sp\eps}{(p-1)(p-1+\eps)}}\left(\dashint_{-(c\rho_j)^{sp}}^0 \left[ \sum_{i=1}^j\mathbf{T}_i\right]^{\frac{p-1+\eps}{p-1}}\d{t}\right)^{\nicefrac{1}{(p-1+\eps)}}.
\end{align}
Unless otherwise stated, the bold symbol $\boldsymbol{\gamma}$ hereafter denotes the general constant depending on $\data$ for simplicity. We start by estimating the integral including $\mathbf{T}_0$. Owing to two relations
\begin{equation}\label{eq:ind-est-3}
|\boldsymbol{\mu}_j^+| \leq \boldsymbol{\om} \quad \mbox{and} \quad u_+ \leq \boldsymbol{\om} \quad\mbox{on $\mathcal{Q}_R$}, 
\end{equation}
we estimate, for all $t \in (-(c\rho_j)^{sp}, 0)$, that
\begin{align}\label{eq:ind-est-4}
\mathbf{T}_0 &\leq \boldsymbol{\gamma} \rho_j^{sp} \int_{\R^d \setminus B_\rho}\frac{|\boldsymbol{\mu}_j^+|^{p-1}+u_+^{p-1}}{|x|^{d+sp}}\d{x} \notag\\
&\!\!\!\stackrel{\eqref{eq:ind-est-3}_1}{\leq} \boldsymbol{\gamma} \rho_j^{sp}\frac{\boldsymbol{\om}^{p-1}}{\rho^{sp}}+\boldsymbol{\gamma}\rho_j^{sp} \int_{B_R \setminus B_\rho}\frac{u_+^{p-1}}{|x|^{d+sp}}\d{x}+\boldsymbol{\gamma}\rho_j^{sp} \int_{\R^d\setminus B_R}\frac{u_+^{p-1}}{|x|^{d+sp}}\d{x} \notag\\
&\!\!\!\stackrel{\eqref{eq:ind-est-3}_2}{\leq} \boldsymbol{\gamma} \rho_j^{sp}\frac{\boldsymbol{\om}^{p-1}}{\rho^{sp}}+\boldsymbol{\gamma}\rho_j^{sp} \int_{\R^d \setminus B_R}\frac{u_+^{p-1}}{|x|^{d+sp}}\d{x}.
\end{align}
Furthermore, using $u_+^{p-1} \lesssim_p |\boldsymbol{\mu}_{j-1}^+|^{p-1}+(u-\boldsymbol{\mu}_{j-1}^+)_+^{p-1}$ we have
\begin{align*}
\rho_j^{sp} &\int_{\R^d \setminus B_R}\frac{u_+^{p-1}}{|x|^{d+sp}}\d{x} \notag\\
&\leq \boldsymbol{\gamma}  \int_{\R^d \setminus B_R}\frac{|\boldsymbol{\mu}_{j-1}^+|^{p-1}}{|x|^{N+sp}}\d{x}+\boldsymbol{\gamma}  \int_{\R^d \setminus B_R}\frac{(u-\boldsymbol{\mu}_{j-1}^+)_+^{p-1}}{|x|^{d+sp}}\d{x}\notag\\
&\!\!\!\!\stackrel{\eqref{eq:ind-est-3}_1}{\leq}\boldsymbol{\gamma} \rho_j^{sp}\frac{\boldsymbol{\om}^{p-1}}{\rho^{sp}}+\boldsymbol{\gamma}  \left(\frac{\rho_j}{\rho_{j-1}}\right)^{sp} \cdot \rho_{j-1}^{sp}\int_{\R^d \setminus B_{j-1}}\frac{(u-\boldsymbol{\mu}_{j-1}^+)_+^{p-1}}{|x|^{d+sp}}\d{x}\\
&= \boldsymbol{\gamma} \rho_j^{sp}\frac{\boldsymbol{\om}^{p-1}}{\rho^{sp}}+\boldsymbol{\gamma}\lambda^{sp}\Big[\tail \left((u-\boldsymbol{\mu}_{j-1}^+)_+(t); B_{j-1}\right)\Big]^{p-1},
\end{align*}
where to obtain the last line we used $\rho_j/\rho_{j-1}=\lambda$. Hence, we have another estimate of $\mathbf{T}_0$: 
\begin{equation}\label{eq:ind-est-5}
\mathbf{T}_0\leq \boldsymbol{\gamma} \rho_j^{sp}\frac{\boldsymbol{\om}^{p-1}}{\rho^{sp}}+\boldsymbol{\gamma}\lambda^{sp}\Big[\tail \left((u-\boldsymbol{\mu}_{j-1}^+)_+(t); B_{j-1}\right)\Big]^{p-1},
\end{equation}
and therefore it holds that
\begin{align*}
&\left(\dashint_{-(c\rho_j)^{sp}}^0 \mathbf{T}_0^{\frac{p-1+\eps}{p-1}}\d{t}\right)^{\nicefrac{1}{(p-1+\eps)}} \\
&\quad \quad \leq \boldsymbol{\gamma}\left(\rho_j^{sp}\frac{\boldsymbol{\om}^{p-1}}{\rho^{sp}}\right)^{\nicefrac{1}{(p-1)}} \\
&\quad \quad \quad \quad +\boldsymbol{\gamma} \left(\lambda^{\frac{sp(p-1+\eps)}{p-1}}\dashint_{-(c\rho_j)^{sp}}^0 \Big[\tail \left((u-\boldsymbol{\mu}_{j-1}^+)_+(t); B_{j-1}\right)\Big]^{p-1+\eps}\d{t}\right)^{\nicefrac{1}{(p-1+\eps)}} \\
&\quad \quad \leq \boldsymbol{\gamma}\left(\rho_j^{sp}\frac{\boldsymbol{\om}^{p-1}}{\rho^{sp}}\right)^{\nicefrac{1}{(p-1)}} \\
&\quad \quad \quad \quad+\boldsymbol{\gamma}\left(\underbrace{\lambda^{\frac{sp\eps}{p-1}}}_{\leq 1}\dashint_{-(c\rho_{j-1})^{sp}}^0\Big[\tail \left((u-\boldsymbol{\mu}_{j-1}^+)_+(t); B_{j-1}\right)\Big]^{p-1+\eps}\d{t}\right)^{\nicefrac{1}{(p-1+\eps)}},
\end{align*}
Here, we recall by definition, for any $i \in \{0,1,\ldots,j\}$, that
\begin{equation}\label{eq:ind-est-6}
2^{i-j}\boldsymbol{\om}_i \leq \boldsymbol{\om}_j \quad \mbox{and}\quad \rho_j=\lambda^{j-i} \rho_i \leq c^{j-i}\rho_i,
\end{equation}
and thus, 
\begin{equation}\label{eq:ind-est-6'}
\left(\rho_j^{sp}\frac{\boldsymbol{\om}^{p-1}}{\rho^{sp}}\right)^{\nicefrac{1}{(p-1)}} \leq \left(2c^{\frac{sp}{p-1}}\right)^j\boldsymbol{\om}_j \leq \left(2c^{\frac{sp\eps}{(p-1)(p-1+\eps)}}\right)^j\boldsymbol{\om}_j
\end{equation}
since $c<1$. This observation eventually leads to
\begin{align}\label{eq:ind-est-7}
&\boldsymbol{\gamma}_{\texttt{K}}\boldsymbol{\gamma}c^{\frac{sp\eps}{(p-1)(p-1+\eps)}}\left(\dashint_{-(c\rho_j)^{sp}}^0 \mathbf{T}_0^{\frac{p-1+\eps}{p-1}}\d{t}\right)^{\nicefrac{1}{(p-1+\eps)}} \notag \\
&\quad \leq \boldsymbol{\gamma}\boldsymbol{\gamma}_{\texttt{K}} \left(2c^{\frac{sp\eps}{(p-1)(p-1+\eps)}}\right)^{j+1}\boldsymbol{\om}_j\notag\\
&\quad \quad \quad  +\boldsymbol{\gamma}\underbrace{\boldsymbol{\gamma}_{\texttt{K}} c^{\frac{sp\eps}{(p-1)(p-1+\eps)}}\left(\dashint_{-(c\rho_{j-1})^{sp}}^0\Big[\tail \left((u-\boldsymbol{\mu}_{j-1}^+)_+(t); B_{j-1}\right)\Big]^{p-1+\eps}\d{t}\right)^{\nicefrac{1}{(p-1+\eps)}}}_{\leq \,\boldsymbol{\om}_j} \notag\\
&\quad \leq \widetilde{\boldsymbol{\gamma}}\left(2c^{\frac{sp\eps}{(p-1)(p-1+\eps)}}\right)^{j+1}\boldsymbol{\om}_j+\widetilde{\boldsymbol{\gamma}}\boldsymbol{\om}_j
\end{align}
for a constant $\widetilde{\boldsymbol{\gamma}}$ depending only on $\data$.

Next, we turn our attention to the second integral on the right side of~\eqref{eq:ind-est-2}. It is easy to check that, for any $i \in \{1,\ldots, j\}$, 
\[
(u-\boldsymbol{\mu}_j^+)_+ \leq \boldsymbol{\mu}_{i-1}^+-\boldsymbol{\mu}_j^- \leq \boldsymbol{\mu}_{i-1}^+-\boldsymbol{\mu}_{i-1}^- \leq \boldsymbol{\om}_{i-1} \quad \mbox{a.e.\,\,in}\,\,Q_{i-1}.
\]
Namely, it holds that, for every $t \in (-(c\rho_j)^{sp}, 0)$, 
\[
\mathbf{T}_i=\rho_j^{sp}\int_{B_{i-1} \setminus B_{i}}\frac{(u-\boldsymbol{\mu}_j^+)_+^{p-1}}{|x|^{d+sp}}\d{x} \leq \boldsymbol{\gamma} \rho_j^{sp} \frac{\boldsymbol{\om}_{i-1}^{p-1}}{\rho_i^{sp}}
\]
and therefore, we have
\begin{align*}
\boldsymbol{\gamma}_{\texttt{K}} &\boldsymbol{\gamma}c^{\frac{sp\eps}{(p-1)(p-1+\eps)}}\left(\dashint_{-(c\rho_j)^{sp}}^0 \left[ \sum_{i=1}^j\mathbf{T}_i\right]^{\frac{p-1+\eps}{p-1}}\d{t}\right)^{\nicefrac{1}{(p-1+\eps)}} \\
&\leq \boldsymbol{\gamma}_{\texttt{K}} \boldsymbol{\gamma}\left(\dashint_{-(c\rho_j)^{sp}}^0\left[ \sum_{i=1}^j c^{\frac{sp\eps}{p-1+\eps}}\rho_j^{sp} \frac{\boldsymbol{\om}_{i-1}^{p-1}}{\rho_i^{sp}}\right]^{\frac{p-1+\eps}{p-1}}\d{t}\right)^{\nicefrac{1}{(p-1+\eps)}}\\
&=\boldsymbol{\gamma}_{\texttt{K}} \boldsymbol{\gamma}\left[ \sum_{i=1}^j c^{\frac{sp\eps}{p-1+\eps}}\rho_j^{sp} \frac{\boldsymbol{\om}_{i-1}^{p-1}}{\rho_i^{sp}}\right]^{\nicefrac{1}{(p-1)}}.
\end{align*}
Thanks to~\eqref{eq:ind-est-6} and the fact that $c<1$, we can bound, for any $i \in \{1,\ldots, j\}$, that
\begin{align*}
c^{\frac{sp\eps}{p-1+\eps}}\rho_j^{sp} \frac{\boldsymbol{\om}_{i-1}^{p-1}}{\rho_i^{sp}} &\leq c^{\frac{sp\eps}{p-1+\eps}}(c^{j-i}\rho_i)^{sp}\frac{(2^{j-i+1}\boldsymbol{\om}_j)^{p-1}}{\rho_i^{sp}}\\
&=(2^{p-1}c^{\frac{sp\eps}{p-1+\eps}})^{j-i+1}\boldsymbol{\om}_j^{p-1},
\end{align*}
thereby getting
\begin{align}\label{eq:ind-est-a}
\boldsymbol{\gamma}_{\texttt{K}} \boldsymbol{\gamma}c^{\frac{sp\eps}{(p-1)(p-1+\eps)}}\left(\dashint_{-(c\rho_j)^{sp}}^0 \left[ \sum_{i=1}^j\mathbf{T}_i\right]^{\frac{p-1+\eps}{p-1}}\d{t}\right)^{\nicefrac{1}{(p-1+\eps)}} 
\leq \boldsymbol{\gamma}_{\texttt{K}} \boldsymbol{\gamma}\boldsymbol{\om}_j\left[\sum_{i=1}^j  \left(2^{p-1}c^{\frac{sp\eps}{p-1+\eps}}\right)^i\right]^{\nicefrac{1}{(p-1)}}.
\end{align}
We now temporary choose the free parameter $c \in (0,1)$ so that
\begin{equation}\label{eq:ind-est-8}
2^{p-1}c^{\frac{sp\eps}{p-1+\eps}} \leq \frac{1}{2} \quad \iff \quad c \leq 2^{-\frac{p-1+\eps}{s\eps}},
\end{equation}
which implies that the above summation is bounded, that is,
\begin{equation}\label{eq:ind-est-9}
\boldsymbol{\gamma}_{\texttt{K}} \boldsymbol{\gamma}c^{\frac{sp\eps}{(p-1)(p-1+\eps)}}\left(\dashint_{-(c\rho_j)^{sp}}^0 \left[ \sum_{i=1}^j\mathbf{T}_i\right]^{\frac{p-1+\eps}{p-1}}\d{t}\right)^{\nicefrac{1}{(p-1+\eps)}} \leq \widetilde{\boldsymbol{\gamma}}\boldsymbol{\om}_j
\end{equation}
for a constant $\widetilde{\boldsymbol{\gamma}}(\data)<\infty$. Here, the special constant $\boldsymbol{\gamma}_{\texttt{K}}$ was incorporated  into $\widetilde{\boldsymbol{\gamma}}$. Consequently, keeping~\eqref{eq:ind-est-8} in mind and joining the preceding estimates~\eqref{eq:ind-est-7},~\eqref{eq:ind-est-9} with~\eqref{eq:ind-est-2}, concludes that 

\begin{align*}
\boldsymbol{\gamma}_{\texttt{K}}c^{\frac{sp\eps}{(p-1)(p-1+\eps)}}\left(\dashint_{-(c\rho_j)^{sp}}^0 \Big[\tail \left((u-\boldsymbol{\mu}_{j}^+)_+(t) ; B_{j}\right)\Big]^{p-1+\eps}\d{t}\right)^{\nicefrac{1}{(p-1+\eps)}} \leq \widetilde{\boldsymbol{\gamma}}\boldsymbol{\om}_j.
\end{align*}
Finally, this together with the definition of $\boldsymbol{\om}_{j+1}$, in turn gives $
\boldsymbol{\om}_{j+1} \leq \widetilde{\boldsymbol{\gamma}} \boldsymbol{\om}_j$. Thus, Step 1 is finally concluded.

\smallskip

\emph{Step 2.}\quad The next task is to verify
\begin{equation}\label{eq:ind-est-10}
\osc_{Q_{j+1}} u \leq \boldsymbol{\om}_{j+1}.
\end{equation}
This is done by simple observation. In fact, as seen in the previous subsection, if we stipulate the free parameter $\lambda$ as $\lambda \leq \min \left\{c, \left(\frac{\delta}{2}\right)^{\nicefrac{1}{sp}}c\right\}$ then the inclusion $Q_{j+1} \subset Q_{c\rho_j}^{(\nicefrac{\delta}{2})}$ is valid as well. From this and~\eqref{eq:ind-est-0}, the estimate~\eqref{eq:ind-est-10} immediately follows. Consequently, the result of~\eqref{eq:ind-est-1} and~\eqref{eq:ind-est-10} completes the induction in the second alternative case. As for the first alternative case, replacing $(u-\boldsymbol{\mu}_j^+)_+$ with $(u-\boldsymbol{\mu}_j^-)_-$, we can run the analogous proof to the above arguments, with a careful inspection of the tail term. It is remarking that, as we will see in the forthcoming estimations, the smallness of $c$ appearing in~\eqref{eq:ind-est-8} will require a further stipulation. 
\subsection{Modulus of continuity}

By construction of the last subsection, we have obtained, for $n \in \N$, that
\[
\osc_{Q_n} u \leq \boldsymbol{\om}_n=\max\Big\{(1-\eta)\boldsymbol{\om}_{n-1}, \mathbf{S}_{n-1} \Big\},
\]
where $\rho_n=\lambda^n\rho$, $Q_n:=Q_{\rho_n}=B_{\rho_n} \times (-\rho_n^{sp}, 0]$ and
\[
\mathbf{S}_{n-1}:=\boldsymbol{\gamma}_{\texttt{K}}c^{\frac{sp\eps}{(p-1)(p-1+\eps)}}\left(\dashint_{-(c\rho_{n-1})^{sp}}^0 \Big[\tail \left((u-\boldsymbol{\mu}_{n-1}^+)_+(t) ; B_{n-1}\right)\Big]^{p-1+\eps}\d{t}\right)^{\nicefrac{1}{(p-1+\eps)}}.
\]
The goal of this subsection is to deduce an explicit modulus of continuity encoded in this oscillation estimate. Using~\eqref{eq:ind-est-2}  and~\eqref{eq:ind-est-a} with $j=n-1$, whereas the integral including $\mathbf{T}_0$ is still to be modified, we have that
\begin{align}\label{eq:modulus-est-1}
\mathbf{S}_{n-1} \leq \overline{\boldsymbol{\gamma}}c^{\frac{sp\eps}{(p-1)(p-1+\eps)}}\left(\dashint_{-(c\rho_{n-1})^{sp}}^0 \mathbf{T}_0^{\frac{p-1+\eps}{p-1}}\d{t}\right)^{\nicefrac{1}{(p-1+\eps)}} +\overline{\boldsymbol{\gamma}}\boldsymbol{\om}_{n-1}\left[\sum_{i=1}^{n-1}  (2^{p-1}c^{\frac{sp\eps}{p-1+\eps}})^i\right]^{\nicefrac{1}{(p-1)}}.
\end{align} 
A minor adjustment requires for $\mathbf{T}_0$. In fact, thanks to~\eqref{eq:ind-est-4} with $j=n-1$ it immediately follows that
\begin{align*}
\mathbf{T}_0 &\leq \boldsymbol{\gamma} \rho_{n-1}^{sp}\frac{\boldsymbol{\om}^{p-1}}{\rho^{sp}}+\boldsymbol{\gamma}\rho_{n-1}^{sp} \int_{\R^d \setminus B_R}\frac{|u|^{p-1}}{|x|^{d+sp}}\d{x} \\
&\leq \boldsymbol{\gamma} \rho_{n-1}^{sp}\frac{\boldsymbol{\om}^{p-1}}{\rho^{sp}}+\lambda^{(n-1)sp}\left(\frac{\rho}{R}\right)^{sp}\Big[\tail\left(u(t) ; B_R\right) \Big]^{p-1}
\end{align*}
for every $t \in (-(c\rho_{n-1})^{sp}, 0)$ and therefore
\begin{align}\label{eq:modulus-est-2}
&\overline{\boldsymbol{\gamma}}c^{\frac{sp\eps}{(p-1)(p-1+\eps)}}\left(\dashint_{-(c\rho_{n-1})^{sp}}^0 \mathbf{T}_0^{\frac{p-1+\eps}{p-1}}\d{t}\right)^{\nicefrac{1}{(p-1+\eps)}} \notag\\
& \quad \leq \overline{\boldsymbol{\gamma}}c^{\frac{sp\eps}{(p-1)(p-1+\eps)}} \left(\rho_{n-1}^{sp}\frac{\boldsymbol{\om}^{p-1}}{\rho^{sp}}\right)^{\nicefrac{1}{(p-1)}} \notag\\
&\quad \quad \quad+\overline{\boldsymbol{\gamma}}c^{\frac{sp\eps}{(p-1)(p-1+\eps)}} \lambda^{\frac{(n-1)sp}{p-1}}\left(\frac{\rho}{R}\right)^{\frac{sp}{p-1}}\left(\dashint_{-(c\rho_{n-1})^{sp}}^0 \Big[\tail\left(u(t) ; B_R\right) \Big]^{p-1+\eps}\d{t} \right)^{\nicefrac{1}{(p-1+\eps)}}\notag\\[4pt]
& \quad \leq  \overline{\boldsymbol{\gamma}} \left(2c^{\frac{sp\eps}{(p-1)(p-1+\eps)}}\right)^n\boldsymbol{\om}_{n-1} +\overline{\boldsymbol{\gamma}}\left(2c^{\frac{sp\eps}{(p-1)(p-1+\eps)}}\right)^n c^{-\frac{sp}{p-1+\eps}}\mathbf{T}, 
\end{align}
where to lighten the notation we denoted
\[
\mathbf{T}:=\left(\frac{\rho}{R}\right)^{\frac{sp\eps}{(p-1)(p-1+\eps)}}\left(\dashint_{-R^{sp}}^0 \Big[\tail\left(u(t) ; B_R\right) \Big]^{p-1+\eps}\d{t} \right)^{\nicefrac{1}{(p-1+\eps)}}.
\]
To obtain in the last line of~\eqref{eq:modulus-est-2}, we used~\eqref{eq:ind-est-6'} with $j=n-1$, whereas the small quantities governing problem were estimated, owing to $\lambda \leq c$, that
\begin{align*}
&c^{\frac{sp\eps}{(p-1)(p-1+\eps)}} \lambda^{\frac{(n-1)sp}{p-1}}\left(\frac{\rho}{R}\right)^{\frac{sp}{p-1}} \left(c\lambda^{n-1} \frac{\rho}{R} \right)^{-\frac{sp}{p-1+\eps}} \\
& \quad =c^{\frac{sp\eps}{(p-1)(p-1+\eps)}-\frac{sp}{p-1+\eps}}\lambda^{\frac{(n-1)sp\eps}{(p-1)(p-1+\eps)}} \left(\frac{\rho}{R}\right)^{\frac{sp\eps}{(p-1)(p-1+\eps)}} \\
& \quad \leq \left(2c^{\frac{sp\eps}{(p-1)(p-1+\eps)}}\right)^n c^{-\frac{sp}{p-1+\eps}}\left(\frac{\rho}{R}\right)^{\frac{sp\eps}{(p-1)(p-1+\eps)}}.
\end{align*}
Substituting~\eqref{eq:modulus-est-2} back to~\eqref{eq:modulus-est-1} renders that
\begin{align*}
\mathbf{S}_{n-1} &\leq \boldsymbol{\om}_{n-1}\left[\sum_{i=1}^{n}  \left(\overline{\boldsymbol{\gamma}}\,2^{p-1}\,c^{\frac{sp\eps}{p-1+\eps}}\right)^i\right]^{\frac{1}{p-1}} +\overline{\boldsymbol{\gamma}}\left(2c^{\frac{sp\eps}{(p-1)(p-1+\eps)}}\right)^n c^{-\frac{sp}{p-1+\eps}}\mathbf{T}
\end{align*} 
At this stage, we further impose the smallness of $c$ so that
\begin{equation}\label{eq:modulus-est-3}
c \leq \left[\overline{\boldsymbol{\gamma}}\,2^{p-1}(2^{p-1}+1) \right]^{-\frac{p-1+\eps}{sp\eps}}\quad \iff \quad 
\overline{\boldsymbol{\gamma}}\,2^{p-1}\,c^{\frac{sp\eps}{p-1+\eps}} \leq \frac{1}{2^{p-1}+1},
\end{equation}
which implies 
\[
\left[\sum_{i=1}^{n}  \left(\overline{\boldsymbol{\gamma}}\,2^{p-1}\,c^{\frac{sp\eps}{p-1+\eps}}\right)^i\right]^{\frac{1}{p-1}} \leq \frac{1}{2} < 1-\eta
\]
because of $\eta \in (0,\nicefrac{1}{2})$. As a consequence, recalling the definition of $\boldsymbol{\om}_n$ and taking the preceding estimate into account, we arrive at
\[
\boldsymbol{\om}_n \leq (1-\eta)\boldsymbol{\om}_{n-1}+\widehat{\boldsymbol{\gamma}}\mathbf{T}
\]
We iterate this estimate to obtain that
\begin{align*}
\osc_{Q_n} u \leq  \boldsymbol{\om}_n &\leq (1-\eta)^n\boldsymbol{\om} +\sum_{i=0}^{n-1
}(1-\eta)^i\widehat{\boldsymbol{\gamma}}\mathbf{T}\\
&\leq (1-\eta)^n\boldsymbol{\om} +\widetilde{\boldsymbol{\gamma}}\mathbf{T}
\end{align*}
with $\widetilde{\boldsymbol{\gamma}}:=\widehat{\boldsymbol{\gamma}} / \eta$. We then take $r \in (0,\rho)$ arbitrarily. A simple observation shows that there exists $n \in \N$ such that $\rho_{n+1}^{sp} \leq r^{sp} <\rho_n^{sp}$,
thereby implying
\[
\left(\frac{r}{\rho}\right)^{sp} \geq \lambda^{sp(n+1)} \quad \Longrightarrow \quad (1-\eta)^n <2(1-\eta)^{n+1} \leq 2\left(\frac{r}{\rho}\right)^{\beta_0}, 
\]
where
\[
\beta_0:=\frac{\log (1-\eta)}{\log \lambda}.
\]
Here, as mentioned before, $\lambda$ is to be stipulated so that $\lambda \leq \min \left\{c, \left(\frac{\delta}{2}\right)^{\nicefrac{1}{sp}}c\right\}$, once the smallness of $c$ is specified by~\eqref{eq:modulus-est-3}. Altogether, we conclude that
\[
\osc_{Q_r} u \leq 2\boldsymbol{\om}\left(\frac{r}{\rho}\right)^{\beta_0}+\widetilde{\boldsymbol{\gamma}}\left(\frac{\rho}{R}\right)^{\frac{sp\eps}{(p-1)(p-1+\eps)}}\left(\dashint_{-R^{sp}}^0 \Big[\tail\left(u(t) ; B_R\right) \Big]^{p-1+\eps}\d{t} \right)^{\nicefrac{1}{(p-1+\eps)}}.
\]
Now, replacing $\rho$ by $(r\rho)^{\nicefrac{1}{2}}$ yields that
\begin{align*}
\osc_{Q_r} u &\leq 2\boldsymbol{\om}\left(\frac{r}{\rho}\right)^{\frac{\beta_0}{2}}+\widetilde{\boldsymbol{\gamma}}\left(\frac{r}{\rho
}\right)^{\frac{sp\eps}{2(p-1)(p-1+\eps)}}+\mathbf{T},
\end{align*}
where to simplify notation we denoted
\[
\mathbf{T}:=\left(\frac{\rho}{R}\right)^{\frac{sp\eps}{(p-1)(p-1+\eps)}}\left(\dashint_{I_R} \Big[\tail\left(u(t) ; B_R\right) \Big]^{p-1+\eps}\d{t} \right)^{\nicefrac{1}{(p-1+\eps)}}.
\]
Finally, letting $\beta:=\frac{\beta_0}{2} \wedge \frac{sp\eps}{2(p-1)(p-1+\eps)}$ finishes the proof of Theorem~\ref{Thm:Holder modulus}.

\makeatletter
\renewcommand{\thesection}{\Alph{section}.\arabic{subsection}}
\makeatother

\appendix

\makeatletter
\renewcommand{\thefigure}{\Alph{section}.\arabic{figure}}
\@addtoreset{figure}{section}
\makeatother

\section{Technical tools} \label{appendixO}
In this appendix we collect the technical tools that are used in the proofs. 

\subsection*{Embeddings} Here we state the following parabolic version of fractional Sobolev inequality, whose proof is similar to~\cite[Propositions A.2 and A.3]{Lia24a}, with a different exponent $\kappa$, since here we collect the supremum over time of the spatial $L^p$ norm of $u$.
\begin{proposition}\label{FS}
Let $p \geq 1$, $s \in (0,1)$ and set
\[
\kappa_\ast:=\begin{cases}
\frac{d}{d-sp} \quad &\textrm{if}\quad sp<d,\\
2 \quad &\textrm{if}\quad sp\geq d.
\end{cases}
\]
For every function $w \in L^\infty\left(t_1,t_2\,; L^p(B_R)\right) \cap L^p(t_1,t_2\,;W^{s,p}(B_R))$
which is compactly supported in $B_{(1-\mathsf{d})R}$ for some $\mathsf{d} \in (0,1)$ and for a.e. $t \in (t_1,t_2)$, there holds that
\begin{align*}
&\int_{t_1}^{t_2}\int_{B_R}|w|^{\kappa p}\dxt \\
& \quad \leq C\left(R^{sp}\int_{t_1}^{t_2}\iint \nolimits_{B_R \times B_R}\frac{|w(x,t)-w(y,t)|^p}{|x-y|^{d+sp}}\dxyt+\frac{1}{\mathsf{d}^{d+sp}}\int_{t_1}^{t_2}\int_{B_R}|w|^p\dxt\right) \\
&\quad \quad \quad \times \left(\sup_{t_1<t<t_2}\dashint_{B_R}|w(t)|^p\d{x}\right)^{\frac{\kappa_\ast-1}{\kappa_\ast}}
\end{align*}
with $\kappa:=1+\frac{\kappa_\ast-1}{\kappa_\ast}$, where $C$ depends only on $s,p$ and $d$.
\end{proposition}

Finally, we provide a Giagliardo--Nirenberg type estimate, retrieved from~\cite[Lemma 2.2]{BK24}.

\begin{lemma}\label{Lm:GN}
Let $1 \leq m \leq \max\{p,2\}$. Assume that the positive exponents $\widetilde{q}$ and $\widetilde{r}$ satisfy
\[
\frac{1}{\widetilde{q}}+\frac{1}{\widetilde{r}}\left(\frac{sp}{d}+\frac{p}{m}-1\right)=\frac{1}{m},
\]
where they obey
\[
\left\{
    \begin{array}{lll}
    \widetilde{q} \in \left[m, \frac{dp}{d-sp}\right], \quad &\widetilde{r} \in [p,\infty),  & \mbox{if $d>sp$,} \\[3mm]
     \widetilde{q} \in [m, \infty), \quad &\widetilde{r} \in \left(\frac{msp}{d}+p-m,\infty\right),  & \mbox{if $d=sp$,} \\[3mm]
      \widetilde{q} \in [m, \infty], \quad &\widetilde{r} \in \left[\frac{msp}{d}+p-m,\infty\right),  & \mbox{if $d<sp$.}
    \end{array}
    \right.
\]
There exists constant $C(d,s,p,\widetilde{q},\widetilde{r})<\infty$ such that
\begin{align*}
\left(\int_{t_1}^{t_2} \|f(t)\|_{L^{\widetilde{q}}(B_R)}^{\widetilde{r}}\d{t}\right)^{\frac{d\widetilde{q}}{sp\widetilde{q}+d\widetilde{r}}} &\leq C \left(\int_{t_1}^{t_2}[f(t)]_{W^{s,p}(B_R)}^p\d{t} \right.\\
& \left. \quad \quad \quad +R^{-sp} \int_{t_1}^{t_2} \|f(t)\|_{L^p(B_R)}^p\d{t} 
+\sup_{t_1<t<t_2}\|f(t)\|_{L^m(B_R)}^m\right)
\end{align*}
provided that the right-hand side is finite.
\end{lemma}

%
%
%
%
%

\subsection*{Inequalities}

A simple observation shows the following algebraic inequality.
\begin{lemma}\label{Lm:alg-est-2}
Let $\alpha>0$. Then
\[
(a+b)^\alpha \leq 2^{(\alpha-1)_+} \left(a^\alpha+b^\alpha\right)
\]
holds whenever $a,b \in \R_{\geq 0}$.
\end{lemma}

We next list the following inequality for the Gagliardo semi-norm.
\begin{lemma}\label{Lm:alg-est-3}
Let $u \in W^{s,p}(B_R)$. Then for any $0\leq b \leq a$, we have
\[
\left[(u-b)_-\right]_{W^{s,p}(B_R)}^p \leq \left[(u-a)_-\right]_{W^{s,p}(B_R)}^p.
\]
and
\[
\left[(u-a)_+\right]_{W^{s,p}(B_R)}^p \leq \left[(u-b)_+\right]_{W^{s,p}(B_R)}^p.
\]
\end{lemma}
\begin{proof}
The proof is exactly similar to the arguments of the paper~\cite[Lemma 2.3]{BK24}.
\end{proof}

Finally, here we exhibit the following useful lemma.
\begin{lemma}\label{Lm:useful}
For $p \in (1,\infty)$ and $\alpha > -1$ we have the two-sided bound
\[
\frac{1}{2(\alpha+p+1)^{p+2}} \leq \int_0^1 \lambda^\alpha(1-\lambda)^p \d{\lambda} \leq \frac{2}{\alpha+p+1}.
\]
\end{lemma} 
\begin{proof}
The upper bound is clear. Indeed,
\begin{align*}
\int_0^1 \lambda^\alpha(1-\lambda)^p \d{\lambda} &=\int_0^{\nicefrac{1}{2}}\lambda^\alpha(1-\lambda)^p \d{\lambda}+\int_{\nicefrac{1}{2}}^1\lambda^\alpha(1-\lambda)^p \d{\lambda}\\
&\leq \int_0^{\nicefrac{1}{2}}(1-\lambda)^{\alpha+p}\d{\lambda}+\int_{\nicefrac{1}{2}}^1\lambda^{\alpha+p}\d{\lambda} \\
&\leq \frac{2}{\alpha+p+1}.
\end{align*}
For the lower bound, an iterative integration of by parts implies that
\begin{align*}
\int_0^1 \lambda^\alpha(1-\lambda)^p \d{\lambda} &\geq \int_0^1 \lambda^\alpha(1-\lambda)^{\lceil p\rceil} \d{\lambda}\\
&=\frac{\lceil p \rceil}{\alpha+1}\int_0^1\lambda^{\alpha+1}(1-\lambda)^{\lceil p \rceil -1}\d{\lambda} \\
&= \frac{ \lceil p \rceil !}{(\alpha+1)(\alpha+2)\cdots (\alpha+\lceil p \rceil )}\int_0^1\lambda^{\alpha+\lceil p \rceil}\d{\lambda} \\
&=\left(\prod_{j=1}^{\lceil p \rceil}\frac{j}{\alpha+j} \right)\cdot \frac{1}{\alpha+\lceil p \rceil +1}.
\end{align*}
Since
\[
\nicefrac{j}{(\alpha+j)} \geq \nicefrac{1}{(\alpha+p+1)}, \quad \forall j \in \{1,\ldots, \lceil p \rceil\}, \quad \forall \alpha \in (-1,\infty)
\]
and $\lceil p \rceil -1 < p \leq \lceil p \rceil$, we further bound
\begin{align*}
\int_0^1 \lambda^\alpha(1-\lambda)^p \d{\lambda} \geq \frac{1}{(\alpha+p+1)^{\lceil p \rceil}}\cdot \frac{1}{\alpha+p+2} \geq \frac{1}{2(\alpha+p+1)^{p+2}}, 
\end{align*}
as desired.
\end{proof}
\subsection*{Convergence lemma}
For the reader’s convenience, we record the well-known~\emph{fast geometric convergence}. See~\cite[Chapter I.4, Lemma 4.1]{DiB93} for details.
\begin{lemma}\label{Lm:FGC} Let $\{Y_i\}_{i \in \N_0}$ be a sequence of positive numbers, satisfying the recursive inequalities
\begin{equation*}
Y_{i+1} \leq C\,b^iY_i^{1+\beta},
\end{equation*}
where $C, b>1$ and $\beta>0$ are given constants independent of $i \in \N_0$. Then
$Y_i \to 0$ provided that the initial value $Y_0$ satisfies
$
Y_0 \leq C^{-1/\beta}b^{-1/\beta^2}.
$
\end{lemma}

\subsubsection*{\bf Acknowledgments}
The authors appreciate the \emph{Erwin Schr\"{o}dinger Institut f\"{u}r Mathematik und Physik of Vienna} for its kind hospitality during the Workshop ``Degenerate
and Singular PDEs'' held in 24--28, February 2025, where this collaboration began. S.C. acknowledges the partial funding of GNAMPA (INdAM), and the department of Mathematics of Bologna. We also thank Professor Juha Kinnunen for encouraging discussions on the subject of this paper, and Professor Giampiero Palatucci and Naian Liao for helpful suggestions.




{\small
\begin{thebibliography}{SIM-KEN26}
\bibliographystyle{alpha}

\bibitem[ADV25]{ADV25} 
N.~Abatangelo, S~Dipierro, and E.~Valdinoci. 
\newblock A Gentle Invitation to the Fractional World.
\newblock {\em Cham, Springer} 2025.


\bibitem[BDLMBS25]{BDLMBS25}
V. B\"{o}gelein, F. Duzaar, N. Liao, G. Molica Bisci, and R. Servadei. 
\newblock Regularity for the fractional $p$-{L}aplace equation. 
\newblock {\em J. Funct. Anal.}, 289(9), 2025.


\bibitem[BK24]{BK24}
S.-S. Byun, and K.~Kim.
\newblock A {H}\"older estimate with an optimal tail for nonlocal parabolic
  $p$-{L}aplace equations.
\newblock {\em Ann. Mat. Pura Appl. (4)}, 203(1):109--147, 2024.

\bibitem[CCMV25]{CCMV25}
F.~Cassanello, S.~Ciani, B.~Majrashi, and V.~Vespri.
\newblock Local Vs Nonlocal De Giorgi Classes: A brief guide in the homogeneous case.
\newblock {\em Rend. Istit. Mat. Univ. Trieste}, Vol. 57, Art. No. 9, 2025.

\bibitem[CGL25]{CiaGiaLi} S.~Ciani, U.~  Gianazza, and Z.~Li.
\newblock Phragm\'en-Lindel\" of-type theorems for functions in Homogeneous De Giorgi Classes. 
\newblock Preprint arXiv:2505.17926.



\bibitem[Coz17]{Coz17}
M.~Cozzi.
\newblock Regularity results and {H}arnack inequalities for minimizers and
  solutions of nonlocal problems: a unified approach via fractional {D}e
  {G}iorgi classes.
\newblock {\em J. Funct. Anal.}, 272(11):4762--4837, 2017.


\bibitem[DCKP14]{DKP14}
A.~Di~Castro, T.~Kuusi, and G.~Palatucci.
\newblock Nonlocal {H}arnack inequalities.
\newblock {\em J. Funct. Anal.}, 267(6):1807--1836, 2014.

\bibitem[DCKP16]{DKP16}
A.~Di~Castro, T.~Kuusi, and G.~Palatucci.
\newblock Local behavior of fractional {$p$}-minimizers.
\newblock {\em Ann. Inst. H. Poincar\'{e} C Anal. Non Lin\'{e}aire},
  33(5):1279--1299, 2016.

\bibitem[DeG57]{DG57} E.~De Giorgi. 
\newblock Sulla Differenziabilit\`a e l'Analiticit\`a degli Integrali Multipli Regolari.
\newblock {\em Mem. Accad. Sci. Torino Cl. Sci. Fis. Mat. Natur.}, 3(3), 25--43, 1957. 

\bibitem[DG16]{DG}
E.~DiBenedetto, and Gianazza.
\newblock Some properties of De Giorgi classes.
\newblock {\em {Rend. Istit. Mat. Univ.}}, 48: 245--260, 2016. 

\bibitem[DG23]{DG23}
E.~DiBenedetto, and U. Gianazza. 
\newblock {\em "Partial diﬀerential equations,” Cornerstones.}
\newblock Birkh\"{a}user-Springer, Cham, 2023.


\bibitem[DGV16]{DGV}
E. DiBenedetto, U. Gianazza and V. Vespri.
\newblock Local clustering of the non-zero set of functions in $W^{1,1}(E)$. 
\newblock {\em Atti Accad. Naz. Lincei Rend. Lincei Mat. Appl.}, 17(3): 223--225, 2016.

\bibitem[DiB93]{DiB93}
E.~DiBenedetto.
\newblock {\em Degenerate parabolic equations}.
\newblock Universitext. Springer-Verlag, New York, 1993.


\bibitem[DNPV12]{DNPV12}
E.~Di Nezza, G.~Palatucci, and E.~Valdinoci. 
\newblock Hitchhiker's guide to the fractional Sobolev space.
\newblock{\em Bull. Sci. Math.}, 136(5): 521--573, 2012.


\bibitem[DT84]{DT}
E. DiBenedetto, and N.S. Trudinger. 
\newblock Harnack inequalities for quasiminima of variational integrals. 
\newblock {\em Ann. Inst. H. Poincar\'{e} Anal. Non Lin\'{e}aire}, 1(4): 295--308, 1984.




\bibitem[DV95]{DV}
E.~DiBenedetto, and V.~Vespri. 
\newblock On the singular equation $\beta(t)_t=\Delta u$.
\newblock {\em Arch. Rational Mech. Anal.}, 132(3): 247--309, 1995.


\bibitem[FeRo24]{FeRo24}
X.~Fern\'{a}ndez-Real, and X.~Ros-Oton. 
\newblock {\em Integro-Differential Elliptic Equations},
\newblock Progress in Mathematics 350, Birkh\"{a}user, 2024.






\bibitem[Gla69]{Gla69}
R.Y.~Glagoleva. 
\newblock Liouville theorems for the solution of a second-order linear parabolic equation with discontinuous coefficients.
\newblock {\em Mat. Zametki}, 5: 599–-606, 1969.





\bibitem[GV06]{GV}
U.~Gianazza, and V.~Vespri. 
\newblock Parabolic De Giorgi classes of order $p$ and the Harnack inequality.
\newblock {\em Calc. Var. Partial Differential Equations}, 26(3): 379--399, 2006.

\bibitem[GG84]{GiaGiu} M.~Giaquinta, E. ~Giusti.
\newblock Quasi-minima.
\newblock {\em Annales de l'Institut Henri Poincaré C, Analyse non linéaire, 1(2): 79--107.}

\bibitem[Giu03]{Giusti} E.~Giusti.
\newblock Direct methods in the calculus of variations.
\newblock {\em World Scientific}, Singapore, 2003.







\bibitem[KMMP12]{KMMP}
J.~Kinnunen, N.~Marola, M.~Miranda, and F.~Paronetto.
\newblock Harnack's inequality for parabolic De Giorgi classes in metric spaces.
\newblock {\em Adv. Differential Equations}, 17(9-10) : 801--832, 2012.




\bibitem[KP18]{KP18}
T.~Kuusi, and G.~Palatucci (Eds.).
\newblock {\em Recent Developments in Nonlocal Theory.}
\newblock De Gruyter, Berlin/Boston, 2018.


\bibitem[KS01]{KS}
J. ~Kinnunen, and N. Shanmugalingam.
\newblock Regularity of quasi-minimizers on metric spaces.
\newblock {\em Manuscripta Math.}, 105(3):401--423, 2001.


  
  
\bibitem[KW24]{KW23a}
M.~Kassmann, and M.~Weidner.
\newblock The parabolic Harnack inequality for nonlocal equations.
\newblock {\em Duke Math. J.}, 173(17):3413--3451, 2024.

\bibitem[Lia21]{Lia21}
N.~Liao.
\newblock Remarks on parabolic De Giorgi classes.
\newblock {\em Ann. Mat. Pura Appl. (4)}, 200(6):2361--2384, 2021.

\bibitem[Lia24a]{Lia24a}
N.~Liao.
\newblock H\"{o}lder regularity for parabolic fractional {$p$}-{L}aplacian.
\newblock {\em Calc. Var. Partial Differential Equations}, 63(1):Paper No. 22, 2024.



\bibitem[Lia24b]{Lia24b}
N.~Liao.
\newblock On the modulus of continuity of solutions to nonlocal parabolic
  equations.
\newblock {\em J. Lond. Math. Soc.}, 110(3):Paper No. e12985, 30 pp., 2024.

\bibitem[Lia25]{Lia25}
N.~Liao.
\newblock Harnack estimates for nonlocal drift diffusion equations.
\newblock {\em arXiv:2402.11986v2}, 2025.

\bibitem[LU68]{LU} O. A. Ladyzhenskaya, and N. N. Ural'tseva.
\newblock Linear and Quasilinear Elliptic Equations. 
\newblock{\em {Academic Press}}, New York, 1968. 


\bibitem[LSU68]{Lady}
O.A. Ladyzhenskaja, V.A. Solonnikov, and N.N. Ural’tzeva. 
\newblock Linear and Quasilinear Equations
of Parabolic Type.
\newblock{\em {AMS Transl. Math. Mono}}, vol. \textbf{23}: Providence RI, USA, 1968.




\bibitem[LW25]{LW25}
N.~Liao, and M.~Weidner,
\newblock Time‐insensitive nonlocal parabolic Harnack estimates. 
\newblock {\em Proc. Lond. Math. Soc.}, 130(5): Paper No. e70051, 53 pp., 2025.

\bibitem[Min11]{Min11}
G.~Mingione.
\newblock Gradient potential estimates.
\newblock {\em Journal of the European Mathematical Society}, 13: 459-486, 2011.




\bibitem[Nak23]{Nak23}
K.~Nakamura.
\newblock Local properties of fractional parabolic De Giorgi classes of order $s,p$.
\newblock {\em J. Funct. Anal.}, 285:110049, 2023.



\bibitem[Str19]{Strom}
M.~Str\"omqvist.
\newblock Harnack's inequality for parabolic nonlocal equations.
\newblock {\em Annales de l'Institut Henri Poincaré. C, Analyse non linéaire}, 36(6), 1709-1745, 2019.




\bibitem[Wan88]{Wang}
G.~Wang. 
\newblock {Harnack inequalities for functions in De Giorgi parabolic class}. 
\newblock Partial differential equations (Tianjin, 1986), 182--201, Lecture Notes in Math., 1306, Springer, Berlin, 1988. 

\bibitem[Wie87]{Wieser}
W.~Wieser.
\newblock {Parabolic Q-minima and minimal solutions to variational flow}. 
\newblock {\em Manuscripta Mathematica}, 59(1): 63--107, 1987.

\end{thebibliography}
}
\end{document}